\documentclass[11pt]{amsart}
\usepackage[english]{babel}
\usepackage[margin=1in]{geometry}
\usepackage{setspace}
\usepackage{subcaption} 
\usepackage{graphicx}
\usepackage[colorlinks=true, allcolors=blue]{hyperref}
\usepackage{amsmath,amssymb,graphicx}
\usepackage{dsfont}
\usepackage{hyperref}
\usepackage{bm}
\usepackage{float}
\usepackage{tabularx}
\usepackage{mathrsfs}
\usepackage{geometry}
\usepackage{enumerate}
\usepackage{enumitem}
\usepackage{tikz}
\usetikzlibrary{arrows.meta,decorations.pathreplacing}
\usepackage{subcaption}
\usepackage{graphicx}
\usepackage{cite}
\usetikzlibrary{shapes.misc}
\tikzset{crosspt/.style={cross out,draw,thick,minimum size=5pt,inner sep=0pt}}

\setlist[enumerate]{label={\upshape(\roman*)}}
\allowdisplaybreaks

\theoremstyle{plain}
\newtheorem{assumption}{Assumption}
\newtheorem{remark}[]{Remark}

\newtheorem{theorem}{Theorem}
\newtheorem{lemma}{Lemma}

\newtheorem{proposition}{Proposition}

\newcommand{\rd}{\mathrm{d}}
\newcommand{\R}{\mathbb{R}}
\newcommand{\X}{\mathbb{X}}
\newcommand{\F}{\mathbb{F}}
\newcommand{\N}{\mathbb{N}}

\newcommand{\E}{\mathbb{E}}
\newcommand{\tl}{\tilde}
\newcommand{\wtl}{\widetilde}
\newcommand{\dl}{\delta}
\newcommand{\sg}{\sigma}
\newcommand{\gm}{\gamma}
\newcommand{\Th}{\Theta}
\newcommand{\ep}{\epsilon}
\newcommand{\lf}{\left}
\newcommand{\rt}{\right}

\title[]{Strong convergence and temporal-spatial regularity for tamed Euler approximations of Lévy-driven SDEs}

\author[Y. Ding]{Yan Ding}
\author[S. Wu]{Sizhou Wu}
\author[Y. Zhang]{Ying Zhang}

\address{Financial Technology Thrust, Society Hub, The Hong Kong University of Science and Technology (Guangzhou), %No.\ 1 Du Xue Rd, Nansha District,
Guangzhou, China}
\email{yding632@connect.hkust-gz.edu.cn}

\address{School of Mathematics, Shanghai University of Finance and Economics, Shanghai, China}
\email{wusizhou@sufe.edu.cn}

\address{Financial Technology Thrust, Society Hub, The Hong Kong University of Science and Technology (Guangzhou), %No.\ 1 Du Xue Rd, Nansha District,
Guangzhou, China}
\email{yingzhang@hkust-gz.edu.cn}

\thanks{
Financial support by the Guangzhou-HKUST(GZ) Joint Funding Program (No. 2025A03J3322) and Fundamental Research Funds for the Central Universities (China) is gratefully acknowledged.}


\begin{document}
\maketitle
\vspace{-5mm}
\begin{abstract}
We study the temporal-spatial regularity properties of tamed Euler approximations for L\'evy-driven SDEs with superlinearly growing drift and diffusion coefficients. We first introduce a novel tamed Euler-type scheme and establish its strong convergence. We then derive temporal-spatial regularity estimates with respect to the initial value, the initial time, and the evaluation time. In particular, we obtain a stability estimate for fixed step size and a corresponding continuity estimate in the vanishing step-size regime. Numerical experiments are presented to support the theoretical results.
\end{abstract}

\section{Introduction}
Stochastic differential equations (SDEs) driven by L\'evy noise constitute an important class of stochastic models in mathematical finance, engineering, and many other areas where problems are influenced by event-driven uncertainties \cite{MR2512800,barndorff2012levy,MR2042661,MR2160585}. In many applications, one needs to approximate the solution for several initial values and initial times. However, straightforward interpolation over a fine grid may be inefficient in such situations. A more efficient alternative is to use multilevel or multigrid approximation. Moreover, multilevel Picard approximations have proved effective for approximating solutions to nonlinear and high-dimensional partial differential equations, see, e.g., \cite{MR4462404,MR4338044,hutzenthaler2020multilevel1,MR4557620,disintegration,MR4075337}. To justify and analyze such methods, it is essential to understand how numerical approximations depend on the initial value, the initial time, and the evaluation time. This naturally leads to the study of temporal-spatial regularity properties of Euler-type approximations.

Motivated by this perspective, \cite{MR4423846} established a strong convergence rate of $1/2$ for Euler-Maruyama approximations of SDEs driven by standard Brownian motion in the temporal-spatial H\"older norm under suitable assumptions on the coefficients. Their analysis shows that the Euler scheme is continuous with respect to perturbations of the initial time, initial value, and evaluation time. However, as the framework in \cite{MR4423846} relies on global Lipschitz assumptions and is restricted to Brownian-motion-driven SDEs, it does not apply to SDEs with superlinearly growing coefficients, nor does it extend to SDEs driven by Lévy noise. Hence, it excludes important applications including, e.g., the $3/2$ volatility model with jumps.

In this paper, we consider L\'evy-driven SDEs whose drift and diffusion coefficients satisfy a one-sided Lipschitz condition and a polynomial Lipschitz condition, while the jump coefficient satisfies an integrability condition and a global Lipschitz condition. Since the classical explicit Euler scheme is not applicable \cite{divergence}, we propose a novel tamed Euler-type scheme based on the taming technique introduced in \cite{2016eulerdiffusion} and the idea of fully implementable Euler schemes presented in \cite{L.GononC.Schwab}. Under appropriate assumptions, we first establish strong $L^p$-convergence of the proposed scheme in Theorem \ref{theorem1}\ref{theorem(i)}. We then derive temporal-spatial regularity estimates showing that: 
\begin{enumerate}
\item For a fixed step size, the proposed scheme is stable with respect to perturbations of the initial value, the initial time, and the evaluation time, as given in Theorem~\ref{theorem1}\ref{theoremnew}.
\item The proposed scheme is continuous with respect to these variables as the step size tends to zero, see Theorem \ref{theorem1}\ref{theorem(ii)}.
\end{enumerate}
To the best of our knowledge, these are the first such results for fully implementable tamed Euler schemes in the literature. Numerical experiments are provided which support our main findings.

The paper is organized as follows. In Section~\ref{Setting and Tamed Euler Scheme}, we introduce the L\'evy-driven SDE under consideration and define the proposed tamed Euler-type scheme. Section~\ref{Assumptions and Main Results} states the assumptions and presents the main results of this paper. In Section~\ref{Numerical Simulations}, we provide numerical experiments for two examples supporting the theoretical findings. Finally, Section~\ref{Proof of the Main Results} contains the proofs of the main results.

\subsection{Notation} We conclude this section by introducing some basic notations that will be used throughout the paper. 
Let $m,d \in \N$. For a vector $\mu \in \R^{d}$, we denote by $\|\mu\|$ its Euclidean norm, and by $\mu^{*}$ its transpose.  For a matrix $\sg \in \R^{d\times m}$, we denote by $\|\sg\|$ its Hilbert-Schmidt norm, and by $\sg^{*}$ its transpose. 
For a vector $\gm \in \R^{d}$, its $i$-th component is denoted by $\gm^{(i)}$, for $i=1,\dots,d$. 
For $x,y \in \R^{d}$, we write $xy$ for their Euclidean inner product. Moreover, \(C(\R^{d},\R^{d})\) denotes the space of continuous functions from \(\R^{d}\) to \(\R^{d}\). Finally, for \(p\ge 1\) and an \(\R^d\)-valued random variable \(X\), we denote by \(\|X\|_{L^p}:=(\E[\|X\|^p])^{1/p}\) its \(L^p\)-norm.

\section{Setting and Tamed Euler Scheme}\label{Setting and Tamed Euler Scheme}

\subsection{Setting}
Fix $T>0$. Let $(\Omega, \mathcal{F}, P)$ be a complete probability space equipped with a filtration $\F := \lf(\mathcal{F}_{t}\rt)_{t \in [0, T]}$ satisfying the usual conditions\footnote{$\mathcal{F}_{0}$ contains all $P$-null sets and the filtration is right continuous.}. 
Let $W:=\lf(W_{t}\rt)_{t \in [0, T]}$ be an $\R^{m}$-valued standard $\F$-Brownian motion and let $\nu(\rd z)$ be a L\'evy measure on $\R^{d}$, i.e., $\nu(\{0\})=0$ and $ \int_{\R^d} (1\wedge \|z\|^2)\,\nu(\rd z)<\infty$.
Moreover, let $\pi(\rd z, \rd t)$ be an stationary $\F$-Poisson random measure on $\lf(\R^{d} \times [0, T], \mathcal{B}(\R^{d}) \otimes \mathcal{B}([0, T])\rt)$ with intensity measure $\nu(\rd z) \otimes \rd t$, and denote by
$
\tl{\pi}(\rd z, \rd t) := \pi(\rd z, \rd t) - \nu(\rd z)\rd t
$
the compensated Poisson random measure associated with $\pi$.

For each $s \in [0,T]$ and $x \in \R^{d}$, we consider the following SDE on $[s,T]$:
\begin{equation}\label{SDE}
	\rd X_{t}^{s, x}
	= \mu\lf(X_{t-}^{s, x}\rt)\rd t
	+ \sg\lf(X_{t-}^{s, x}\rt)\rd W_{t}
	+ \int_{\R^d} \gm\lf(X_{t-}^{s, x}, z\rt)\, \tl{\pi}(\rd z,\rd t),
\end{equation}
with the initial condition $X_{s}^{s, x} = x$, where $\mu \in C(\R^{d}, \R^{d})$, $\sg \in C(\R^{d}, \R^{d \times m})$, and $\gm \in C(\R^{d} \times \R^{d}, \R^{d})$. 

\subsection{Tamed Euler Scheme}
We first introduce a time-discretized approximation of \eqref{SDE} based on an Euler-Maruyama-type scheme. 
To this end, let $\Th$ denote the collection of all discretization maps $\dl : [0,T] \to [0,T]$ such that
\[
\Th := \biggl\{ \dl : [0,T]\to[0,T] :
\begin{aligned}
	&\exists \; n\in\N,\ 0 = t_0 < t_1 < \dots < t_n = T, \\
	&\;\text{s.t.}\; \dl([t_0,t_1]) = \{t_0\},\ \dl((t_{i-1},t_i]) = \{t_{i-1}\},\ i=1,\dots,n,\\
	&\text{and } |t_i - t_{i-1}| < 1,\ \forall i=1,\dots,n
\end{aligned}
\biggr\}.
\]
Let $\iota : [0,T] \to [0,T]$ be the identity mapping defined by $\iota(t) = t$, and set $\wtl{\Th} := \Th \cup \{\iota\}$. 
Define $|\cdot| : \wtl{\Th} \to [0,T]$ by $|\iota| = 0$ and, for $\dl \in \Th$,
\[
|\dl| := \max \lf\{ |s-t| : s,t \in \dl([0,T]),\ s<t,\ (s,t)\cap \dl([0,T]) = \emptyset \rt\}.
\]

For $s \in [0,T]$, $x \in \R^d$, and identity mapping $\iota$, let $\lf(X^{s,x,\iota}_t\rt)_{t\in[s,T]}$ be an $\F$-adapted stochastic process satisfying that $X^{s,x,\iota}_s = x$ and
\begin{equation}\label{exactsolu}
	\rd X^{s,x,\iota}_t
	= \mu\lf(X^{s,x,\iota}_{\max\{s,t-\}}\rt)\rd t
	+ \sg\lf(X^{s,x,\iota}_{\max\{s,t-\}}\rt)\rd W_t
	+ \int_{\R^{d}} \gm\lf(X_{\max\{s,t-\}}^{s,x,\iota}, z\rt)\, \tl{\pi}(\rd z,\rd t),
\end{equation}
which can be viewed as an alternative form of SDE \eqref{SDE}.

We now consider the case where coefficients $\mu$ and $\sg$ may grow superlinearly. Since the classical Euler scheme is known to fail to approximate \eqref{exactsolu} (see, e.g., \cite{divergence}), we introduce a novel tamed Euler-type scheme based on \cite{L.GononC.Schwab} and \cite{2013anoteoneuler}: for \(s\in[0,T]\), \(x\in\R^d\), \(\dl\in\Th\),  \(\ep\in(0,1)\), and \(\mathcal{M}\in\N\), let \(\bigl(X^{s,x,\dl,\ep,\mathcal{M}}_t\bigr)_{t\in[s,T]}\) be an $\F$-adapted stochastic process satisfying that \(X^{s,x,\dl,\ep,\mathcal{M}}_s=x\) and
\begin{equation}\label{Tamedscheme}
	\begin{aligned}
		\rd X^{s,x,\dl,\ep,\mathcal{M}}_t
		&= \mu^{\dl}\bigl(X^{s,x,\dl,\ep,\mathcal{M}}_{\max\{s,\dl(t-)\}}\bigr)\,\rd t
		+ \sg^{\dl}\bigl(X^{s,x,\dl,\ep,\mathcal{M}}_{\max\{s,\dl(t-)\}}\bigr)\,\rd W_t + \int_{A_\ep} \gm\bigl(X_{\max\{s,\dl(t-)\}}^{s,x,\dl,\ep,\mathcal{M}}, z\bigr)\, \pi(\rd z,\rd t) \\
		&\quad - \Bigl( \frac{\nu(A_\ep)}{\mathcal{M}} \sum_{i=1}^{\mathcal{M}}
		\gm\bigl(X_{\max\{s,\dl(t-)\}}^{s,x,\dl,\ep,\mathcal{M}}, V^{\dl,\ep,\mathcal{M}}_{i,\max\{s,\dl(t-)\}} \bigr) \Bigr)\,\rd t,
	\end{aligned}
\end{equation}
where, for $\chi \geq 1$, the tamed coefficients are given by
\begin{equation}\label{tameddef}
	\mu^{\dl}(x)
	:= \frac{\mu(x)}{\bigl(1+|\dl|^{\frac{1}{2}}\|x\|^{\chi}\bigr)^{\frac{1}{2}}},
	\qquad
	\sg^{\dl}(x)
	:= \frac{\sg(x)}{\bigl(1+|\dl|^{\frac{1}{2}}\|x\|^{\chi}\bigr)^{\frac{1}{2}}},
\end{equation}
for $\ep\in(0,1)$, the set $A_\ep$ is defined by
\begin{equation}\label{Aep}
A_\ep:=\{z\in\R^d:\|z\|\ge \ep\},
\end{equation}
and for $\mathcal{M} \in \N$,
$\lf(V^{\dl,\ep,\mathcal{M}}_{i,t_j}\rt)_{i=1,...,\mathcal{M}, j =0,...,n}$ are i.i.d.\ \(\R^d\)-valued random variables with distribution
\begin{equation}\label{nuep}
\nu_\ep(B):=\frac{\nu(B\cap A_\ep)}{\nu(A_\ep)},\qquad B\in\mathcal{B}(\R^d),
\end{equation}
independent of \(W\) and \(\pi\).

\begin{remark}
	The taming factor given in \eqref{tameddef} is specifically designed for the Lévy-driven SDEs considered in this paper. In addition to controlling the superlinear growth of the coefficients and preserving consistency to the original coefficients as $|\dl|\to0$, it is chosen so that the tamed coefficients satisfy the estimate
	\begin{equation}\label{strongerbound}
		\max_{\xi\in\{\mu,\sg\}}\|\xi^\dl(x)\|^2\le K|\dl|^{-1/2}(1+\|x\|^2),
	\end{equation}
	as shown in Remark~\ref{remark2} below, which  is fundamental to our moment and convergence analysis.
	
	By comparison, the taming factors proposed in \cite{MR2985171,2016eulerdiffusion,MR3364862} typically yield only
	\begin{equation}\label{weakbound}
		\max_{\xi\in\{\mu,\sg\}}\|\xi^\dl(x)\|\le K|\dl|^{-1/2}(1+\|x\|),
	\end{equation}
	which is often sufficient for SDEs driven by Brownian motion, but appears to be insufficient for establishing the moment bounds and convergence estimates in the L\'evy-driven setting studied here, see, e.g., \eqref{L2B_1} in Lemma \ref{L3}.
\end{remark}

\begin{remark}
		The proposed scheme \eqref{Tamedscheme} is obtained by combining two approximation steps for Lévy-driven SDEs with a taming procedure addressing superlinearly growing coefficients.
		
\begin{enumerate}
	 \item\label{remark(i)}\textbf{Taming.} We adopt the taming technique and propose a novel taming factor to obtain the following tamed Euler approximation to SDE \eqref{exactsolu}: for $s\in[0,T]$, $x\in\R^d$, and $\dl \in \Th$, let $\bigl(X^{s,x,\dl}_t\bigr)_{t\in[s,T]}$ be an $\F$-adapted stochastic process satisfying that $X^{s,x,\dl}_s = x$ and
		\begin{equation}\label{tamedscheme1}
			\begin{aligned}
				\rd X^{s,x,\dl}_t
				&= \mu^{\dl}\lf(X^{s,x,\dl}_{\max\{s,\dl(t-)\}}\rt)\rd t
				+ \sg^{\dl}\lf(X^{s,x,\dl}_{\max\{s,\dl(t-)\}}\rt)\rd W_t + \int_{\R^d} \gm\lf(X_{\max\{s,\dl(t-)\}}^{s,x,\dl}, z\rt)\, \tl{\pi}(\rd z,\rd t),
			\end{aligned}
		\end{equation}
		where $	\mu^{\dl}(x)$ and $\sg^{\dl}(x)$ are given in \eqref{tameddef}.
		
	\item\textbf{ Truncation of small jumps.}
		The compensated jump integral in \eqref{tamedscheme1} is not directly simulatable. Following \cite{L.GononC.Schwab}, we introduce the $\ep$-truncated tamed Euler scheme: for $s\in [0,T]$, $x\in \R^d$, $\dl \in \Th$, and $\ep \in (0,1)$, let $\lf(X^{s,x,\dl,\ep}_t\rt)_{t\in [s,T]}$ be an $\F$-adapted stochastic process satisfying that $X^{s,x,\dl,\ep}_s = x$ and
		\begin{equation}\label{tamedscheme2}
			\begin{aligned}
				\rd X^{s,x,\dl,\ep}_t
				&= \mu^{\dl}\lf(X^{s,x,\dl,\ep}_{\max\{s,\dl(t-)\}}\rt)\rd t
				+ \sg^{\dl}\lf(X^{s,x,\dl,\ep}_{\max\{s,\dl(t-)\}}\rt)\rd W_t \\
				&\quad + \int_{A_\ep} \gm\lf(X_{\max\{s,\dl(t-)\}}^{s,x,\dl,\ep}, z\rt)\, \pi(\rd z,\rd t)
				- \int_{A_\ep} \gm\lf(X_{\max\{s,\dl(t-)\}}^{s,x,\dl,\ep}, z\rt)\, \nu(\rd z)\rd t,
			\end{aligned}
		\end{equation}
		where the set $A_\ep$ is defined in \eqref{Aep} such that $\nu(A_\ep) < \infty$.
		
	\item\label{remark(iii)}\textbf{ Monte Carlo approximation of the compensator.}
		To obtain a fully implementable discretization, we replace the compensator integral by the Monte Carlo estimator
		\[
		\int_{A_\ep}\gm(x,z)\,\nu(\rd z)
		=\nu(A_\ep)\int_{\R^d}\gm(x,z)\,\nu_\ep(\rd z)
		\approx \frac{\nu(A_\ep)}{\mathcal{M}}\sum_{i=1}^{\mathcal{M}}\gm(x,V_i),
		\]
		where $\lf(V_i\rt)_{i=1,...,\mathcal{M}}$ are i.i.d. $\R^d$-valued random variables with distribution $\nu_\ep$ defined in \eqref{nuep}, independent of $W$ and $\pi$.
\end{enumerate}
		Finally, combining \ref{remark(i)}--\ref{remark(iii)} yields the proposed scheme \eqref{Tamedscheme}. 
\end{remark}

\section{Assumptions and Main Results}\label{Assumptions and Main Results}

We fix $p_0 \geq 2$ and make the following assumptions on the coefficients $\mu$, $\sg$, and $\gm$ of SDE \eqref{SDE}. Proofs for the remarks are postponed to Appendix \ref{app:Assumptions and Main Results}.

\begin{assumption}\label{A1}
	There exists a constant $L > 0$ such that, for all $x,y \in \R^d$,
	\[
	2(x-y)\lf(\mu(x)-\mu(y)\rt) + (p_0 -1)\|\sg(x)-\sg(y)\|^2 \leq L\|x-y\|^2.
	\]
\end{assumption}

\begin{assumption}\label{A2}
	There exist constants $L>0$ and $ \chi \geq 1$ such that, for all $x,y \in \R^d$,
	\[
	\max_{\xi \in\{\mu,\sg\}}\|\xi(x)-\xi(y)\|
	\leq L\|x-y\|\lf(1+\|x\|^{\frac{\chi}{2}}+\|y\|^{\frac{\chi}{2}}\rt).
	\]
\end{assumption}

\begin{remark}
	Throughout this paper, the constant $K>0$ may take different values at different places, it depends on $L,\chi,p_0, T$, but is always independent of $\dl \in \wtl{\Th}$ and $d \in \N$.
\end{remark}

\begin{remark}\label{remark1}
	Under Assumption \ref{A2}, there exists a constant $K>0$ such that, for all $x \in \R^d$,
	\[
	\max_{\xi \in\{\mu,\sg\}}\|\xi(x)\| \leq K\lf(1+\|x\|^{\frac{\chi}{2}+1}\rt).
	\]
\end{remark}

\begin{remark}\label{exist}
	By Assumption \ref{A1} and Remark \ref{remark1}, there exists a constant $K>0$ such that, for all $x \in \R^d$,
	\[
	2x\mu(x) + (p_0-1)\|\sg(x)\|^2 \leq K(1 + \|x\|^2).
	\]
\end{remark}

\begin{remark}\label{remark2}
	By Remark \ref{remark1} and the definition of the tamed coefficients in \eqref{tameddef}, one observes that, for all $x \in \R^d$ and $\dl \in \Th$,
	\[
	\max_{\xi\in\{\mu,\sg\}}\|\xi^{\dl}(x)\|^2
	\leq \min\lf( K|\dl|^{-\frac{1}{2}}(1+\|x\|^2), \|\xi(x)\|^2 \rt).
	\]
\end{remark}

\begin{remark}\label{remark3}
	Combining Remarks \ref{remark1} and \ref{remark2}, one further observes that, for all $x \in \R^d$ and $\dl \in \Th$,
	\begin{align*}
		\max_{\xi \in \{\mu,\sg\}}\|\xi(x)-\xi^\dl(x)\|
		\leq K|\dl|^{\frac{1}{2}}\lf(1+\|x\|^{\frac{3}{2}\chi+1}\rt).
	\end{align*}
\end{remark}

\begin{remark}\label{remark4}
By Assumption \ref{A2} and Remark \ref{remark1}, there exists a constant $K>0$ such that, for all $x,y \in \R^d$ and $\dl \in \Th$,
\begin{align*}
	\max_{\xi \in \{\mu,\sg\}}\|\xi^\dl(x)-\xi^\dl(y)\|
	&\leq  K|\dl|^{-\frac14}\|x-y\|.
\end{align*}
\end{remark}

\begin{assumption}\label{A3}
	There exists a constant $N_d>0$ such that, for all $x \in \R^d$ and $\rho \in [1,p_0]$,
	\[
	\int_{\R^d}\|\gm(x,z)\|^{\rho}\nu(\rd z) \leq N_d\lf(1+\|x\|^{\rho}\rt).
	\]
\end{assumption}

\begin{assumption}\label{A4}
	There exists a constant $L_d>0$ such that, for all $x,y \in \R^d$ and $\rho \in [1,p_0]$,
	\[
	\int_{\R^d}\|\gm(x,z)-\gm(y,z)\|^{\rho}\nu(\rd z)
	\leq L_d\|x-y\|^{\rho}.
	\]
\end{assumption}

\begin{remark}
	Note that Assumptions \ref{A1}--\ref{A4} guarantee the existence and uniqueness of a solution to SDE \eqref{exactsolu} (see e.g., \cite[Theorem 2]{I.Gyongy1}).
\end{remark}

Moreover, for any $p \in [2, p^*]$, where
\begin{equation}\label{p*}
p^* := \min\left\{ \frac{2p_0}{\chi+2}-\zeta,\; \frac{2p_0}{3\chi+2},\;
\frac{-\chi\zeta+\sqrt{\chi\zeta(\chi\zeta+8 p_0)}}{2\chi} \right\}
\end{equation}
with some arbitrarily small $\zeta > 0$, we impose the following condition:

\begin{assumption}\label{A5}
	There exist constants $\mathcal{Z}_d,q > 0$ such that for all $\ep \in (0,1)$ and $x \in \R^d$,
	\begin{align}
		\int_{\R^d}(1 \wedge \|z\|^2)\nu(\rd z)
		&\leq \mathcal{Z}_d, \label{A51}  \\
		\int_{\|z\|\leq \ep}\|\gm(x,z)-\gm(y,z)\|^{p}\nu(\rd z)
		&\leq \ep^{q}\mathcal{Z}_d\|x-y\|^{p}, \label{A52} \\
		\int_{\|z\|\leq \ep}\|\gm(x,z)\|^{p}\nu(\rd z)
		&\leq \ep^{q}\mathcal{Z}_d\lf(1+\|x\|^{p}\rt). \label{A53}
	\end{align}
\end{assumption}

\begin{remark}\label{newscheme}
	By Assumption \ref{A5} and the definition of $A_{\ep}$ in \eqref{Aep}, we have, for all $\ep \in (0,1)$, that
	\begin{equation}\label{nudef}
		\nu(A_\ep) = \int_{A_\ep} \frac{1 \wedge \|z\|^2}{1 \wedge\|z\|^2} \nu(\rd z) \leq \ep^{-2}\int_{\R^d} (1 \wedge \|z\|^2)\nu(\rd z) \leq  \ep^{-2}\mathcal{Z}_d,
	\end{equation}
	which implies that $\nu$ is finite on $(A_\ep,\mathcal{B}(A_\ep))$. Thus, for all $t\in[s,T]$, we can rewrite scheme \eqref{Tamedscheme} as
	\begin{equation}\label{Tamedscheme2}
		\begin{aligned}
			X^{s,x,\dl,\ep,\mathcal{M}}_t &= x+ \int_s^t\mu^\dl\lf(X^{s,x,\dl,\ep,\mathcal{M}}_{\max\{s,\dl(r-)\}}\rt)\rd r+\int_s^t\sg^\dl\lf(X^{s,x,\dl,\ep,\mathcal{M}}_{\max\{s,\dl(r-)\}}\rt)\rd W_r\\
			&\quad+\int_s^t\int_{A_\ep} \gm\lf(X_{\max\{s,\dl(r-)\}}^{s,x,\dl,\ep,\mathcal{M}}, z\rt) \tl{\pi}(\rd z,\rd r)+\int_s^t\int_{A_\ep} \gm\lf(X_{\max\{s,\dl(r-)\}}^{s,x,\dl,\ep,\mathcal{M}}, z\rt) \nu(\rd z)\rd r\\
			&\quad- \int_s^t \lf(\frac{\nu(A_\ep)}{\mathcal{M}}\sum_{i=1}^{\mathcal{M}}\gm\lf(X_{\max\{s,\dl(r-)\}}^{s,x,\dl,\ep,\mathcal{M}},V^{\dl,\ep,\mathcal{M}}_{i,\max\{s,\dl(r-)\}}\rt)\rt)\rd r.
		\end{aligned} 
	\end{equation}
\end{remark}

Under these assumptions, we establish a strong convergence result for scheme \eqref{Tamedscheme}. Moreover, we show that scheme \eqref{Tamedscheme} is continuous w.r.t. the initial value, the initial time, and the evaluation time. The proofs for the main results can be found in Section \ref{Proof of the Main Results}.

\begin{theorem}\label{theorem1}
	Let Assumptions \ref{A1}--\ref{A5} hold and define $\lf(\X^{s,x,\dl}_t\rt)_{\dl \in \wtl\Th}$ such that for all $ s\in [0,T]$, $x\in \R^d$, $\dl \in \wtl{\Th}$, $\ep \in (0,1)$, and $\mathcal{M} \in \N$ with $\mathcal{M}\geq (1 \vee \ep^{-2}\mathcal{Z}_d)^{2}$,
	\[
	\X^{s,x,\dl}_t := 
	\begin{cases}
		X^{s,x,\dl,\ep,\mathcal{M}}_t, &\text{when} \; \dl \in \Th, \\
		X^{s,x,\iota}_t, &\text{when}\;\dl = \iota. 
	\end{cases}
	\]
	Then, for $p\in [2,p^*]$ with $p^*$ defined in \eqref{p*}, we obtain the following results:
	\begin{enumerate}
		\item\label{theorem(i)}It holds for all $ s\in [0,T]$, $x \in \R^d$, $\dl \in 
		\Th$, $\ep \in (0,1)$, and $\mathcal{M} \in \N$ with $\mathcal{M}\geq (1 \vee \ep^{-2}\mathcal{Z}_d)^{2}$, that
		\begin{align*}
			&\sup_{t\in[s,T]}\E  \lf\| X^{s,x,\iota}_t-X^{s,x,\dl,\ep,\mathcal{M}}_t \rt\|^p  \\
			&\leq K(|\dl|^{\frac{p}{p+\zeta}}+\ep^q(1+\mathcal{Z}_d^{\frac{p}{2}})+(\ep^{-2}\mathcal{Z}_d)^{p-1}\mathcal{M}^{-\frac{p}{2}}) (1+L_d^{p}) \big(1+\|x\|^{p_0}+N
			_d^{p_0}\big)e^{K(1+L_d^{p_0})}.
		\end{align*} 
		\item\label{theoremnew} It holds for all $s,\tl{s}\in [0,T]$, $t \in[s,T]$, $\tl{t} \in [\tl{s},T]$, $x,\tl{x} \in \R^d$, $\dl \in 
		\wtl{\Th}$, $\ep \in (0,1)$, and $\mathcal{M} \in \N$ with $\mathcal{M}\geq (1 \vee \ep^{-2}\mathcal{Z}_d)^{2}$, that
		\begin{align*}
			&\E\lf\|\X^{s,x,\dl}_t - \X^{\tl{s},\tl{x},\dl}_{\tl{t}} \rt\|^p \leq  K\bigg\{\|x -\tl{x}\|^p+(1+L_d^{\frac{p}{2}})\big(1+\|x\|^{p_0}+\|\tl{x}\|^{p_0}+N_d^{p_0}\big)\\
			&\quad\times\big(|\tl{t}-t|+|\tl{s}-s|+(|\tl{t}-t|^{p}+|\tl{s}-s|^p)(1+(  \ep^{-2}\mathcal{Z}_d )^{p-1}\mathcal{M}^{-\frac{p}{2}})
			\big)\bigg\}e^{K(|\dl|^{-\frac{p}{4}}+L_d^{p_0})}.
		\end{align*}
		\item\label{theorem(ii)} It holds for all $s,\tl{s}\in [0,T]$, $t \in[s,T]$, $\tl{t} \in [\tl{s},T]$, $x,\tl{x} \in \R^d$, $\dl \in 
		\wtl{\Th}$, $\ep \in (0,1)$, and $\mathcal{M} \in \N$ with $\mathcal{M}\geq (1 \vee \ep^{-2}\mathcal{Z}_d)^{2}$, that
		\begin{align*}
			&\E\lf\|\X^{s,x,\dl}_t - \X^{\tl{s},\tl{x},\dl}_{\tl{t}} \rt\|^p \leq  K\bigg\{\|x -\tl{x}\|^p+(1+L_d^{\frac{p}{2}})\big(1+\|x\|^{p_0}+\|\tl{x}\|^{p_0}+N_d^{p_0}\big)\\
			&\quad\times\big(|\tl{t}-t|+|\tl{s}-s|+(|\tl{t}-t|^{p}+|\tl{s}-s|^p)(1+(  \ep^{-2}\mathcal{Z}_d )^{p-1}\mathcal{M}^{-\frac{p}{2}})+|\dl|^{\frac{p}{p+\zeta}}\big)\bigg\}e^{K(1+L_d^{p_0})}.
		\end{align*}
			\end{enumerate}
	\end{theorem}

\begin{remark}
Theorem~\ref{theorem1}\,\ref{theorem(i)} shows scheme \eqref{Tamedscheme} converges to the exact solution \eqref{exactsolu} in $L^p$-sense with a rate of convergence arbitrarily close to $1/p$ which is consistent with the existing literature; see, e.g., \cite{2017eulerdiffusionlevynoise}. More precisely, the estimate in Theorem~\ref{theorem1}\,\ref{theorem(i)} quantifies the strong error between the exact solution \eqref{exactsolu} and scheme \eqref{Tamedscheme}, and presents three error contributions: the time-discretization error $|\dl|^{\frac{p}{p+\zeta}}$, the truncation error $\ep^{q}(1+\mathcal{Z}_d^{\frac p2})$, and the Monte Carlo error $(\ep^{-2}\mathcal{Z}_d)^{p-1}\mathcal{M}^{-\frac p2}$.

Moreover, Theorem~\ref{theorem1}\,\ref{theoremnew}--\ref{theorem(ii)} provide stability and continuity estimates for $(\X^{s,x,\dl}_t)$ with respect to perturbations of the initial time $s$, the initial value $x$, and the evaluation time $t$.
The estimate in Theorem~\ref{theorem1}\,\ref{theoremnew} shows that, for a fixed step size $|\dl|$, scheme \eqref{Tamedscheme} is stable with respect to $s,x$, and $t$. Furthermore, Theorem~\ref{theorem1}\,\ref{theorem(ii)} provides a continuity estimate for scheme \eqref{Tamedscheme} in the vanishing step-size regime, i.e., $|\dl| \to 0$.
\end{remark}

\section{Numerical Simulations}\label{Numerical Simulations}
In this section, we provide numerical experiments through two examples. The first example is a standard polynomial-growth model, whereas the second one is a L\'evy-driven SDE with \(3/2\)-type volatility \cite{MR3225549,MR3070371,MR3428878}. We show that both examples satisfy the assumptions and simulation results support our main findings.
\subsection{1D Lévy-driven SDE}\label{ex:1}
Consider the one-dimensional Lévy-driven SDE 
\begin{equation*}
	\rd X_t=(X_{t-}-X_{t-}^3)\,\rd t+X_{t-}^2\,\rd W_t+X_{t-}\,\rd L_t,
\end{equation*}
for $t\in[0,T]$ with initial condition $X_0=2$, where
 \[ L_t:=\sum_{k=1}^{N_t} Z_k-\lambda \E[Z_1]\,t. \] Here, $(N_t)_{t\ge0}$ is a Poisson process with intensity $\lambda>0$, and $\{Z_k\}_{k\in\N}  \sim \mathcal{N}(0.05,0.15^2)$ is an i.i.d.\ sequence of random variables. Define the Poisson random meaure
 \[
 \pi((0,t] \times E) := \#\{s\in(0,t]: L_s-L_{s-} \in E\}, \qquad t\in(0,T],\qquad E \in \mathcal{B}(\R \backslash \{0\}),
 \]
and its compensated Poisson random measure $\tl{\pi}(\rd z,\rd t):= \pi(\rd z,\rd t) - \nu(\rd z)\rd t$, where
\begin{equation}\label{jumpsetting}\tag{$\lozenge$}
	\nu(\rd z)=\lambda \phi(z)\,\rd z,
	\qquad 
	\phi(z)=\frac{1}{\sqrt{2\pi\cdot 0.15^2}}
	e^{-\frac{(z-0.05)^2}{2\cdot 0.15^2}}.
\end{equation}
Then, for every $t\in[0,T]$, we have
\begin{equation}\label{SDEexample}
	\rd X_t = (X_{t-}-X_{t-}^3)\,\rd t + X_{t-}^2\rd W_t +  \int_{\R} X_{t-}z \,\tl{\pi}(\rd z,\rd t).
\end{equation}
\begin{proposition}\label{prop1}
	The coefficients of SDE \eqref{SDEexample} satisfy Assumptions \ref{A1}--\ref{A5}.
	\begin{proof}
		See Appendix \ref{app:Numerical Simulations}.
	\end{proof}
\end{proposition}

Fix $T = 1$. Let $\dl\in\Th$ correspond to the uniform grid $t_i=i|\dl|$, $i=0,1,\dots,n$, with step size $|\dl|=1/n$. The tamed Euler scheme of \eqref{SDEexample} reads as follows: for $i=1,\dots,n$,
\begin{equation}\label{Schemeexample}
	\begin{aligned}
		Y^{x,\dl,\ep,\mathcal{M}}_{t_i}
		&= Y^{x,\dl,\ep,\mathcal{M}}_{t_{i-1}}
		+ \frac{Y^{x,\dl,\ep,\mathcal{M}}_{t_{i-1}}-(Y^{x,\dl,\ep,\mathcal{M}}_{t_{i-1}})^3}
		{\bigl(1+|\dl|^{\frac{1}{2}}|Y^{x,\dl,\ep,\mathcal{M}}_{t_{i-1}}|^{4}\bigr)^{\frac{1}{2}}}\,|\dl| + \frac{(Y^{x,\dl,\ep,\mathcal{M}}_{t_{i-1}})^2}
		{\bigl(1+|\dl|^{\frac{1}{2}}|Y^{x,\dl,\ep,\mathcal{M}}_{t_{i-1}}|^{4}\bigr)^{\frac{1}{2}}}\,(W_{t_i}-W_{t_{i-1}}) \\
		&\quad+ \int_{t_{i-1}}^{t_i}
		\int_{A_\ep} Y^{x,\dl,\ep,\mathcal{M}}_{t_{i-1}} z\, \pi(\rd z,\rd r)
		-\Bigl(\frac{\nu(A_\ep)}{\mathcal{M}}\,Y^{x,\dl,\ep,\mathcal{M}}_{t_{i-1}}
		\sum_{j=1}^{\mathcal{M}} V^{\dl,\ep,\mathcal{M}}_{j,t_{i-1}}\Bigr)\,|\dl|,
	\end{aligned}
\end{equation}
with $Y^{x,\dl,\ep,\mathcal{M}}_{t_0}=x=X_0$, where
$
\nu(A_\ep)=\lambda \int_{A_\ep}\phi(z)\,\rd z
$
is the intensity of large jumps.
Here,
\[
\int_{t_{i-1}}^{t_i}\int_{A_\ep} Y^{x,\dl,\ep,\mathcal{M}}_{t_{i-1}} z\, \pi(\rd z,\rd r)
\]
represents the contribution of all jumps with size at least $\ep$ occurring on the interval $(t_{i-1},t_i]$. To simulate this term, we let $\pi_\ep(t):=\pi( (0,t]\times A_\ep)$
denote the counting process of jumps of $L$ with sizes in $A_\ep$, and let \(\{\mathscr{Z}_k\}_{k\in\N}\) be the corresponding jump sizes, which are i.i.d. with distribution
\[
\nu_\ep(B)=\frac{\nu(B\cap A_\ep)}{\nu(A_\ep)},
\qquad B\in\mathcal{B}(\R).
\]
Then, we have
\[
\int_{t_{i-1}}^{t_i}\int_{A_\ep} Y^{x,\dl,\ep,\mathcal{M}}_{t_{i-1}} z\, \pi(\rd z,\rd r)
\overset{d}{=}
Y^{x,\dl,\ep,\mathcal{M}}_{t_{i-1}}
\sum_{k=\pi_\ep(t_{i-1})+1}^{\pi_\ep(t_i)} \mathscr{Z}_k.
\]

We now examine the numerical behaviour of scheme \eqref{Schemeexample}. Figure~\ref{fig:convergence1}(\subref{fig:ex1thm1}) illustrates the strong $L^2$ convergence rate calculated using
\[
\bigg(\frac{1}{N}\sum_{k=1}^{N}|X_{T,k}-Y^{x,\dl,\ep,\mathcal{M}}_{t_n,k}|^2\bigg)^{\frac{1}{2}}
\]
with $N = 1000$, where $X_{T,k}$ is approximated by scheme \eqref{Schemeexample} using the fine step size $|\dl|=2^{-18}$, while $Y^{x,\dl,\ep,\mathcal{M}}_{t_n,k}$ is computed with $|\dl| \in \{ 2^{-17},2^{-16},...,2^{-11},2^{-10}\}$, $\ep=0.05$, $\mathcal{M}=1000$, and the jump intensity $\lambda = 3$.
The numerical results presented in Figure~\ref{fig:convergence1}(\subref{fig:ex1thm1}) indicate that the strong $L^2$ convergence rate of the proposed method approaches $1/2$, in alignment with the theoretical result established in Theorem \ref{theorem1}\,\ref{theorem(i)}.

Figure~\ref{fig:convergence1}(\subref{fig:ex1thm2}) illustrates the stability of scheme \eqref{Schemeexample} with respect to perturbations of the initial value. Specifically, we fix the step size \( |\dl|=2^{-18} \), \( \epsilon=0.05 \), \( \mathcal{M}=5000 \), and $\lambda = 0.5$, then compute
\begin{equation}\label{ex13}
	\bigg(\frac{1}{N}\sum_{k=1}^{N}|Y_{t_n,k}^{x,\dl,\epsilon,\mathcal{M}}-Y^{\tilde x,\dl,\ep,\mathcal{M}}_{t_n,k}|^2\bigg)^{\frac{1}{2}}
\end{equation}
for different initial gaps \( |x-\tilde x| \in \{10^{-10},10^{-\frac{68}{7}}, 10^{-\frac{66}{7}},...,10^{-\frac{60}{7}},10^{-\frac{58}{7}}, 10^{-8}\} \) with $N = 1000$.
The resulting log-log plot shows an approximately linear dependence of the $L^2$ distance on the initial condition, with observed slope close to $1$. This is consistent with the estimate established in Theorem~\ref{theorem1}\,\ref{theoremnew}.

Figures~\ref{fig:convergence1}(\subref{fig:ex1thm31})-\ref{fig:convergence1}(\subref{fig:ex1thm32}) further examine the joint dependence of the numerical solution on both the initial perturbation \( |x-\tilde x| \) and the time step size \( |\dl| \). We let 
\[
|x-\tilde x| \in \{10^{-3},10^{-\frac{19}{7}}, 10^{-\frac{17}{7}},...,10^{-\frac{11}{7}},10^{-\frac{9}{7}}, 10^{-1}\}, \quad |\dl| \in \{ 2^{-17},2^{-16},...,2^{-11},2^{-10}\},
\]
and for each pair \( (|x-\tilde x|,|\dl|) \), we compute the error in \eqref{ex13}, with $\epsilon=0.05$, $\mathcal{M}=1000$, and $\lambda = 3$ fixed. The heatmap and the surface plot in Figures~\ref{fig:convergence1}(\subref{fig:ex1thm31})--\ref{fig:convergence1}(\subref{fig:ex1thm32}) show that the $L^2$ distance decreases as either the initial gap or the step size becomes smaller, which confirms the result in Theorem~\ref{theorem1}\,\ref{theorem(ii)}.

\begin{figure}[t]
	\centering
	\begin{subfigure}{0.44\textwidth}
		\centering
		\includegraphics[width=\textwidth]{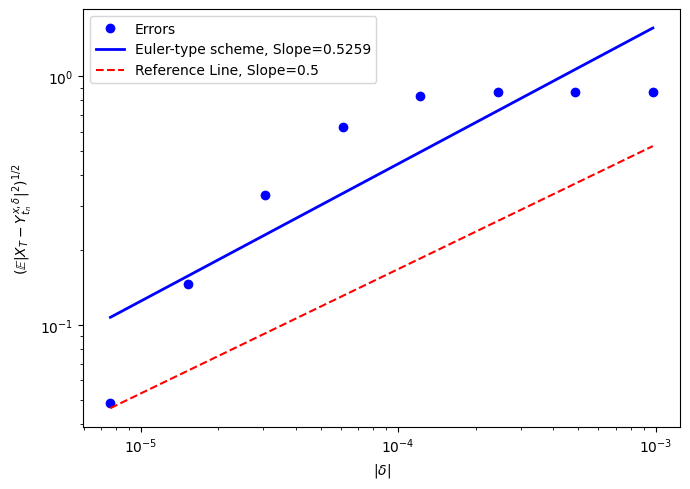}
		\subcaption{(a)}
		\label{fig:ex1thm1}
	\end{subfigure}
	\begin{subfigure}{0.44\textwidth}
		\centering
		\includegraphics[width=\textwidth]{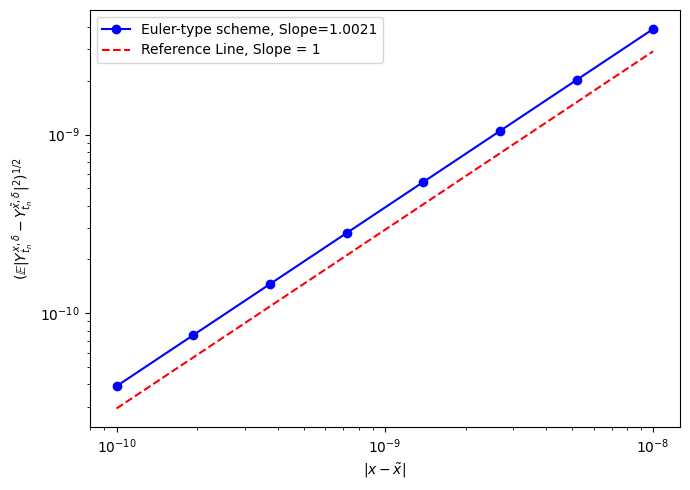}
		\subcaption{(b)}
		\label{fig:ex1thm2}
	\end{subfigure}
	\begin{subfigure}{0.48\textwidth}
		\centering
		\includegraphics[width=\textwidth]{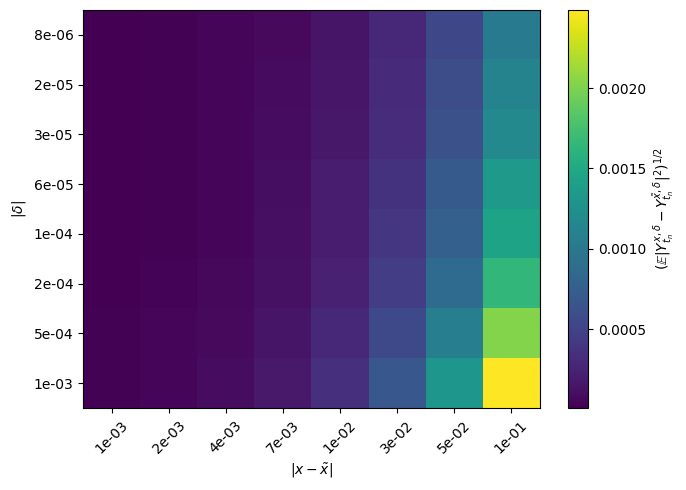}
		\subcaption{(c)}
		\label{fig:ex1thm31}
	\end{subfigure}
	\begin{subfigure}{0.44\textwidth}
		\centering
		\includegraphics[width=\textwidth]{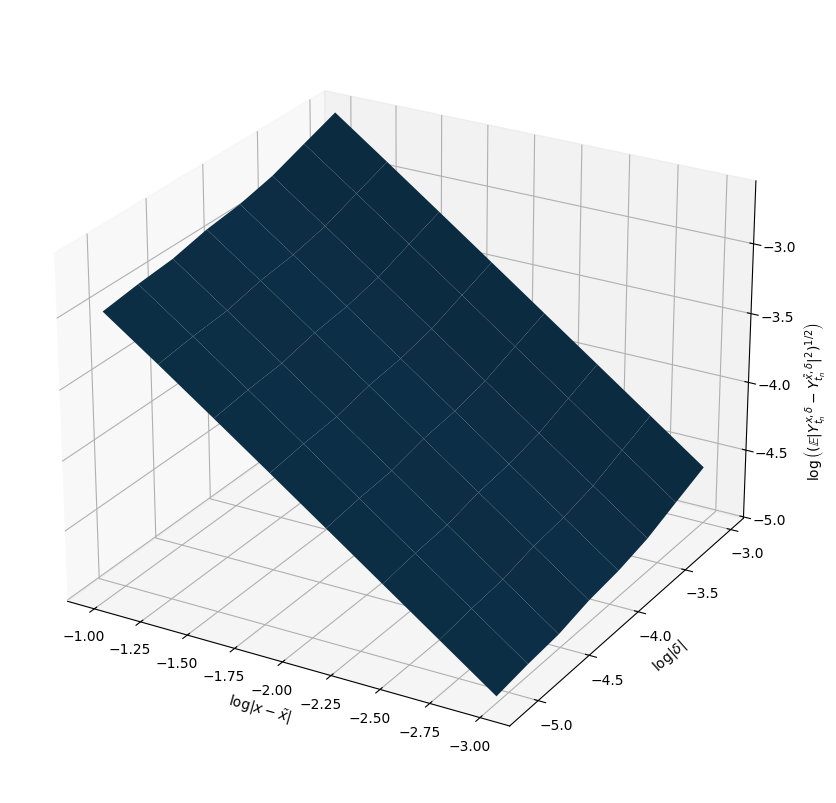}
		\subcaption{(d)}
		\label{fig:ex1thm32}
	\end{subfigure}
	\caption{Simulation results for the tamed Euler scheme \eqref{Schemeexample} corresponding to Theorem \ref{theorem1}\,\ref{theorem(i)}--\ref{theorem(ii)}. For brevity, we denote $Y_{t_n}^{x,\delta,\ep,\mathcal{M}}$ and $Y_{t_n}^{\tilde{x},\delta,\ep,\mathcal{M}}$ by $Y_{t_n}^{x,\delta}$ and $Y_{t_n}^{\tilde{x},\delta}$, respectively.}
	\label{fig:convergence1}
\end{figure}

\subsection{1D SDE in financial model}\label{ex:2}
\begin{figure}[t]
	\centering
	\begin{subfigure}{0.44\textwidth}
		\centering
		\includegraphics[width=\textwidth]{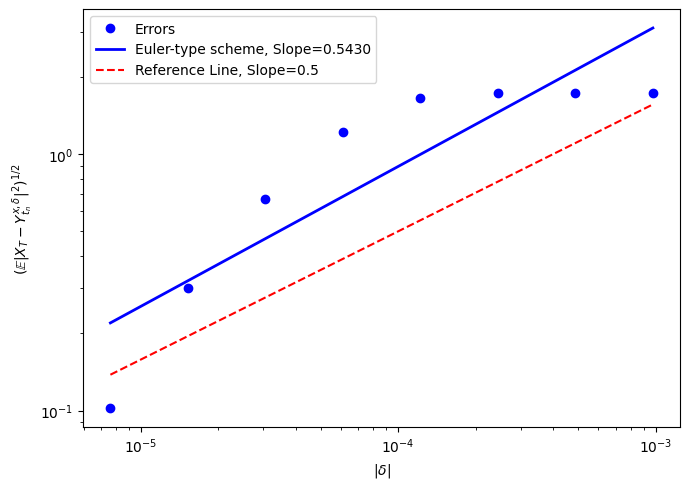}
		\subcaption{(a)}
		\label{fig:ex2thm1}
	\end{subfigure}
	\begin{subfigure}{0.44\textwidth}
		\centering
		\includegraphics[width=\textwidth]{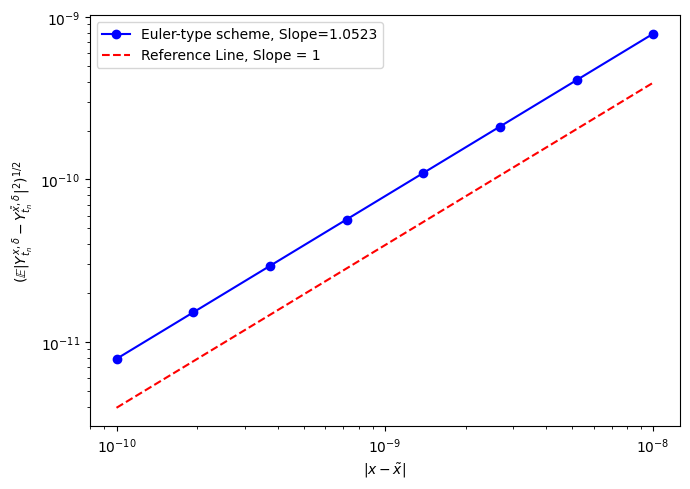}
		\subcaption{(b)}
		\label{fig:ex2thm2}
	\end{subfigure}
	\begin{subfigure}{0.48\textwidth}
		\centering
		\includegraphics[width=\textwidth]{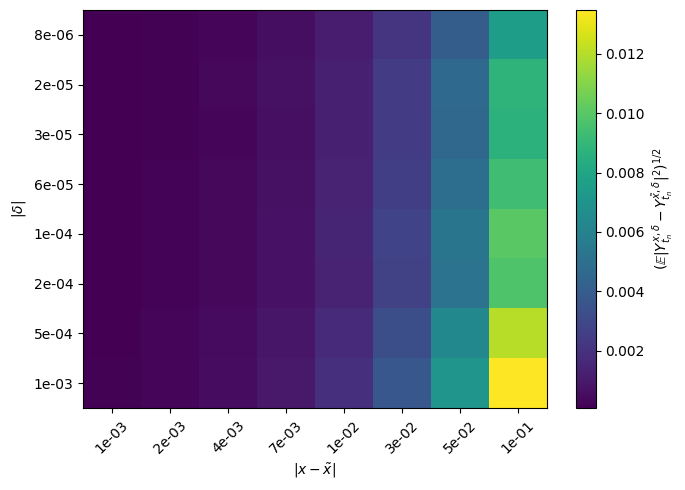}
		\subcaption{(c)}
		\label{fig:ex2thm31}
	\end{subfigure}
	\begin{subfigure}{0.44\textwidth}
		\centering
		\includegraphics[width=\textwidth]{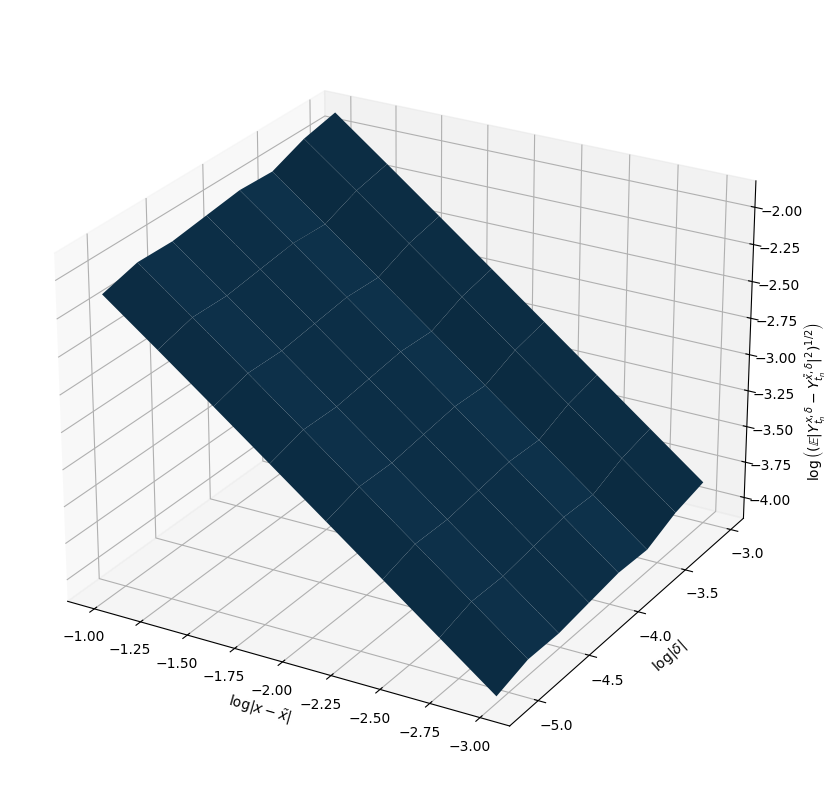}
		\subcaption{(c)}
		\label{fig:ex2thm32}
	\end{subfigure}
	\caption{Simulation results for the tamed Euler scheme \eqref{Schemeexample2} corresponding to Theorem \ref{theorem1}\,\ref{theorem(i)}--\ref{theorem(ii)}.}
	\label{fig:convergence2}
\end{figure}
Consider the one-dimensional Lévy-driven SDE with 3/2-type volatility
	\begin{equation}\label{SDEexample2}
		\rd X_t = aX_{t-}(b-|X_{t-}|)\,\rd t + c|X_{t-}|^{\frac32}\,\rd W_t
		+ X_{t-}\int_{\R} z\,\tilde\pi(\rd z,\rd t),
	\end{equation}
	for all $t \in[0,1]$, with parameters $a,b,c\ge0$ and initial value $X_0=2$. We use the Lévy measure defined in \eqref{jumpsetting}. 
	
	\begin{proposition}\label{prop2}
		The coefficients of SDE \eqref{SDEexample2} satisfy Assumptions \ref{A1}--\ref{A5}.
		\begin{proof}
			See Appendix \ref{app:Numerical Simulations}.
		\end{proof}
	\end{proposition}
	We fix $T = 1$ and choose $a=c=1$, $b=2$. Let $\dl\in\Th$ correspond to the uniform grid $t_i=i|\dl|$, $i=0,1,\dots,n$, with step size $|\dl|=1/n$. 
	The tamed Euler scheme of \eqref{SDEexample2} is given by, for $i=1,\dots,n$,
	\\
	\begin{equation}\label{Schemeexample2}
		\begin{aligned}
			Y_{t_i}^{x,\dl,\ep,\mathcal{M}}
			&=Y_{t_{i-1}}^{x,\dl,\ep,\mathcal{M}}
			+\frac{Y_{t_{i-1}}^{x,\dl,\ep,\mathcal{M}}(2-|Y_{t_{i-1}}^{x,\dl,\ep,\mathcal{M}}|)}
			{\bigl(1+|\dl|^{\frac12}|Y_{t_{i-1}}^{x,\dl,\ep,\mathcal{M}}|^{2}\bigr)^{\frac12}}\,|\dl|+\frac{|Y_{t_{i-1}}^{x,\dl,\ep,\mathcal{M}}|^{\frac32}}
			{\bigl(1+|\dl|^{\frac12}|Y_{t_{i-1}}^{x,\dl,\ep,\mathcal{M}}|^{2}\bigr)^{\frac12}}(W_{t_i}-W_{t_{i-1}}) \\
			&\quad+ \int_{t_{i-1}}^{t_i}
			\int_{A_\ep} Y^{x,\dl,\ep,\mathcal{M}}_{t_{i-1}} z\, \pi(\rd z,\rd r)
			-\Bigl(\frac{\nu(A_\ep)}{\mathcal{M}}Y_{t_{i-1}}^{x,\dl,\ep,\mathcal{M}}\sum_{j=1}^{\mathcal{M}}V^{\dl,\ep,\mathcal{M}}_{j,t_{i-1}}\Bigr)|\dl|.
		\end{aligned}
	\end{equation}
with $Y^{x,\dl,\ep,\mathcal{M}}_{t_0}=x=X_0$. Using the same numerical setting as in Example~\ref{ex:1}, we present the simulation results in Figure~\ref{fig:convergence2}. In Figure~\ref{fig:convergence2}(\subref{fig:ex2thm1}), the log-log plot of the strong $L^2$ error against the step size $|\dl|$ exhibits an empirical slope close to $1/2$, in agreement with Theorem~\ref{theorem1}\,\ref{theorem(i)}. 

Figure~\ref{fig:convergence2}(\subref{fig:ex2thm2}) shows that the $L^2$ error between two numerical solutions (cf. \eqref{ex13}) scales linearly with the initial perturbation $|x-\tilde x|$, which is consistent with Theorem~\ref{theorem1}\,\ref{theoremnew}. 

Finally, Figures~\ref{fig:convergence2}(\subref{fig:ex2thm31})--\ref{fig:convergence2}(\subref{fig:ex2thm32}) illustrate the joint dependence of the $L^2$ distance between two numerical solutions \eqref{ex13} on the initial gap $|x-\tilde x|$ and the step size $|\dl|$. The numerical error decreases as either the initial perturbation or the step size becomes smaller, which confirms to the estimate given in Theorem~\ref{theorem1}\,\ref{theorem(ii)}.

\section{Proof of the Main Results}\label{Proof of the Main Results}
This section is devoted to the proof of Theorem~\ref{theorem1}. The argument is built in two stages. 
We first establish uniform moment bounds for the exact solution to SDE \eqref{exactsolu} and for the successive approximations  \eqref{Tamedscheme}, \eqref{tamedscheme1}, and \eqref{tamedscheme2} in Section~\ref{section5.1}. These estimates provide the basic stability and regularity properties needed to control both discretization errors and perturbations of initial value and time. 
We then combine these auxiliary bounds to prove each part of Theorem~\ref{theorem1}.
For clarity, the proof is organized into three subsections from Section \ref{section5.2} to \ref{section5.4} corresponding to assertions Theorem \ref{theorem1}\,\ref{theorem(i)}--\ref{theorem(ii)}.

\subsection{Moment estimates}\label{section5.1}
 In this section, we adapt the approach of \cite{L.GononC.Schwab} and \cite{MR4951649} to derive uniform $p_0$-th moment bounds of SDE \eqref{exactsolu} and its numerical approximations \eqref{Tamedscheme}, \eqref{tamedscheme1}, and \eqref{tamedscheme2}. The proofs are postponed to Appendix \ref{app:section5.1}.
\begin{lemma}\label{exactmomentbounds}
	Let Assumptions \ref{A1}--\ref{A4} hold, then there exists $K>0$ such that for all $s\in [0,T]$ and $x\in \R^d$,
	\[
	\sup_{t\in[s,T]}\E\|X^{s,x,\iota}_t\|^{p_0} \leq K\big( 1+\|x\|^{p_0}+N_d^{\frac{p_0}{2}}\big)e^{K(1+L_d^{\frac{p_0}{2}})}.\\
	\]
\end{lemma}

\begin{lemma}\label{L3}
	Let Assumptions \ref{A1}--\ref{A4} hold, then there exists $K>0$ such that, for all  $ s\in [0,T]$, $x\in \R^d$, and $\dl \in \Th$,
	\begin{equation*}
			\sup_{t\in[s,T]}\E\|X^{s,x,\dl}_t\|^{p_0}  \leq  K\big( 1+\|x\|^{p_0}+N_d^{\frac{p_0}{2}}\big)e^{K\big(1+L_d^{\frac{p_0}{2}}\big)}.
	\end{equation*}
\end{lemma}

\begin{lemma}\label{epbounds}
	Let Assumptions \ref{A1}--\ref{A4} hold, then there exists $K>0 $ such that, for all $s\in [0,T]$, $x\in \R^d$, $\dl \in \Th$, and $\ep \in(0,1)$,
	\[
	\sup_{t\in[s,T]}\E\|X^{s,x,\dl,\ep}_t\|^{p_0} \leq  K\big( 1+\|x\|^{p_0}+N_d^{\frac{p_0}{2}}\big)e^{K\big(1+L_d^{\frac{p_0}{2}}\big)}.
	\]
\end{lemma}

\begin{lemma}\label{scheme3bounds}
	Let Assumptions \ref{A1}--\ref{A5} hold, then there exists $K>0$ such that, for all $ s\in [0,T]$, $x\in \R^d$, $\dl \in \Th$, $\ep \in (0,1)$, and $\mathcal{M} \in \N$ with $\mathcal{M}\geq (1 \vee \ep^{-2}\mathcal{Z}_d)^{2}$,
	\[
	\sup_{t\in[s,T]}   \E\|X^{s,x,\dl,\ep,\mathcal{M}}_{t}\|^{p_0} 
	\leq K  \lf(1+ \|x\|^{p_0}+N_d^{\frac{p_0}{2}}\rt)e
	^{K(1+L_d^{\frac{p_0}{2}})}.
	\]
\end{lemma}

\subsection{Proof of Theorem \ref{theorem1}\,\ref{theorem(i)}}\label{section5.2}
In this section, we establish several auxiliary estimates needed for the proof of Theorem~\ref{theorem1}\,\ref{theorem(i)}. Their proofs are deferred to Appendix~\ref{app:section5.2}. We then use these estimates to complete the proof of Theorem~\ref{theorem1}\,\ref{theorem(i)}.

\begin{lemma}\label{convergencescheme1}
	Let Assumptions \ref{A1}--\ref{A4} hold, then there exists $K>0$ such that, for all $s\in [0,T]$, $ x\in \R^d$, and $\dl \in \Th$,
	\[
	\sup_{t\in[s,T]} \E \|X_t^{s,x,\iota} - X_t^{s,x,\dl} \|^{p}  \leq K |\dl|^{\frac{p}{p+\zeta}}(1+L_d^{p})(1+\|x\|^{p_0}+N_d^{\frac{p_0}{2}})e^{K(1+L_d^{\frac{p_0}{2}})}.
	\]
\end{lemma}
\begin{lemma}\label{convergencescheme2}
	Let Assumptions \ref{A1}--\ref{A5} hold, then there exists $K>0$ such that, for all $s\in [0,T]$, $x\in \R^d$, $\dl \in \Th$, and $ \ep \in(0,1)$,
	\begin{equation*}
		\sup_{t\in[s,T]}	\E\|X_t^{s,x,\dl} - X_t^{s,x,\dl,\ep} \|^{p} \leq K \big(|\dl|^{\frac{p}{p+\zeta}}+\ep^{q}(1+\mathcal{Z}_d^{\frac{p}{2}})\big)(1+L_d^{\frac{p}{2}})(1+\|x\|^{p_0}+N_d^{\frac{p_0}{2}})e^{K(1+L_d^{\frac{p_0}{2}})}.
	\end{equation*}
\end{lemma}

\begin{lemma}\label{convergencescheme3}
	Let Assumptions \ref{A1}--\ref{A5} hold, then there exists $K>0$ such that, for all  $s\in [0,T]$, $x\in \R^d$, $\dl \in \Th$, $\ep \in (0,1)$, and $\mathcal{M} \in \N$ with $\mathcal{M}\geq (1 \vee \ep^{-2}\mathcal{Z}_d)^{2}$,
	\begin{align*}
		&\sup_{t\in[s,T]}	\E\|X_t^{s,x,\dl,\ep} - X_t^{s,x,\dl,\ep,\mathcal{M}} \|^{p}\leq  K(|\dl|^{\frac{p}{p+\zeta}}+(\ep^{-2}\mathcal{Z}_d)^{p-1}\mathcal{M}^{-\frac{p}{2}}) (1+L_d^{\frac{p}{2}})\big(1+\|x\|^{p_0}+N_d^{p_0}\big)e^{K(1+L_d^{p_0})}.
	\end{align*}
\end{lemma}

\begin{proof}[\textbf{Proof of Theorem \ref{theorem1}\,\ref{theorem(i)}}]
	Fix $s\in [0,T]$, $x\in \R^d$, $\dl \in \Th$, $\ep\in(0,1)$, and $\mathcal{M}\in\N$ with $\mathcal{M}\ge (1\vee \ep^{-2}\mathcal{Z}_d)^{2}$. By the triangle inequality,
	\begin{align*}
		&\sup_{t\in[s,T]}\E\|X_t^{s,x,\iota}-X_t^{s,x,\dl,\ep,\mathcal{M}}\|^p\\
		&\leq K\bigg(
		\sup_{t\in[s,T]}\E\|X_t^{s,x,\iota}-X_t^{s,x,\dl}\|^p
		+\sup_{t\in[s,T]}\E\|X_t^{s,x,\dl}-X_t^{s,x,\dl,\ep}\|^p
		+\sup_{t\in[s,T]}\E\|X_t^{s,x,\dl,\ep}-X_t^{s,x,\dl,\ep,\mathcal{M}}\|^p
		\bigg).
	\end{align*}
	The three terms are controlled by Lemmas \ref{convergencescheme1}, \ref{convergencescheme2}, and \ref{convergencescheme3}, respectively. Collecting these bounds yields
	\begin{align*}
		\sup_{t\in[s,T]}\E\|X_t^{s,x,\iota}-X_t^{s,x,\dl,\ep,\mathcal{M}}\|^p
		&\le K\Bigl(|\dl|^{\frac{p}{p+\zeta}}+\ep^q\bigl(1+\mathcal{Z}_d^{\frac{p}{2}}\bigr)+(\ep^{-2}\mathcal{Z}_d)^{p-1}\mathcal{M}^{-\frac{p}{2}}\Bigr)\\
		&\quad\times(1+L_d^{p})\bigl(1+\|x\|^{p_0}+N_d^{p_0}\bigr)e^{K(1+L_d^{p_0})}.
	\end{align*}
\end{proof}

\subsection{Proof of Theorem \ref{theorem1}\,\ref{theoremnew}}\label{section5.3}
This section collects several auxiliary estimates that will be used in the proof of Theorem \ref{theorem1}\,\ref{theoremnew}. They quantify the stability of both the exact solution and its numerical approximation under perturbations of the evaluation time, the initial value, and the initial time. 
More precisely, Lemmas \ref{exactdifferenttimet}–\ref{exactdifferents} establish temporal and spatial regularity estimates for SDE \eqref{exactsolu}, while Lemmas \ref{differenttimet}–\ref{differents'} provide the corresponding bounds for scheme \eqref{Tamedscheme}. The proofs of these lemmas are postponed to Appendix \ref{app:section5.3}.

\begin{lemma}\label{exactdifferenttimet}
	Let Assumptions \ref{A1}--\ref{A4} hold, then there exists $K>0$ such that, for all $s\in [0,T]$, $t, \tl{t} \in[s,T]$, and $x\in \R^d$,
	\begin{align*}
		\E\|X_{\tl{t}}^{s,x,\iota} - X_t^{s,x,\iota} \|^{p}\leq K\lf(|\tl{t}-t|+|\tl{t}-t|^{p}\rt)(1+L_d^{\frac{p}{2}})\big(1+\|x\|^{p_0}+N_d^{\frac{p_0}{2}}\big)e^{K(1+L_d^{\frac{p_0}{2}})}.
	\end{align*}
\end{lemma}

\begin{lemma}\label{exactdifferentx}
	Let Assumptions \ref{A1}--\ref{A4} hold, then there exists $K>0$ such that, for all $s\in[0,T]$, and $x,\tl{x}\in\R^d$,
	\[
	\sup_{t\in[s,T]}\E\|X_{t}^{s,x,\iota} - X_t^{s,\tl{x},\iota} \|^{p}\leq  K\|x-\tl{x}\|^pe^{K(1+L_d^{\frac{p_0}{2}})}.
	\]
\end{lemma}

\begin{lemma}\label{exactdifferents}
	Let Assumptions \ref{A1}--\ref{A4} hold, then there exists $K>0$ such that, for all $s, \tl{s}\in[0,T]$, $t\in[s \vee \tl{s},T]$, and $x\in\R^d$,
	\[
	\E\|X_{t}^{s,x,\iota} - X_t^{\tl{s},x,\iota} \|^{p}\leq  K\big(|\tl{s}-s|+|\tl{s}-s|^p\big)(1+L_d^{\frac{p}{2}}) \big(1+\|x\|^{p_0}+N_d^{\frac{p_0}{2}}\big)e^{K(1+L_d^{\frac{p_0}{2}})}.
	\]
\end{lemma}

\begin{lemma}\label{differenttimet}
	Let Assumptions \ref{A1}--\ref{A5} hold, then there exists $K>0$ such that, for all $s\in [0,T]$, $t, \tl{t} \in[s,T]$, $x\in \R^d$, $\dl \in \Th$, $\ep \in (0,1)$, and $\mathcal{M} \in \N$ with $\mathcal{M}\geq (1 \vee \ep^{-2}\mathcal{Z}_d)^{2}$,
	\begin{align*}
		&\E\|X_{\tl{t}}^{s,x,\dl,\ep,\mathcal{M}} - X_t^{s,x,\dl,\ep,\mathcal{M}} \|^{p}\\
		&\leq K\lf(|\tl{t}-t|+|\tl{t}-t|^{p}(1+(\ep^{-2}\mathcal{Z}_d)^{p-1}\mathcal{M}^{-\frac{p}{2}})\rt) (1+L_d^{\frac{p}{2}})\big(1+\|x\|^{p_0}+N_d^{p_0}\big)e^{K(1+L_d^{p_0})}.
	\end{align*}
\end{lemma}

\begin{lemma}\label{differentx'}
	Let Assumptions \ref{A1}--\ref{A5} hold, then there exists $K>0$ such that, for all $s\in[0,T]$, $t\in[s,T]$, $x,\tl{x}\in\R^d$, $\dl \in \Th$, $\ep \in (0,1)$, and $\mathcal{M} \in \N$ with $\mathcal{M}\geq (1 \vee \ep^{-2}\mathcal{Z}_d)^{2}$,
	\begin{align*}
	&\E\|X_{t}^{s,x,\dl,\ep,\mathcal{M}} - X_{t }^{s,\tl{x},\dl,\ep,\mathcal{M}} \|^{p} \leq K\|x -\tl{x}\|^p e^{K\big(|\dl|^{-\frac{p}{4}}+L_d^{\frac{p}{2}} \big)}.
\end{align*}
\end{lemma}

\begin{lemma}\label{differents'}
	Let Assumptions \ref{A1}--\ref{A5} hold, then there exists $K>0$ such that, for all $s, \tl{s} \in[0,T]$, $t\in[s \vee \tl{s},T]$, $x\in\R^d$, $\dl \in \Th$, $\ep \in (0,1)$, and $\mathcal{M} \in \N$ with $\mathcal{M}\geq (1 \vee \ep^{-2}\mathcal{Z}_d)^{2}$,
	\begin{align*}
		&\E\|X_{t}^{s,x,\dl,\ep,\mathcal{M}} - X_{t}^{\tl{s},x,\dl,\ep,\mathcal{M}} \|^p\\
&\leq K\big( |\tl{s}-s|+|\tl{s}-s|^p(1+(\ep^{-2}\mathcal{Z}_d)^{p-1}\mathcal{M}^{-\frac{p}{2}}) \big)(1+L_d^{\frac{p}{2}})(1+\|x\|^{p_0}+N_d^{p_0})e^{K(|\dl|^{-\frac{p}{4}}+L_d^{p_0})}.
	\end{align*}
\end{lemma}

Now, we are ready to prove Theorem \ref{theorem1}\,\ref{theoremnew}.
\begin{proof}[\textbf{Proof of Theorem \ref{theorem1}\,\ref{theoremnew}}]
	We first consider the case $\dl = \iota$. Fix $s,\tl{s} \in[0,T]$, $x,\tl{x}\in \R^d$, and observe the following estimates hold for all $t\in[s,T]$ and $\tl{t} \in [\tl{s},T]$:
	\begin{align*}
		&\E\|X^{s,x,\iota}_t - X^{\tl{s},\tl{x},\iota}_{\tl{t}} \|^p=\E\|X_{\max\{s,t\}}^{s,x,\iota} - X_{\max\{\tl{s},\tl{t}\}}^{\tl{s},\tl{x},\iota} \|^{p}\nonumber\\
		& \leq K\bigg( \E\|X_{\max\{s,t\}}^{s,x,\iota} - X_{\max\{s,t,\tl{s},\tl{t}\}}^{s,x,\iota} \|^{p}+ \E\|X_{\max\{s,t,\tl{s},\tl{t}\}}^{s,x,\iota} - X_{\max\{s,t,\tl{s},\tl{t}\}}^{s,\tl{x},\iota} \|^{p}\nonumber\\
		&\hspace{10mm}+\E\|X_{\max\{s,t,\tl{s},\tl{t}\}}^{s,\tl{x},\iota} - X_{\max\{s,t,\tl{s},\tl{t}\}}^{\tl{s},\tl{x},\iota} \|^{p}+ \E\|X_{\max\{s,t,\tl{s},\tl{t}\}}^{\tl{s},\tl{x},\iota} - X_{\max\{\tl{s},\tl{t}\}}^{\tl{s},\tl{x},\iota} \|^{p}\bigg).
	\end{align*}
	By Lemmas \ref{exactdifferenttimet}--\ref{exactdifferents} and the fact that $|\max\{s,t,\tl{s},\tl{t}\} -\max\{\tl{s},\tl{t}\} | \leq |\tl{t}-t|$ for all $t\in[s,T]$ and $\tl{t}\in[\tl{s},T]$, we obtain that
	\begin{align}\label{exactii}
		&\E\|X^{s,x,\iota}_t - X^{\tl{s},\tl{x},\iota}_{\tl{t}} \|^p \nonumber\\
		&\leq K(|\tl{t}-t|+|\tl{t}-t|^{p})(1+L_d^{\frac{p}{2}})\big(1+\|x\|^{p_0}+N_d^{\frac{p_0}{2}}\big)e^{K(1+L_d^{\frac{p_0}{2}})}\nonumber\\
		&\quad+ K \|x -\tl{x}\|^p e^{K(1+L_d^{\frac{p_0}{2}})}+   K\big(|\tl{s}-s|+|\tl{s}-s|^p\big)(1+L_d^{\frac{p}{2}}) \big(1+\|\tl{x}\|^{p_0}+N_d^{\frac{p_0}{2}}\big)e^{K(1+L_d^{\frac{p_0}{2}})}\nonumber\\
		&\quad +K\lf(|\tl{t}-t|+|\tl{t}-t|^{p}\rt)(1+L_d^{\frac{p}{2}})\big(1+\|\tl{x}\|^{p_0}+N_d^{\frac{p_0}{2}}\big)e^{K(1+L_d^{\frac{p_0}{2}})}\nonumber\\
		&\leq K\bigg\{\lf(|\tl{t}-t|+|\tl{t}-t|^{p}+|\tl{s}-s|+|\tl{s}-s|^p\rt)(1+L_d^{\frac{p}{2}})\big(1+\|x\|^{p_0}+\|\tl{x}\|^{p_0}+N_d^{\frac{p_0}{2}}\big)\\
		&\hspace{14mm}+\|x -\tl{x}\|^p\bigg\}e^{K(1+L_d^{\frac{p_0}{2}})}.\nonumber
	\end{align}
	
We next turn to the case $\dl \in \Th$.  Fix $s,\tl{s} \in[0,T]$, $x,\tl{x}\in \R^d$, $\dl \in \Th$, $\ep\in(0,1)$, and $\mathcal{M} \in \N$ with $\mathcal{M} \geq (1 \vee \ep^{-2}\mathcal{Z}_d)^{2}$. We combine Lemmas \ref{differenttimet}--\ref{differents'} and the fact that $|\max\{s,t,\tl{s},\tl{t}\} -\max\{\tl{s},\tl{t}\} | \leq |\tl{t}-t|$ for all $t\in[s,T]$ and $\tl{t}\in[\tl{s},T]$ to obtain
\begin{align}\label{schemeii'}
	&\E\|X^{s,x,\dl,\ep,\mathcal{M}}_t - X^{\tl{s},\tl{x},\dl,\ep,\mathcal{M}}_{\tl{t}} \|^p \nonumber\\
	& \leq K\bigg( \E\|X_{\max\{s,t\}}^{s,x,\dl,\ep,\mathcal{M}} - X_{\max\{s,t,\tl{s},\tl{t}\}}^{s,x,\dl,\ep,\mathcal{M}} \|^{p}+ \E\|X_{\max\{s,t,\tl{s},\tl{t}\}}^{s,x,\dl,\ep,\mathcal{M}} - X_{\max\{s,t,\tl{s},\tl{t}\}}^{s,\tl{x},\dl,\ep,\mathcal{M}} \|^{p} \nonumber\\
	&\hspace{10mm}+\E\|X_{\max\{s,t,\tl{s},\tl{t}\}}^{s,\tl{x},\dl,\ep,\mathcal{M}} - X_{\max\{s,t,\tl{s},\tl{t}\}}^{\tl{s},\tl{x},\dl,\ep,\mathcal{M}} \|^{p}+ \E\|X_{\max\{s,t,\tl{s},\tl{t}\}}^{\tl{s},\tl{x},\dl,\ep,\mathcal{M}} - X_{\max\{\tl{s},\tl{t}\}}^{\tl{s},\tl{x},\dl,\ep,\mathcal{M}} \|^{p}\bigg) \nonumber\\
	&\leq K\big(|\tl{t}-t|+|\tl{t}-t|^{p}(1+(\ep^{-2}\mathcal{Z}_d)^{p-1}\mathcal{M}^{-\frac{p}{2}})\big) (1+L_d^{\frac{p}{2}})\big(1+\|x\|^{p_0}+N_d^{p_0}\big)e^{K(1+L_d^{p_0})} \nonumber\\
	&\quad+ K\big(\|x -\tl{x}\|^p  e^{K\big(|\dl|^{-\frac{p}{4}}+L_d^{\frac{p}{2}} \big)} \nonumber\\
	&\quad +K\big( |\tl{s}-s|+|\tl{s}-s|^p(1+(\ep^{-2}\mathcal{Z}_d)^{p-1}\mathcal{M}^{-\frac{p}{2}}) \big)(1+L_d^{\frac{p}{2}})(1+\|\tl{x}\|^{p_0}+N_d^{p_0})e^{K(|\dl|^{-\frac{p}{4}}+L_d^{p_0})} \nonumber\\
	&\quad +K\big(|\tl{t}-t|+|\tl{t}-t|^{p}(1+(\ep^{-2}\mathcal{Z}_d)^{p-1}\mathcal{M}^{-\frac{p}{2}})\big) (1+L_d^{\frac{p}{2}})\big(1+\|\tl{x}\|^{p_0}+N_d^{p_0}\big)e^{K(1+L_d^{p_0})} \nonumber\\
	&\leq  K\bigg\{\lf(|\tl{t}-t|+|\tl{s}-s|+(|\tl{t}-t|^{p}+|\tl{s}-s|^p)(1+(\ep^{-2}\mathcal{Z}_d)^{p-1}\mathcal{M}^{-\frac{p}{2}})\rt) \\
	&\hspace{10mm}\times(1+L_d^{\frac{p}{2}})\big(1+\|x\|^{p_0}+\|\tl{x}\|^{p_0}+N_d^{p_0}\big)+\|x -\tl{x}\|^p\bigg\}e^{K(|\dl|^{-\frac{p}{4}}+L_d^{p_0})}.\nonumber
\end{align}
Consequently, combining \eqref{exactii} and \eqref{schemeii'} finishes the proof.
\end{proof}

\subsection{Proof of Theorem \ref{theorem1}\,\ref{theorem(ii)}}\label{section5.4}
In this section, we present several auxiliary estimates that are required for the proof of Theorem \ref{theorem1}\,\ref{theorem(ii)}. Lemmas \ref{differentx} and \ref{differents} provide continuity estimates of scheme \eqref{Tamedscheme} with respect to perturbations of the initial value and the initial time. Their proofs are deferred to Appendix \ref{app:section5.4}.

\begin{lemma}\label{differentx}
	Let Assumptions \ref{A1}--\ref{A5} hold, then there exists $K>0$ such that, for all $s\in[0,T]$, $t\in[s,T]$, $x,\tl{x}\in\R^d$, $\dl \in \Th$, $\ep \in (0,1)$, and $\mathcal{M} \in \N$ with $\mathcal{M}\geq (1 \vee \ep^{-2}\mathcal{Z}_d)^{2}$,
	\begin{align*}
		&\E\|X_{t}^{s,x,\dl,\ep,\mathcal{M}} - X_t^{s,\tl{x},\dl,\ep,\mathcal{M}} \|^{p} \\
		&\leq   K\big(\|x -\tl{x}\|^p +|\dl|^{\frac{p}{p+\zeta}} (1+L_d^{\frac{p}{2}})(1+\|x\|^{p_0}+\|\tl{x}\|^{p_0}+N_d^{p_0}) \big)e^{K(1+L_d^{p_0})}.
	\end{align*}
\end{lemma}

\begin{lemma}\label{differents}
	Let Assumptions \ref{A1}--\ref{A5} hold, then there exists $K>0$ such that, for all $s, \tl{s} \in[0,T]$, $t\in[s \vee \tl{s},T]$, $x\in\R^d$, $\dl \in \Th$, $\ep \in (0,1)$, and $\mathcal{M} \in \N$ with $\mathcal{M}\geq (1 \vee \ep^{-2}\mathcal{Z}_d)^{2}$,
	\begin{align*}
		&\E\|X_{t}^{s,x,\dl,\ep,\mathcal{M}} - X_{t}^{\tl{s},x,\dl,\ep,\mathcal{M}} \|^p\\
		&\leq K \big( |\tl{s}-s|+|\tl{s}-s|^p(1+(\ep^{-2}\mathcal{Z}_d)^{p-1}\mathcal{M}^{-\frac{p}{2}})+ |\dl|^{\frac{p}{p+\zeta}}  \big)(1+L_d^{\frac{p}{2}})(1+\|x\|^{p_0}+N_d^{p_0})e^{K(1+L_d^{p_0})}.
	\end{align*}
\end{lemma}

Now, we are ready to prove Theorem \ref{theorem1}\,\ref{theorem(ii)}.
\begin{proof}[\textbf{Proof of Theorem \ref{theorem1}\,\ref{theorem(ii)}}]
	For the case $\dl = \iota$, the desired estimate has already been established in the proof of Theorem \ref{theorem1}\,\ref{theoremnew} (cf.\;\eqref{exactii}). It therefore remains to prove the corresponding bound for the case $\dl \in \Th$. Fix $s,\tl{s} \in[0,T]$, $x,\tl{x}\in \R^d$, $\dl \in \Th$, $\ep\in(0,1)$, and $\mathcal{M} \in \N$ with $\mathcal{M} \geq (1 \vee \ep^{-2}\mathcal{Z}_d)^{2}$. We combine Lemmas \ref{differenttimet}, \ref{differentx}, and \ref{differents} to obtain, for all $t\in[s,T]$ and $\tl{t}\in[\tl{s},T]$, that
	\begin{align}\label{schemeii}
		&\E\|X^{s,x,\dl,\ep,\mathcal{M}}_t - X^{\tl{s},\tl{x},\dl,\ep,\mathcal{M}}_{\tl{t}} \|^p \nonumber\\
		& \leq K\bigg( \E\|X_{\max\{s,t\}}^{s,x,\dl,\ep,\mathcal{M}} - X_{\max\{s,t,\tl{s},\tl{t}\}}^{s,x,\dl,\ep,\mathcal{M}} \|^{p}+ \E\|X_{\max\{s,t,\tl{s},\tl{t}\}}^{s,x,\dl,\ep,\mathcal{M}} - X_{\max\{s,t,\tl{s},\tl{t}\}}^{s,\tl{x},\dl,\ep,\mathcal{M}} \|^{p} \nonumber\\
		&\hspace{10mm}+\E\|X_{\max\{s,t,\tl{s},\tl{t}\}}^{s,\tl{x},\dl,\ep,\mathcal{M}} - X_{\max\{s,t,\tl{s},\tl{t}\}}^{\tl{s},\tl{x},\dl,\ep,\mathcal{M}} \|^{p}+ \E\|X_{\max\{s,t,\tl{s},\tl{t}\}}^{\tl{s},\tl{x},\dl,\ep,\mathcal{M}} - X_{\max\{\tl{s},\tl{t}\}}^{\tl{s},\tl{x},\dl,\ep,\mathcal{M}} \|^{p}\bigg) \nonumber\\
		&\leq K\big(|\tl{t}-t|+|\tl{t}-t|^{p}(1+(\ep^{-2}\mathcal{Z}_d)^{p-1}\mathcal{M}^{-\frac{p}{2}})\big) (1+L_d^{\frac{p}{2}})\big(1+\|x\|^{p_0}+N_d^{p_0}\big)e^{K(1+L_d^{p_0})} \nonumber\\
		&\quad+ K\big(\|x -\tl{x}\|^p +|\dl|^{\frac{p}{p+\zeta}} (1+L_d^{\frac{p}{2}})(1+\|x\|^{p_0}+\|\tl{x}\|^{p_0}+N_d^{p_0}) \big)e^{K(1+L_d^{p_0})} \nonumber\\
		&\quad + K \big( |\tl{s}-s|+|\tl{s}-s|^p(1+(\ep^{-2}\mathcal{Z}_d)^{p-1}\mathcal{M}^{-\frac{p}{2}})+ |\dl|^{\frac{p}{p+\zeta}}  \big)(1+L_d^{\frac{p}{2}})(1+\|\tl{x}\|^{p_0}+N_d^{p_0})e^{K(1+L_d^{p_0})} \nonumber\\
		&\quad +K\big(|\tl{t}-t|+|\tl{t}-t|^{p}(1+(\ep^{-2}\mathcal{Z}_d)^{p-1}\mathcal{M}^{-\frac{p}{2}})\big) (1+L_d^{\frac{p}{2}})\big(1+\|\tl{x}\|^{p_0}+N_d^{p_0}\big)e^{K(1+L_d^{p_0})} \nonumber\\
		&\leq  K\bigg\{\lf(|\tl{t}-t|+|\tl{s}-s|+(|\tl{t}-t|^{p}+|\tl{s}-s|^p)(1+(\ep^{-2}\mathcal{Z}_d)^{p-1}\mathcal{M}^{-\frac{p}{2}})+|\dl|^{\frac{p}{p+\zeta}}\rt) \\
		&\hspace{10mm}\times(1+L_d^{\frac{p}{2}})\big(1+\|x\|^{p_0}+\|\tl{x}\|^{p_0}+N_d^{p_0}\big)+\|x -\tl{x}\|^p\bigg\}e^{K(1+L_d^{p_0})}.\nonumber
	\end{align}
	Consequently, combining \eqref{exactii} and \eqref{schemeii} finishes the proof.
\end{proof}

\newpage
\appendix

\makeatletter
\@addtoreset{lemma}{section}      % reset lemma each appendix section
\renewcommand{\thelemma}{\thesection.\arabic{lemma}}
\allowdisplaybreaks

\section{Proof of results in Section \ref{Assumptions and Main Results}}\label{app:Assumptions and Main Results}

\begin{proof}[\textbf{Proof of Remark \ref{remark1}}]
By Assumption \ref{A2}, we obtain, for all $x\in\R^d$, that
\[
\|\xi(x)-\xi(0)\|
\le L\|x\|\bigl(1+\|x\|^{\frac{\chi}{2}}\bigr).
\]
By the triangle inequality,
\[
\|\xi(x)\|\le \|\xi(0)\|+L\|x\|+L\|x\|^{\frac{\chi}{2}+1}.
\]
By using $\|x\|\le 1+\|x\|^{\frac{\chi}{2}+1}$, we obtain
\[
\|\xi(x)\|
\le \|\xi(0)\|+L\bigl(1+\|x\|^{\frac{\chi}{2}+1}\bigr)+L\|x\|^{\frac{\chi}{2}+1}
\le K\bigl(1+\|x\|^{\frac{\chi}{2}+1}\bigr),
\]
for a constant $K>0$ depending on $L$ and $\|\xi(0)\|$.
\end{proof}

\begin{proof}[\textbf{Proof of Remark \ref{exist}}]
		By Assumption \ref{A1}, we have, for all $x\in\R^d$, that
		\[
		2x\lf(\mu(x)-\mu(0)\rt) + (p_0 -1)\|\sg(x)-\sg(0)\|^2 \leq L\|x\|^2,
		\]
		which implies
		\begin{align*}
		2x\mu(x) + (p_0 -1)\|\sg(x)\|^2 &\leq 2x\lf(\mu(x)-\mu(0)\rt) +2x\mu(0) + (p_0 -1)\|\sg(x)-\sg(0)\|^2 +(p_0 -1)\|\sg(0)\|^2 \\
		&\leq L\|x\|^2 + 2\|x\|\|\mu(0)\| +(p_0 -1)\|\sg(0)\|^2.
		\end{align*}
       Then, applying Remark \ref{remark1} yields that,
       	\[
       \|\mu(0)\|\le K\bigl(1+\|0\|^{\frac{\chi}{2}+1}\bigr)\le K,
       \quad\quad
       \|\sg(0)\|^2 \le  K\bigl(1+\|0\|^{\frac{\chi}{2}+1}\bigr)^2 \le K.
		\]
		Combining the two estimates gives that,
		\[
		2x\mu(x) + (p_0 -1)\|\sg(x)\|^2
		\le L\|x\|^2 + 2K\|x\|+K(p_0 -1)\le K(1+\|x\|^2),
		\]
        for a constant $K>0$ depending on $L, \|\xi(0)\|$, and $p_0$.
		
\end{proof}

\begin{proof}[\textbf{Proof of Remark \ref{remark2}}]
By the definition of the tamed coefficients in \eqref{tameddef}, we have, for all $x\in\R^d$ and $\dl \in \Th$, that
\[
\|\xi^{\dl}(x)\|^2	= \frac{\|\xi(x)\|^2}{1+|\dl|^{\frac{1}{2}}\|x\|^{\chi}}
\leq \|\xi(x)\|^2.
\]
Moreover, by Remark \ref{remark1}, we obtain that
	\[
	\|\xi^{\dl}(x)\|^2 \le K \frac{\bigl(1+\|x\|^{\chi+2}\bigr)}{\bigl(1+|\dl|^{\frac{1}{2}}\|x\|^{\chi}\bigr)}
	\leq K |\dl|^{-\frac{1}{2}} \frac{\bigl(1+\|x\|^{\chi}\bigr)\bigl(1+\|x\|^{2}\bigr)}{1+\|x\|^{\chi}}
	\leq  K |\dl|^{-\frac{1}{2}} \bigl(1+\|x\|^{2}\bigr),
\]
for a constant $K>0$ depending on $L$ and $\|\xi(0)\|$.
\end{proof}

\begin{proof}[\textbf{Proof of Remark \ref{remark3}}]
By the definition of the tamed coefficients in \eqref{tameddef} and Remark \ref{remark1}, we have, for all $x\in\R^d$ and $\dl \in \Th$, that
		\begin{align*}
	\|\xi(x)-\xi^\dl(x)\|
		= \lf\|\xi(x)-\frac{\|\xi(x)\|}{\bigl(1+|\dl|^{\frac{1}{2}}\|x\|^{\chi}\bigr)^{\frac{1}{2}}} \rt\|
		\leq \frac{\|\xi(x)\|\bigl(1+|\dl|^{\frac{1}{2}}\|x\|^{\chi}\bigr) -\|\xi(x)\|}{\bigl(1+|\dl|^{\frac{1}{2}}\|x\|^{\chi}\bigr)^{\frac{1}{2}}}
	\leq \frac{|\dl|^{\frac{1}{2}}\|x\|^{\chi}\|\xi(x)\|}{\bigl(1+|\dl|^{\frac{1}{2}}\|x\|^{\chi}\bigr)^{\frac{1}{2}}}.
		\end{align*}
Further applying Remark \ref{remark1} yields that
		\begin{align*}
		\|\xi(x)-\xi^\dl(x)\|
		&\leq K |\dl|^{\frac{1}{2}}\|x\|^{\chi}
		\frac{1+\|x\|^{\frac{\chi}{2}+1} }{(1+|\dl|^{\frac{1}{2}}\|x\|^\chi)^{\frac{1}{2}}}\\
		&\leq K |\dl|^{\frac{1}{2}}\|x\|^{\chi}
		 \bigl(1+\|x\|^{\frac{\chi}{2}+1}\bigr)  \\
		&\leq K|\dl|^{\frac{1}{2}}\big(1+\|x\|^{\frac{3}{2}\chi+1}\big),
	\end{align*}
for a constant $K>0$ depending on $L$ and $\|\xi(0)\|$.
\end{proof}

\begin{proof}[\textbf{Proof of Remark \ref{remark4}}]
Let $\dl \in \Th$ and fix $x,y \in \R^d$.
Define $g^{x,y,\dl}(t) = \xi^{\dl}(tx+(1-t)y)$ for $t\in[0,1]$.
Then,
\[
\frac{d}{dt} g^{x,y,\dl}(t)= (D\xi^{\dl})(tx+(1-t)y)(x-y),
\]
where $D\xi^{\dl}$ denotes the Fr\'echet derivative of $\xi^\dl$. Since $\chi \geq 1$, we have, for all $x, v\in\R^d$, that
\begin{equation}\label{nablexi}
	(D\xi^\dl(x))v
	=\frac{(D\xi(x))v}{\bigl(1+|\dl|^{\frac12}\|x\|^\chi\bigr)^{\frac12}}
	-\frac{\chi}{2}\,|\dl|^{\frac12}\,
	\frac{\|x\|^{\chi-2}(x\cdot v)\,\xi(x)}{\bigl(1+|\dl|^{\frac12}\|x\|^\chi\bigr)^{\frac32}}.
\end{equation}
Moreover, by Assumption~\ref{A2}, for any $x, v\in\R^d$, and $h\in\R$, we have that
\[
\|\xi(x+h v)-\xi(x)\|
\le  L|h|\,\|v\|\Bigl(1+\|x\|^{\frac{\chi}{2}}+\|x+h v\|^{\frac{\chi}{2}}\Bigr).
\]
Dividing by $|h|$ and letting $h\to 0$ yields
\[
\|(D\xi(x))\,v\|
=\left\|\lim_{h\to 0}\frac{\xi(x+h v)-\xi(x)}{h}\right\|
\le K\|v\|\Bigl(1+\|x\|^{\frac{\chi}{2}}\Bigr).
\]
It is followed by
\[
\|D\xi(x)\|_{\mathrm{op}}:= \sup_{v\in \R^d: \|v\|=1}\|(D\xi(x))\,v\|  \leq  K\Bigl(1+\|x\|^{\frac{\chi}{2}}\Bigr),
\]
where $\|\cdot\|_{\mathrm{op}}$ denotes the operator norm. This, \eqref{nablexi}, and Remark \ref{remark1} yield, for all $x\in \R^d$, that
\begin{align*}
 \|(D \xi^\dl(x) )v\|& \leq \frac{\|D \xi(x)v\| }{\bigl(1+|\dl|^{\frac12}\|x\|^\chi\bigr)^{\frac12}} +  \lf\|\frac{\frac{\chi}{2}  |\dl|^{\frac12}\|x\|^{\chi-2}(x \cdot  v)  \xi(x)}{\bigl(1+|\dl|^{\frac12}\|x\|^\chi\bigr)^{\frac32}} \rt\| \\
 &\leq \frac{\|D \xi(x)\|_{\mathrm{op}}\|v\|}{\bigl(1+|\dl|^{\frac12}\|x\|^\chi\bigr)^{\frac12}} + \frac{\frac{\chi}{2}  |\dl|^{\frac12} (1+\|x\|^{\frac{\chi}{2}+1}) \|x\|^{\chi-1} \|v\| }{\bigl(1+|\dl|^{\frac12}\|x\|^\chi\bigr)^{\frac32}}\\
&\leq K\|v\|\lf(\frac{1+\|x\|^{\frac{\chi}{2}}}{|\dl|^{\frac14} \bigl(1+\|x\|^\frac{\chi}{2}\bigr)} + \frac{\frac{\chi}{2}  |\dl|^{\frac12} \bigl(1+\|x\|^{\frac{3}{2}\chi}\bigr)  }{|\dl|^{\frac34}\bigl(1+\|x\|^{\frac{3}{2}\chi}\bigr)}\rt)\\
&\leq K\|v\||\dl|^{-\frac{1}{4}},
\end{align*}
which implies
\[
\|D \xi^\dl(x) \|_{\mathrm{op}}:= \sup_{v\in \R^d: \|v\|=1}\|(D \xi^\dl(x))v \| \leq K|\dl|^{-\frac{1}{4}}. 
\]
Thus, for all $x,y \in \R^d$ and $\dl \in \Th$, we have that
\begin{align*}
\|\xi^\dl(x)-\xi^\dl(y)\| &\leq \int_0^1 \lf\|D\xi^{\dl}(tx+(1-t)y) (x-y)\rt\| \rd t    \\
& \leq \int_0^1 \|D\xi^{\dl}(tx+(1-t)y)\|_{\mathrm{op}} \|x-y\| \rd t  \\
& \leq  K|\dl|^{-\frac{1}{4}}\|x-y\|.
\end{align*}
\end{proof}

\section{Proof of results in Section \ref{Numerical Simulations}}\label{app:Numerical Simulations}
\begin{proof}[\textbf{Proof of Proposition \ref{prop1}}]
Recall the definition of one-dimensional SDE \eqref{SDEexample}:
\[
\rd X_t = (X_{t-}-X_{t-}^3)\rd t + X_{t-}^2 \rd W_t + X_{t-} \int_{\R} z \tl{\pi}(\rd z,\rd t),    \quad \forall t\in[0,T]
\]
with initial value $X_0 >0$.
\begin{enumerate}[label=(\alph*)]
	\item
	To verify Assumption \ref{A1}, we calculate
	\begin{align*}
		&2(x-y)(\mu(x)-\mu(y)) + (p_0-1)(\sg(x)-\sg(y))^2\\
		&=2(x-y) ((x-y) - (x^3-y^3))+ (p_0-1)( x^2-y^2)^2\\
		&=2(x-y)^2\big(1-(x^2+xy+y^2)\big) +(p_0-1) (x-y)^2(x+y)^2\\
		&\leq 2(x-y)^2+\big( \tfrac43(p_0-1)-  2 \big)(x^2+xy+y^2)(x-y)^2.
	\end{align*}
	Thus, Assumption \ref{A1} is satisfied	for $p_0 \leq \tfrac52$ and $L \geq 2$.
	\item
	For any $x,y\in\R$,
\[	
		|\mu(x)-\mu(y)|
		=|(x-y)-(x^3-y^3)|
		\le|x-y|+|x-y|\,(x^2+|xy|+y^2),
\]
	which implies that, by $|xy|\le \frac{x^2+y^2}{2}$,
	\[
	|\mu(x)-\mu(y)| \le \frac{3}{2}\,|x-y|\bigl(1+|x|^{2}+|y|^{2}\bigr).
	\]
	Similarly,
	$|\sigma(x)-\sigma(y)|=|x^2-y^2|=|x-y||x+y|
	\le |x-y|\bigl(1+|x|+|y|\bigr).$\\
	Let \(\chi=4\), then Assumption \ref{A2} is satisfied for $L\geq \frac{3}{2}$.
	\item
	For Assumptions \ref{A3} and \ref{A4}, we have
	\begin{align*}
		\int_{\R} |xz|^\rho\nu(dz) &= |x|^\rho \int_{\R} |z|^\rho\nu(dz),\\
		\int_{\R} |xz-yz|^\rho\nu(dz) &= |x-y|^\rho \int_{\R} |z|^\rho\nu(dz).
	\end{align*}
	Define $m :=\sup_{\rho\in[1,p_0]}\int_{\mathbb R}|z|^\rho\,\nu(dz)<\infty$, which is true for Gaussian jump sizes. Thus, Assumptions \ref{A3} and \ref{A4} are satisfied if we take $L_d=N_d =m.$
	\item
	To verify Assumption \ref{A5}, we observe that \eqref{A51} is satisfied for $\nu(\R)= \lambda < \infty$. Let $\mathcal{Z}_d=\max\{m, \lambda\}$, then, for $q \in (0,p]$,
	\begin{align*}
		\int_{|z|\leq \ep} |xz-yz|^p \nu(dz) &= \int_{|z|\leq \ep} |x-y|^p|z|^{p-q}|z|^{q} \nu(dz)\\
		& \leq  \ep^{q}|x-y|^p \int_{\R} |z|^{p-q} \nu(dz)
		  \leq  \ep^{q}\mathcal{Z}_d  |x-y|^p.
	\end{align*}
	Thus, \eqref{A52} is satisfied and similar argument applied for \eqref{A53}.
\end{enumerate}
\end{proof}

\begin{proof}[\textbf{Proof of Proposition \ref{prop2}}]
	Recall the definition of one-dimensional SDE \eqref{SDEexample2}:
\[
\rd X_t = a X_t(b-|X_t|) \rd t + c |X_t|^{3/2} \rd W_t + X_{t}\int_{\mathbb R} z\,\tilde\pi(\rd t,\rd z), \quad \forall t\in[0,T]
\]
with parameters $a,b,c\geq0$ and initial value $X_0 = x_0 >0$.
\begin{enumerate}[label=(\alph*)]
	\item
	To verify Assumption \ref{A1}, we calculate
	\begin{align}\label{23}
		&(x-y)(\mu(x)-\mu(y)) \leq 2ab|x-y|^2-2a(|x|+|y|)(|x|-|y|)^2,\nonumber \\
		&(\sg(x)-\sg(y))^2 
		= c^2(|x|^{\frac{1}{2}}-|y|^{\frac{1}{2}})^2(|x|^2+2|x|^{\frac{3}{2}}|y|^{\frac{1}{2}}+3|x||y|+2|x|^{\frac{1}{2}}|y|^{\frac{3}{2}}+|y|^2)\nonumber\\
		&\hspace{24.5mm}\leq \frac{3c^2}{2}(|x|^{\frac{1}{2}}-|y|^{\frac{1}{2}})^2(|x|^2+2|x|^{\frac{3}{2}}|y|^{\frac{1}{2}}+2|x||y|+2|x|^{\frac{1}{2}}|y|^{\frac{3}{2}}+|y|^2)\nonumber\\
		&\hspace{24.5mm}\leq \frac{3c^2}{2}(|x|^{\frac{1}{2}}-|y|^{\frac{1}{2}})^2(|x|^{\frac{1}{2}}+|y|^{\frac{1}{2}})^2(|x|+|y|) \nonumber \\
		&\hspace{24.5mm}\leq  \frac{3c^2}{2}(|x|+|y|)(|x|-|y|)^2.
	\end{align}
Consequently,
	\begin{align*}
		&2(x-y)(\mu(x)-\mu(y)) + (p_0-1)(\sg(x)-\sg(y))^2\\
		&\leq 2ab|x-y|^2 + (\frac{3c^2}{2}(p_0-1)-2a)(|x|+|y|)(|x|-|y|)^2.
		\end{align*}
Thus, Assumption \ref{A1} is satisfied for $p_0 = \frac{4a}{3c^2}+1$ and $L \geq 2ab$.
	\item
   Fix $x,y\in\R$, define
   \[
   g(t):=\mu(tx-(1-t)y), \qquad t\in[0,1].
   \]
   We note that $g'(t)= \mu'(tx-(1-t)y)(x-y)$. Then, we have
   \[
   g'(t)=a(x-y)\bigl(b-2|tx+(1-t)y|\bigr).
   \]
   Therefore, we obtain that
   \begin{align*}
   	|\mu(y)-\mu(x)|
   	 \le \int_0^1 |g'(t)|\rd t &\leq a|x-y|\left(b+2\int_0^1 |tx+(1-t)y|\rd t\right) \\
   	& \le a(b+1)|x-y|\bigl(1+|x|+|y|\bigr).
   \end{align*}
Moreover, by \eqref{23}, we have that
\begin{align*}
|\sigma(x)-\sigma(y)| \le  \sqrt{\frac{3}{2}}c(|x|+|y|)^{\frac{1}{2}}\big||x|-|y|\big| \leq \sqrt{\frac{3}{2}}c|x-y|(|x|^{\frac{1}{2}}+|y|^{\frac{1}{2}}).
\end{align*}
Thus, choosing \(\chi=2\) makes both bounds controlled and Assumption \ref{A2} is satisfied for \(L\geq \max\{a(b+1),\sqrt{\tfrac32} c\}\).
	\item
Finally, note that SDE \eqref{SDEexample2} has the same jump coefficient as SDE \eqref{SDEexample}. Therefore, Assumptions~\ref{A3}--\ref{A5} can be verified by the same arguments as in the proof of Proposition \ref{prop1}.
\end{enumerate}	
\end{proof}

\section{Proof of results in Section \ref{Proof of the Main Results}}\label{app:Proof of the Main Results}
In view of Remark~\ref{newscheme}, we use scheme \eqref{Tamedscheme2} in place of scheme \eqref{Tamedscheme} throughout Appendix~\ref{app:Proof of the Main Results}.

\subsection{Proof of results in Section \ref{section5.1}}\label{app:section5.1}

\begin{proof}[\textbf{Proof of Lemma \ref{exactmomentbounds}}]
Fix $s\in [0,T]$ and $x\in \R^d$. Then, we apply It\^o's formula (see, e.g., \cite[Theorem 2.3]{Itosformula}) to obtain, for any $t\in[s,T]$, that
	\begin{align*}
		\|X^{s,x,\iota}_t\|^{p_0} &=\|x\|^{p_0}+p_0\int^{t}_s\|X^{s,x,\iota}_{r-}\|^{p_0-2} X^{s,x,\iota}_{r-} \mu(X^{s,x,\iota}_{r-}) \rd r + p_0\int^{t}_s\|X^{s,x,\iota}_{r-}\|^{p_0-2} X^{s,x,\iota}_{r-} \sg(X^{s,x,\iota}_{r-})\rd W_r\\
		&\quad+\frac{p_0(p_0-2)}{2} \int^{t}_s\|X^{s,x,\iota}_{r-}\|^{p_0-4} \|\sg(X^{s,x,\iota}_{r-})^*X^{s,x,\iota} \|^{2}\rd r+\frac{p_0}{2}\int^{t}_s\|X^{s,x,\iota}_{r-}\|^{p_0-2}\|\sg(X^{s,x,\iota}_{r-})\|^2\rd r\\
		&\quad+p_0\int^t_s\int_{\R^d}\|X^{s,x,\iota}_{r-}\|^{p_0-2} X^{s,x,\iota}_{r-}\gm(X^{s,x,\iota}_{r-},z)\tl{\pi}(\rd z,\rd r)\\
		&\hspace{-8mm}+\int_s^t\int_{\R^d}\lf\{\|X^{s,x,\iota}_{r-}+\gm(X^{s,x,\iota}_{r-},z)\|^{p_0}-\|X^{s,x,\iota}_{r-}\|^{p_0}-p_0\|X^{s,x,\iota}_{r-}\|^{p_0-2}X^{s,x,\iota}_{r-}\gm(X^{s,x,\iota}_{r-},z)\rt\}\pi(\rd z,\rd r),
	\end{align*}
	almost surely\footnote{We adopt the convention that $0/0:=0$.}. Then, we define a stopping time $\tau^{s,x}_R:= \inf\{t\geq s: \|X^{s,x,\iota}_t\|\geq R\}\wedge T$ with $R \in \N$. Thus, on taking expectation and using Cauchy-Schwarz inequality, we obtain, for any $t\in [s,T],$ that
	\begin{align*}
		\E\|X^{s,x,\iota}_{t\wedge\tau^{s,x}_R}\|^{p_0} &\leq  \|x\|^{p_0}+\frac{p_0}{2}\E\int^{t\wedge\tau^{s,x}_{R}}_s\|X^{s,x,\iota}_{r}\|^{p_0-2}\{2 X^{s,x,\iota}_{r}\mu(X^{s,x,\iota}_{r})+(p_0-1)\|\sg(X^{s,x,\iota}_{r})\|^2\}\rd r\\
		&\hspace{-20mm}+\E\int^{t\wedge\tau^{s,x}_{R}}_s\int_{\R^d} \lf\{\|X^{s,x,\iota}_{r}+\gm(X^{s,x,\iota}_{r},z)\|^{p_0}-\|X^{s,x,\iota}_{r}\|^{p_0}-p_0\|X^{s,x,\iota}_{r}\|^{p_0-2}X^{s,x,\iota}_{r}\gm(X^{s,x,\iota}_{r},z)\rt\}\nu(\rd z)\rd r.
	\end{align*}
	Moreover, by Lemma \ref{Formula for the remainder}, we have, for all $t\in [s,T]$, that
	\begin{align*}
		\E\|X^{s,x,\iota}_{t\wedge\tau^{s,x}_R}\|^{p_0} &\leq \|x\|^{p_0}+\frac{p_0}{2}\E\int^{t\wedge\tau^{s,x}_{R}}_s\|X^{s,x,\iota}_{r}\|^{p_0-2}\{2 X^{s,x,\iota}_{r}\mu(X^{s,x,\iota}_{r})+(p_0-1)\|\sg(X^{s,x,\iota}_{r})\|^2\}\rd r\\
		&\quad+K\E\int^{t\wedge\tau^{s,x}_{R}}_s\int_{\R^d} \lf\{\|X^{s,x,\iota}_{r}\|^{p_0-2}\|\gm(X^{s,x,\iota}_{r},z)\|^2+\|\gm(X^{s,x,\iota}_{r},z)\|^{p_0}\rt\}\nu(\rd z)\rd r\\
		&\leq \|x\|^{p_0}+\frac{p_0}{2}\E\int^{t\wedge\tau^{s,x}_{R}}_s\|X^{s,x,\iota}_{r}\|^{p_0-2}\{2 X^{s,x,\iota}_{r}\mu(X^{s,x,\iota}_{r})+(p_0-1)\|\sg(X^{s,x,\iota}_{r})\|^2\}\rd r\\
		&\quad+K\E\int^{t\wedge\tau^{s,x}_{R}}_s\int_{\R^d} \big\{\|X^{s,x,\iota}_{r}\|^{p_0-2}\|\gm(X^{s,x,\iota}_{r},z)-\gm(0,z)\|^2\\
		&\hspace{40mm}+\|\gm(X^{s,x,\iota}_{r},z)-\gm(0,z)\|^{p_0}\big\}\nu(\rd z)\rd r\\
		&\quad+K\E\int^{t\wedge\tau^{s,x}_{R}}_s\int_{\R^d} \lf\{\|X^{s,x,\iota}_{r}\|^{p_0-2}\|\gm(0,z)\|^2+\|\gm(0,z)\|^{p_0}\rt\}\nu(\rd z)\rd r.
	\end{align*}
	On using Assumptions \ref{A3}, \ref{A4}, Remark \ref{exist}, and Young's inequality, we obtain, for any $t\in [s,T],$ that
	\begin{align*}
		\E\|X^{s,x,\iota}_{t\wedge\tau^{s,x}_R}\|^{p_0} &\leq \|x\|^{p_0}+K(1+N_d^{\frac{p_0}{2}})(t-s)+K(1+L_d^{\frac{p_0}{2}})\E\int^{t\wedge\tau^{s,x}_{R}}_s\|X^{s,x,\iota}_{r}\|^{p_0}\rd r.
	\end{align*}
	This implies
	\begin{align*}
		\sup_{t\in[s,T]}\E\|X^{s,x,\iota}_{t\wedge\tau^{s,x}_R}\|^{p_0} &\leq \|x\|^{p_0}+K(1+N_d^{\frac{p_0}{2}})(T-s)+K(1+L_d^{\frac{p_0}{2}})\int^{T}_s \sup_{u\in[s,r]}\E\|X^{s,x,\iota}_{u\wedge\tau_R^{s,x}}\|^{p_0}\rd r,
	\end{align*}
	which, by applying Gr\"onwall's lemma, yields
	\[
	\sup_{t\in[s,T]}\E\|X^{s,x,\iota}_{t\wedge\tau^{s,x}_R}\|^{p_0}  \leq K\lf( 1+\|x\|^{p_0}+N_d^{\frac{p_0}{2}}\rt)e^{K(1+L_d^{\frac{p_0}{2}})}.
	\]
	Finally, the application of Fatou's lemma completes the proof.
\end{proof}

\begin{lemma}\label{corollary1}
	Let Assumptions \ref{A1}--\ref{A5} hold, then there exists $K>0 $ such that, 
		\begin{enumerate}
		\item\label{lemma2} for all $s\in [0,T]$, $t\in [s,T]$, $x\in \R^d$, $\dl \in \Th$, and $\rho \in [1,p_0]$,
		\begin{equation*}
			\E(\|X^{s,x,\dl}_t - X^{s,x,\dl}_{\max\{s,\dl(t)\}} \|^\rho|\mathcal{F}_{\max\{s,\dl(t)\}}) \leq K(|\dl|^{\frac{\rho}{4}} +|\dl| ) \big(1+N_d^{\frac{\rho}{2}}+(1+L_d^{\frac{\rho}{2}})\|X^{s,x,\dl}_{\max\{s,\dl(t)\}}\|^{\rho} \big),
		\end{equation*}
		\item  for all $s\in [0,T]$, $t\in [s,T]$, $x\in \R^d$, $\dl \in \Th$, $\ep\in(0,1)$, and $\rho \in [1,p_0]$,
		\[
		\E(\|X^{s,x,\dl,\ep}_t - X^{s,x,\dl,\ep}_{\max\{s,\dl(t)\}} \|^\rho|\mathcal{F}_{\max\{s,\dl(t)\}}) \leq K(|\dl|^{\frac{\rho}{4}} +|\dl| ) \big(1+N_d^{\frac{\rho}{2}}+(1+L_d^{\frac{\rho}{2}})\|X^{s,x,\dl,\ep}_{\max\{s,\dl(t)\}}\|^{\rho} \big),
		\]
		\item \label{lemma3} for all $ s\in [0,T]$, $t\in [s,T]$, $x\in \R^d$, $\dl \in \Th$, $\ep\in(0,1)$, $\mathcal{M} \in \N$, and $\rho \in [1,p_0]$,
		\begin{align*}
			&\E(\|X^{s,x,\dl,\ep,\mathcal{M}}_t - X^{s,x,\dl,\ep,\mathcal{M}}_{\max\{s,\dl(t)\}} \|^\rho|\mathcal{F}_{\max\{s,\dl(t)\}}) \leq K(|\dl|^{\frac{\rho}{4}} +|\dl| ) \big(1+N_d^{\frac{\rho}{2}}+(1+L_d^{\frac{\rho}{2}})\|X^{s,x,\dl,\ep,\mathcal{M}}_{\max\{s,\dl(t)\}}\|^{\rho} \big)\\
			&\quad+  |\dl|^{\rho} \bigg\|\int_{A_\ep} \gm(X^{s,x,\dl,\ep,\mathcal{M}}_{\max\{s,\dl(t)\}},z)\nu(\rd z) - \lf(\frac{\nu(A_\ep)}{\mathcal{M}}\sum_{i=1}^{\mathcal{M}} \gm\lf(X_{\max\{s,\dl(t)\}}^{s,x,\dl,\ep,\mathcal{M}},V^{\dl,\ep,\mathcal{M}}_{i,\max\{s,\dl(t)\}}\rt)\rt)\bigg\|^\rho.
		\end{align*}
	\end{enumerate}
\begin{proof}[Proof of (i):]
		Fix $s\in [0,T]$, $x\in \R^d$, and $\dl \in \Th$ throughout this proof. By the definition of scheme \eqref{tamedscheme1},  we obtain, for all $t\in[s,T]$, that
		\begin{equation}\label{L2A13}
			\E(\|X^{s,x,\dl}_t - X^{s,x,\dl}_{\max\{s,\dl(t)\}} \|^\rho|\mathcal{F}_{\max\{s,\dl(t)\}}) \leq 3^{\rho-1}\sum_{k=1}^3 A_k(t), \tag{$\star$}
		\end{equation}
		where
		\begin{align*}
			A_1(t):=&\E \lf(\bigg\|\int^{t}_{\max\{s,\dl(t)\}} \mu^{\dl}(X^{s,x,\dl}_{\max\{s,\dl(r)\}})\rd r\bigg\|^\rho \bigg|\mathcal{F}_{\max\{s,\dl(t)\}} \rt),\\
			A_2(t):=&\E\lf(\bigg\|\int^{t}_{\max\{s,\dl(t)\}} \sg^{\dl}(X^{s,x,\dl}_{\max\{s,\dl(r)\}})\rd W_r \bigg\|^\rho \bigg|\mathcal{F}_{\max\{s,\dl(t)\}}\rt),\\
			A_3(t):=&\E\lf(\bigg\|\int^{t}_{\max\{s,\dl(t)\}} \int_{\R^d} \gm(X^{s,x,\dl}_{\max\{s,\dl(r)\}},z) \tl\pi(\rd z,\rd r)\bigg\|^\rho \bigg|\mathcal{F}_{\max\{s,\dl(t)\}}\rt).
		\end{align*}
		On the application of H\"older's inequality and Remark \ref{remark2} gives, for all $t\in[s,T]$, that
		\begin{equation}\label{L2A_1}
			\begin{aligned}
				A_1(t) &\leq |\dl|^{\rho-1}\E \lf(\int^{t}_{\max\{s,\dl(t)\}} \|\mu^{\dl}(X^{s,x,\dl}_{\max\{s,\dl(r)\}})\|^{\rho} \rd r \bigg|\mathcal{F}_{\max\{s,\dl(t)\}} \rt)\\
				& \leq K|\dl|^{\frac{3\rho}{4}}  \lf( 1+ \|X^{s,x,\dl}_{\max\{s,\dl(t)\}}\|^{\rho}     \rt).
			\end{aligned}
		\end{equation}
		Moreover, by using \cite[Theorem 7.3]{xuerongmao} and Remark \ref{remark2}, we obtain, for all $t\in[s,T]$, that
		\begin{equation}\label{L2A_2}
			\begin{aligned}
				A_2(t) &\leq K \E  \lf(\lf(\int^{t}_{\max\{s,\dl(t)\}} \|\sg^{\dl}(X^{s,x,\dl}_{\max\{s,\dl(r)\}})\|^{2} \rd r\rt)^{\frac{\rho}{2}}  \bigg|\mathcal{F}_{\max\{s,\dl(t)\}} \rt)\\
				&\leq  K  |\dl|^{\frac{\rho}{4}}\lf(1+\|X^{s,x,\dl}_{\max\{s,\dl(t)\}}\|^{\rho}\rt).
			\end{aligned}
		\end{equation}
		We now estimate $A_3(t)$. For $\rho \in [1,2]$, Jensen's inequality, Burkholder-Davis-Gundy inequality (see, e.g., \cite[Theorem VII.92]{ProbabilityB}), Assumptions \ref{A3} and \ref{A4} yield
		\begin{align}\label{L2A_3(t)_1}
			A_3(t) \leq & K\lf(\E\lf[\bigg\|\int^{t}_{\max\{s,\dl(t)\}} \int_{\R^d} \gm(X^{s,x,\dl}_{\max\{s,\dl(r)\}},z) \tl\pi(\rd z,\rd r)\bigg\|^2 \bigg|\mathcal{F}_{\max\{s,\dl(t)\}}\rt]\rt)^{\frac{\rho}{2}}\nonumber\\
			\leq&K\lf(\E \lf[ \int^{t}_{\max\{s,\dl(t)\}} \int_{\R^d}  \|\gm(X^{s,x,\dl}_{\max\{s,\dl(r)\}},z)-\gm(0,z)\|^2  \nu(\rd z)\rd r \big|\mathcal{F}_{\max\{s,\dl(t)\}}\rt]\rt)^{\frac{\rho}{2}} \nonumber\\
			\quad&+K\lf(\E \lf[ \int^{t}_{\max\{s,\dl(t)\}} \int_{\R^d}  \|\gm(0,z)\|^2  \nu(\rd z)\rd r \big|\mathcal{F}_{\max\{s,\dl(t)\}}\rt]\rt)^{\frac{\rho}{2}} \nonumber\\
			\leq& K|\dl|^{\frac{\rho}{2}}\big(L_d^{\frac{\rho}{2}}\|X^{s,x,\dl}_{\max\{s,\dl(t)\}}\|^{\rho}+ N_d^{\frac{\rho}{2}}\big).
		\end{align}
		For $\rho \in (2,p_0]$, we apply \cite[Lemma 1]{Mikulevicius}, together with Assumptions \ref{A3} and \ref{A4} to obtain, for all $t\in[s,T]$, that
		\begin{align}\label{L2A_3(t)_2}
			A_3(t) 
			&\leq K \E \lf[ \int^{t}_{\max\{s,\dl(t)\}} \int_{\R^d} \|\gm(X^{s,x,\dl}_{\max\{s,\dl(r)\}},z)\|^\rho\nu(\rd z)\rd r \big|\mathcal{F}_{\max\{s,\dl(t)\}}\rt]\nonumber\\
			&\quad +K\E\lf[\lf(\int^{t}_{\max\{s,\dl(t)\}} \int_{\R^d} \|\gm(X^{s,x,\dl}_{\max\{s,\dl(r)\}},z)\|^2 \nu(\rd z)\rd r\rt)^\frac{\rho}{2} \big|\mathcal{F}_{\max\{s,\dl(t)\}} \rt]\nonumber\\
			&\leq K|\dl|\big(L_d
			\|X^{s,x,\dl}_{\max\{s,\dl(t)\}}\|^\rho+N_d\big)+K|\dl|^{\frac{\rho}{2}}\big(L_d^{\frac{\rho}{2}}\|X^{s,x,\dl}_{\max\{s,\dl(t)\}}\|^\rho+N_d^{\frac{\rho}{2}}\big).
		\end{align}
		Thus, substituting \eqref{L2A_1}, \eqref{L2A_2}, and \eqref{L2A_3(t)_1} back into \eqref{L2A13} gives, for all $\rho \in [1,2]$ and $t\in[s,T]$, that
		\begin{equation*}
			\E(\|X^{s,x,\dl}_t - X^{s,x,\dl}_{\max\{s,\dl(t)\}} \|^\rho|\mathcal{F}_{\max\{s,\dl(t)\}}) 
			\leq K|\dl|^{\frac{\rho}{4}}\big(1+N_d^{\frac{\rho}{2}}+(1+L_d^{\frac{\rho}{2}})\|X^{s,x,\dl}_{\max\{s,\dl(t)\}}\|^{\rho} \big).
		\end{equation*}
		Similarly, combining \eqref{L2A_1}, \eqref{L2A_2}, and \eqref{L2A_3(t)_2}, we obtain, for all $\rho \in (2,p_0]$ and $t\in[s,T]$, that
		\begin{equation*}
			\E(\|X^{s,x,\dl}_t - X^{s,x,\dl}_{\max\{s,\dl(t)\}} \|^\rho|\mathcal{F}_{\max\{s,\dl(t)\}}) 
			\leq K(|\dl|^{\frac{\rho}{4}} +|\dl| ) \big(1+N_d^{\frac{\rho}{2}}+(1+L_d^{\frac{\rho}{2}})\|X^{s,x,\dl}_{\max\{s,\dl(t)\}}\|^{\rho} \big).
		\end{equation*}
		This completes the proof.
	\end{proof}
		\begin{proof}[Proof of (ii):]
Since, for any non-negative integrable function $g$,
$
\int_{A_\ep} g(z)\,\nu(\rd z)\le \int_{\R^d} g(z)\,\nu(\rd z),
$
the proof follows by the same arguments as in the proof of Lemma~\ref{lemma2}.
	\end{proof}
		\begin{proof}[Proof of (iii):]
Fix $ s\in [0,T]$, $x\in \R^d$, $\dl \in \Th$, $\ep \in (0,1)$, and $\mathcal{M} \in \N$ throughout this proof. By the definition of scheme $\bigl(X^{s,x,\dl,\ep,\mathcal{M}}_t \bigr)_{t\in[s,T]}$ in \eqref{Tamedscheme2}, we have, for all $t\in[s,T]$, that
\begin{equation}\label{l3B1-4}
	\E(\|X^{s,x,\dl,\ep,\mathcal{M}}_t - X^{s,x,\dl,\ep,\mathcal{M}}_{\max\{s,\dl(t)\}} \|^\rho|\mathcal{F}_{\max\{s,\dl(t)\}}) \leq 4^{\rho-1}\sum_{k=1}^4 B_k(t), \tag{$\dagger$}
\end{equation}
where
\begin{align*}
	B_1(t):=&\E \lf(\bigg\|\int^{t}_{\max\{s,\dl(t)\}} \mu^{\dl}(X^{s,x,\dl,\ep,\mathcal{M}}_{\max\{s,\dl(r)\}})\rd r\bigg\|^\rho \bigg|\mathcal{F}_{\max\{s,\dl(t)\}} \rt),\\
	B_2(t):=&\E\lf(\bigg\|\int^{t}_{\max\{s,\dl(t)\}} \sg^{\dl}(X^{s,x,\dl,\ep,\mathcal{M}}_{\max\{s,\dl(r)\}})\rd W_r \bigg\|^\rho \bigg|\mathcal{F}_{\max\{s,\dl(t)\}}\rt),\\
	B_3(t):=&\E\lf(\bigg\|\int^{t}_{\max\{s,\dl(t)\}} \int_{A_\ep} \gm(X^{s,x,\dl,\ep,\mathcal{M}}_{\max\{s,\dl(r)\}},z) \tl\pi(\rd z,\rd r)\bigg\|^\rho \bigg|\mathcal{F}_{\max\{s,\dl(t)\}}\rt),\\
	B_4(t):=&\E\bigg(\bigg\|\int^{t}_{\max\{s,\dl(t)\}} \int_{A_\ep} \gm(X^{s,x,\dl,\ep,\mathcal{M}}_{\max\{s,\dl(r)\}},z)\nu(\rd z) \\
	&\quad - \lf\{\frac{\nu(A_\ep)}{\mathcal{M}}\sum_{i=1}^{\mathcal{M}}\gm\lf(X_{\max\{s,\dl(r)\}}^{s,x,\dl,\ep,\mathcal{M}},V^{\dl,\ep,\mathcal{M}}_{i,\max\{s,\dl(r)\}}\rt)\rt\}\rd r \bigg\|^\rho \bigg|\mathcal{F}_{\max\{s,\dl(t)\}}\bigg).\\
\end{align*}
The estimates for $B_1(t)$--$B_3(t)$ follow as in the proof of \ref{lemma2} (cf.~\eqref{L2A_1}--\eqref{L2A_3(t)_2}), yielding for all $t\in[s,T]$, that
\begin{equation}\label{L3A_13}
	\begin{aligned}
		B_1(t)  & \leq K|\dl|^{\frac{3\rho}{4}}  \big( 1+ \|X^{s,x,\dl,\ep,\mathcal{M}}_{\max\{s,\dl(t)\}}\|^{\rho}     \big),\\
		B_2(t) &\leq  K  |\dl|^{\frac{\rho}{4}}\big(1+\|X^{s,x,\dl,\ep,\mathcal{M}}_{\max\{s,\dl(t)\}}\|^{\rho}\big),\\
		B_3(t) &\leq
		\begin{cases}
			K|\dl|^{\frac{\rho}{2}}\big(L_d^{\frac{\rho}{2}}\|X^{s,x,\dl}_{\max\{s,\dl(t)\}}\|^{\rho}+ N_d^{\frac{\rho}{2}}\big),    & \forall \rho\in[1,2], \\
			K|\dl|\big(N_d+L_d
			\|X^{s,x,\dl}_{\max\{s,\dl(t)\}}\|^\rho\big)+K|\dl|^{\frac{\rho}{2}}\big(N_d^{\frac{\rho}{2}}+L_d^{\frac{\rho}{2}}\|X^{s,x,\dl}_{\max\{s,\dl(t)\}}\|^\rho \big),&\forall \rho\in(2,p_0].
		\end{cases}
	\end{aligned}
\end{equation}
For $B_4(t)$, one observes, for all $t\in[s,T]$, that
\begin{align}\label{L3A4}
	B_4(t) \leq |\dl|^{\rho} \bigg\|\int_{A_\ep} \gm(X^{s,x,\dl,\ep,\mathcal{M}}_{\max\{s,\dl(t)\}},z)\nu(\rd z) - \lf(\frac{\nu(A_\ep)}{\mathcal{M}}\sum_{i=1}^{\mathcal{M}} \gm\lf(X_{\max\{s,\dl(t)\}}^{s,x,\dl,\ep,\mathcal{M}},V^{\dl,\ep,\mathcal{M}}_{i,\max\{s,\dl(t)\}}\rt)\rt)\bigg\|^\rho.
\end{align}
Combining \eqref{L3A_13} with \eqref{L3A4} and \eqref{l3B1-4} completes the proof.
	\end{proof}
\end{lemma}

\begin{lemma}\label{L3A_4}
	Let Assumptions \ref{A1}--\ref{A5} hold, then there exists $K > 0$ such that, for all $s\in [0,T]$, $t \in [s,T]$, $x\in \R^d$, $\dl \in \Th$, $\ep \in (0,1)$, $\mathcal{M} \in \N$ with $\mathcal{M} \geq (1 \vee \ep^{-2}\mathcal{Z}_d)^{2}$, and $\rho \in [1,p_0]$,
	\begin{align*}
		& \E \bigg\|\int_{A_\ep} \gm(X^{s,x,\dl,\ep,\mathcal{M}}_{\max\{s,\dl(t)\}},z)\nu(\rd z) - \lf(\frac{\nu(A_\ep)}{\mathcal{M}}\sum_{i=1}^{\mathcal{M}} \gm\lf(X_{\max\{s,\dl(t)\}}^{s,x,\dl,\ep,\mathcal{M}},V^{\dl,\ep,\mathcal{M}}_{i,\max\{s,\dl(t)\}}\rt)\rt)\bigg\|^\rho \\
		&\leq K \big( N_d+L_d\E \|X_{\max\{s,\dl(r)\}}^{s,x,\dl,\ep,\mathcal{M}}\|^{\rho}\big).
	\end{align*}
\end{lemma}
	\begin{proof}
		Fix $ s\in [0,T]$, $x\in \R^d$, $\dl \in \Th$, $\ep \in (0,1)$, $\mathcal{M} \in \N$ with $\mathcal{M} \geq (1 \vee \ep^{-2}\mathcal{Z}_d)^{2}$, and $\rho \in [1,p_0]$. Moreover, by the definition of $V^{\dl,\ep,\mathcal{M}}_{i,t_j}$ and $\nu_\ep$, one observes for all $x\in\R^d$, $t\in[s,T]$, $i=1,2,...,\mathcal{M}$ and $j=0,1,...,n,$ that
		\begin{equation}\label{defV}
			\nu(A_\ep)\E\lf[\gm(x,V^{\dl,\ep,\mathcal{M}}_{i,t_j})\rt] = \int_{A_\ep} \gm(x,z)\nu(\rd z).
		\end{equation}
		By using \eqref{nudef}, \eqref{defV}, the independence of $V^{\dl,\ep,\mathcal{M}}_{i,t_j}$, Marcinkiewicz–Zygmund inequality, Jensen's inequality, Assumptions \ref{A3} and \ref{A4}, we infer that for all $t\in[s,T]$,
		\begin{align}\label{L3A_4(t)}
			&\E \lf[\bigg\|\int_{A_\ep}\gm\lf(X_{\max\{s,\dl(t)\}}^{s,x,\dl,\ep,\mathcal{M}}, z\rt) \nu(\rd z)- \bigg(\frac{\nu(A_\ep)}{\mathcal{M}}\sum_{i=1}^{\mathcal{M}}\gm\lf(X_{\max\{s,\dl(t)\}}^{s,x,\dl,\ep,\mathcal{M}},V^{\dl,\ep,\mathcal{M}}_{i,\max\{s,\dl(t)\}}\rt)\bigg)\bigg\|^{\rho}\rt] \nonumber\\
			&=\E \lf[\E \bigg[\bigg\|\int_{A_\ep}\gm\lf(x, z\rt) \nu(\rd z)- \bigg(\frac{\nu(A_\ep)}{\mathcal{M}}\sum_{i=1}^{\mathcal{M}} \gm\lf(x,V^{\dl,\ep,\mathcal{M}}_{i,\max\{s,\dl(t)\}}\rt)\bigg)\bigg\|^{\rho} \bigg] \Biggr\rvert_{x=X_{\max\{s,\dl(t)\}}^{s,x,\dl,\ep,\mathcal{M}}} \rt]\nonumber\\
			& =\nu(A_\ep)^{\rho} \mathcal{M}^{-\rho}\E\lf[\E \lf\|\sum_{i=1}^{\mathcal{M}} \lf(\E\lf[\gm\lf(x,V^{\dl,\ep,\mathcal{M}}_{i,t_1}\rt)\rt]-\gm\lf(x,V^{\dl,\ep,\mathcal{M}}_{i,t_1}\rt)\rt)\rt\|^{\rho} \Biggr\rvert_{x=X_{\max\{s,\dl(t)\}}^{s,x,\dl,\ep,\mathcal{M}}} \rt]\nonumber\\
			& \leq K\nu(A_\ep)^{\rho} \mathcal{M}^{-\rho}\E\lf[\E \lf[\sum_{i=1}^{\mathcal{M}}\lf\|\E\lf[\gm\lf(x,V^{\dl,\ep,\mathcal{M}}_{i,t_1}\rt)\rt]-\gm\lf(x,V^{\dl,\ep,\mathcal{M}}_{i,t_1}\rt)\rt\|^2\rt]^{\frac{\rho}{2}} \Biggr\rvert_{x=X_{\max\{s,\dl(t)\}}^{s,x,\dl,\ep,\mathcal{M}}} \rt]\nonumber\\
			& \leq K \nu(A_\ep)^{\rho} \mathcal{M}^{-\rho}\big(1+\mathcal{M}^{\frac{\rho}{2}-1}\big)\E\lf[\E \lf[\sum_{i=1}^{\mathcal{M}}\lf\|\E\lf[\gm\lf(x,V^{\dl,\ep,\mathcal{M}}_{i,t_1}\rt)\rt]-\gm\lf(x,V^{\dl,\ep,\mathcal{M}}_{i,t_1}\rt)\rt\|^{\rho}\rt] \Biggr\rvert_{x=X_{\max\{s,\dl(t)\}}^{s,x,\dl,\ep,\mathcal{M}}} \rt]\nonumber\\
			& \leq K \nu(A_\ep)^{\rho}\big(  \mathcal{M}^{-\rho+1}+ \mathcal{M}^{-\frac{\rho}{2}} \big)\E\lf[\E \lf[\lf\|\E\lf[\gm\lf(x,V^{\dl,\ep,\mathcal{M}}_{1,t_1}\rt)\rt]-\gm\lf(x,V^{\dl,\ep,\mathcal{M}}_{1,t_1}\rt)\rt\|^{\rho}\rt] \Biggr\rvert_{x=X_{\max\{s,\dl(t)\}}^{s,x,\dl,\ep,\mathcal{M}}} \rt]\nonumber\\
			& \leq K  \nu(A_\ep)^{\rho}\big(  \mathcal{M}^{-\rho+1}+ \mathcal{M}^{-\frac{\rho}{2}} \big)\E\lf[  \E\lf\|\gm\lf(x,V^{\dl,\ep,\mathcal{M}}_{1,t_1}\rt)\rt\|^{\rho} \Biggr\rvert_{x=X_{\max\{s,\dl(t)\}}^{s,x,\dl,\ep,\mathcal{M}}} \rt]\nonumber\\
			& \leq K  \nu(A_\ep)^{\rho-1}\big(  \mathcal{M}^{-\rho+1}+ \mathcal{M}^{-\frac{\rho}{2}} \big) \E\lf[\int_{A_\ep} \|\gm(X_{\max\{s,\dl(t)\}}^{s,x,\dl,\ep,\mathcal{M}},z)\|^{\rho}\nu(\rd z) \rt]\nonumber\\
			& \leq K\nu(A_\ep)^{\rho-1}\big(  \mathcal{M}^{-\rho+1}+ \mathcal{M}^{-\frac{\rho}{2}} \big) \E\lf[\int_{A_\ep} \|\gm(X_{\max\{s,\dl(t)\}}^{s,x,\dl,\ep,\mathcal{M}},z) -\gm(0,z) \|^{\rho}+\|\gm(0,z)\|^{\rho}\nu(\rd z) \rt]\nonumber\\
			&\leq K (\ep^{-2}\mathcal{Z}_d)^{\rho-1} \big(  \mathcal{M}^{-\rho+1}+ \mathcal{M}^{-\frac{\rho}{2}} \big)  \big( L_d\E\|X_{\max\{s,\dl(t)\}}^{s,x,\dl,\ep,\mathcal{M}}\|^{\rho}+ N_d\big).
		\end{align}
		Since $\mathcal{M}\geq (1 \vee \ep^{-2}\mathcal{Z}_d)^{2}$, we conclude, for all $t\in[s,T]$, that
		\begin{align*}
			&\E \lf[\bigg\|\int_{A_\ep}\gm\lf(X_{\max\{s,\dl(t)\}}^{s,x,\dl,\ep,\mathcal{M}}, z\rt) \nu(\rd z)- \bigg(\frac{\nu(A_\ep)}{\mathcal{M}}\sum_{i=1}^{\mathcal{M}}\gm\lf(X_{\max\{s,\dl(t)\}}^{s,x,\dl,\ep,\mathcal{M}},V^{\dl,\ep,\mathcal{M}}_{i,\max\{s,\dl(t)\}}\rt)\bigg)\bigg\|^{\rho}\rt]\\
			&\quad \leq K \big( N_d+L_d\E \|X_{\max\{s,\dl(t)\}}^{s,x,\dl,\ep,\mathcal{M}}\|^{\rho}\big).
		\end{align*}
	This completes the proof.
	\end{proof}

As a preparatory step for the subsequent estimate analysis, we establish uniform moment bounds for schemes \eqref{tamedscheme1}, \eqref{tamedscheme2}, and \eqref{Tamedscheme2}.

\begin{lemma}\label{schemebounds}
	Let Assumptions \ref{A1}--\ref{A4} hold and let $\dl \in \Th$ be fixed. Then, for all $ s\in [0,T]$, $x\in \R^d$, $\ep \in (0,1)$, $\mathcal{M} \in \N$ with $\mathcal{M} \geq (1 \vee \ep^{-2}\mathcal{Z}_d)^{2}$, and $\rho \geq 1$, it holds that
	\begin{equation*}
		\begin{aligned}
			\sup_{t\in[s,T]}\E\|X^{s,x,\dl}_t\|^{\rho}<  \infty, \quad \sup_{t\in[s,T]}\E\|X^{s,x,\dl,\ep}_t\|^{\rho}<  \infty, \quad
			\sup_{t\in[s,T]}\E\|X^{s,x,\dl,\ep,\mathcal{M}}_t\|^{\rho}<  \infty.
		\end{aligned}
	\end{equation*}
\begin{proof}
	By Remark \ref{remark2}, the coefficients $\mu^{\dl}$ and $\sg^\dl$ are growing linearly for fixed $\dl \in \Th$, hence standard arguement applies, see e.g., \cite{MR1214374}.
\end{proof}
\end{lemma}

\begin{proof}[\textbf{Proof of Lemma \ref{L3}}]
	Fix $s \in [0,T]$, $x \in \R^d$, and $\dl \in \Th$ throughout this proof. Then, applying It\^o's formula to $\|X_t^{s,x,\dl}\|^{p_0}$ yields, for all $t \in[s,T]$, that
	\begin{equation*}
		\begin{aligned}
			\|X^{s,x,\dl}_t\|^{p_0} &=\|x\|^{p_0}+p_0\int^{t}_s\|X^{s,x,\dl}_{r-}\|^{p_0-2} X^{s,x,\dl}_{r-} \mu^\dl(X^{s,x,\dl}_{\max\{s,\dl(r-)\}}) \rd r \\
			&\quad+ p_0\int^{t}_s\|X^{s,x,\dl}_{r-}\|^{p_0-2} X^{s,x,\dl}_{r-} \sg^\dl(X^{s,x,\dl}_{\max\{s,\dl(r-)\}})\rd W_r\\
			&\quad+\frac{p_0(p_0-2)}{2} \int^{t}_s\|X^{s,x,\dl}_{r-}\|^{p_0-4} \|\sg^\dl(X^{s,x,\dl}_{\max\{s,\dl(r-)\}})^*X^{s,x,\dl}_{r-} \|^{2}\rd r\\
			&\quad+\frac{p_0}{2}\int^{t}_s\|X^{s,x,\dl}_{r-}\|^{p_0-2}\|\sg^\dl(X^{s,x,\dl}_{\max\{s,\dl(r-)\}})\|^2\rd r\\
			&\quad+p_0\int^t_s\int_{\R^d}\|X^{s,x,\dl}_{r-}\|^{p_0-2} X^{s,x,\dl}_{r-}\gm(X^{s,x,\dl}_{\max\{s,\dl(r-)\}},z)\tl{\pi}(\rd z,\rd r)\\
			&\quad+\int_s^t\int_{\R^d}\big\{\|X^{s,x,\dl}_{r-}+\gm(X^{s,x,\dl}_{\max\{s,\dl(r-)\}},z)\|^{p_0}-\|X^{s,x,\dl}_{r-}\|^{p_0}\\
			&\quad\;\;-p_0\|X^{s,x,\dl}_{r-}\|^{p_0-2}X^{s,x,\dl}_{r-}\gm(X^{s,x,\dl}_{\max\{s,\dl(r-)\}},z)\big\}\pi(\rd z,\rd r),
		\end{aligned}
	\end{equation*}
	almost surely. By Lemmas \ref{schemebounds}, \ref{Formula for the remainder}, and Cauchy-Schwarz inequality, we have, for all $t \in [s,T]$, that
	\begin{equation}\label{L3A13}
		\E\|X^{s,x,\dl}_{t}\|^{p_0} \leq\|x\|^{p_0}+\sum_{k=1}^3 C_k(t), \tag{$\triangle$}
	\end{equation}
	where
	\begin{align*}
		C_1(t)&:=p_0\E \int ^{t }_s \|X^{s,x,\dl}_r\|^{p_0-2}(X^{s,x,\dl}_r-X^{s,x,\dl}_{\max\{s,\dl(r)\}})\mu^\dl(X^{s,x,\dl}_{\max\{s,\dl(r)\}}) \rd r,\\
		C_2(t)&:=\frac{p_0} {2}\E\int^{t }_s\|X^{s,x,\dl}_r\|^{p_0-2} \{2 X^{s,x,\dl}_{\max\{s,\dl(r)\}} \mu^\dl(X^{s,x,\dl}_{\max\{s,\dl(r)\}})+(p_0-1) \|\sg^\dl(X^{s,x,\dl}_{\max\{s,\dl(r)\}})\|^2\} \rd r, \\
		C_3(t)&:=K\E\int^{t }_s\int_{\R^d} \{\|X^{s,x,\dl}_r\|^{p_0-2}\|\gm(X^{s,x,\dl}_{\max\{s,\dl(r)\}},z)\|^2+\|\gm(X^{s,x,\dl}_{\max\{s,\dl(r)\}},z)\|^{p_0}\}\nu(\rd z)\rd r.
	\end{align*}
	By using Young's inequality, Remark \ref{remark2}, and Lemma \ref{corollary1}\,\ref{lemma2}, we obtain, for all $t\in[s,T]$, that
	\begin{align}\label{L2B_1}
		C_1(t) &\leq K|\dl|^{-\frac{1}{4}}\E\int_s^{t } \|X_r^{s,x,\dl}\|^{p_0-2}\|X^{s,x,\dl}_r-X^{s,x,\dl}_{\max\{s,\dl(r)\}}\|(1+\|X_{\max\{s,\dl(r)\}}^{s,x,\dl}\|)\rd r\nonumber\\ 
		&\leq K|\dl|^{-\frac{1}{4}}\bigg\{\E\int_s^{t } \|X_r^{s,x,\dl}-X^{s,x,\dl}_{\max\{s,\dl(r)\}}\|^{p_0-1}(1+\|X_{\max\{s,\dl(r)\}}^{s,x,\dl}\|)\rd r\nonumber\\
		&\quad+ \E\int_s^{t} \|X^{s,x,\dl}_{\max\{s,\dl(r)\}}\|^{p_0-2}\|X^{s,x,\dl}_r-X^{s,x,\dl}_{\max\{s,\dl(r)\}}\|(1+\|X_{\max\{s,\dl(r)\}}^{s,x,\dl}\|))\rd r\bigg\}\nonumber\\
		&\leq K|\dl|^{-\frac{1}{4}}\bigg\{\E\int_s^{t} (1+\|X_{\max\{s,\dl(r)\}}^{s,x,\dl}\|)\E\big[\|X_r^{s,x,\dl}-X^{s,x,\dl}_{\max\{s,\dl(r)\}}\|^{p_0-1}|\mathcal{F}_{\max\{s,\dl(r)\}} \big]\rd r\nonumber\\
		&\quad+ \E\int_s^{t} \|X^{s,x,\dl}_{\max\{s,\dl(r)\}}\|^{p_0-2}(1+\|X_{\max\{s,\dl(r)\}}^{s,x,\dl}\|) \E\big[\|X^{s,x,\dl}_r-X^{s,x,\dl}_{\max\{s,\dl(r)\}}\||\mathcal{F}_{\max\{s,\dl(r)\}} \big] \rd r\bigg\}\nonumber\\
		&\leq K\E \int_s^{t} 1+N_d^{\frac{p_0}{2}}+(1+L_d^{\frac{p_0}{2}})\|X^{s,x,\dl}_{\max\{s,\dl(r)\}}\|^{p_0} \rd r \nonumber\\ 
		&\leq  K(t-s)(1+N_d^{\frac{p_0}{2}})+K(1+L_d^{\frac{p_0}{2}})\int_s^t\sup_{u\in[s,r]} \E\|X_{u}^{s,x,\dl}\|^{p_0} \rd r.
	\end{align}
	Moreover, by Assumption \ref{A1}, Remarks \ref{remark1}, \ref{remark2}, and Young's inequality,
	\begin{align}\label{L2B_2}
		C_2(t)&\leq \frac{p_0} {2}\E\int^{t}_s\|X^{s,x,\dl}_r\|^{p_0-2} \{2 X^{s,x,\dl}_{\max\{s,\dl(r)\}} \mu(X^{s,x,\dl}_{\max\{s,\dl(r)\}})+(p_0-1) \|\sg(X^{s,x,\dl}_{\max\{s,\dl(r)\}})\|^2\} \rd r \nonumber\\
		&\leq \frac{p_0}{2}\E\int^{t}_s\|X^{s,x,\dl}_{r}\|^{p_0-2}\{2 X^{s,x,\dl}_{\max\{s,\dl(r)\}}(\mu(X^{s,x,\dl}_{\max\{s,\dl(r)\}})-\mu(0)) \nonumber\\
		&\hspace{50mm}+(p_0-1)\|\sg(X^{s,x,\dl}_{\max\{s,\dl(r)\}})-\sg(0)\|^2\}\rd r \nonumber\\
		&\quad+ \frac{p_0}{2}\E\int^{t}_s\|X^{s,x,\dl}_{r}\|^{p_0-2}\{2X^{s,x,\dl}_{\max\{s,\dl(r)\}}\mu(0)+(p_0-1)\|\sg(0)\|^2\}\rd r   \nonumber \\
		&\leq K\E\int^{t}_s 1+\|X^{s,x,\dl}_{r}\|^{p_0}+\|X^{s,x,\dl}_{\max\{s,\dl(r)\}}\|^{p_0}\rd r\nonumber \\
		&\leq K(t-s)+ K\int_s^{t} \sup_{u\in[s,r]} \E\|X_{u }^{s,x,\dl}\|^{p_0} \rd r.
	\end{align}
	Finally, on using Assumptions \ref{A3}, \ref{A4}, and Young's inequality, one obtains that
	\begin{align}\label{L2B_3}
		C_3(t) &\leq K\E\int^{t}_s\int_{\R^d} \big\{\|X^{s,x,\dl}_{r}\|^{p_0-2}\|\gm(X^{s,x,\dl}_{\max\{s,\dl(r)\}},z)-\gm(0,z)\|^2\nonumber\\
		&\hspace{35mm} +\|\gm(X^{s,x,\dl}_{\max\{s,\dl(r)\}},z)-\gm(0,z)\|^{p_0}\big\}\nu(\rd z)\rd r \nonumber\\
		&\quad+K\E\int^{t}_s\int_{\R^d} \lf\{\|X^{s,x,\dl}_{r}\|^{p_0-2}\|\gm(0,z)\|^2+\|\gm(0,z)\|^{p_0}\rt\}\nu(\rd z)\rd r \nonumber\\
		&\leq K\E\int^{t}_s N_d+N_d^{\frac{p_0}{2}}+\|X^{s,x,\dl}_{r}\|^{p_0}+(L_d+L_d^{\frac{p_0}{2}})\|X^{s,x,\dl}_{\max\{s,\dl(r)\}}\|^{p_0} \rd r\nonumber\\
		&\leq  K(t-s)(N_d+N_d^{\frac{p_0}{2}})+K(1+L_d^{\frac{p_0}{2}})\int^{t}_s \sup_{u\in[s,r]} \E\|X_{u }^{s,x,\dl}\|^{p_0} \rd r.
	\end{align}
	Substituting \eqref{L2B_1}, \eqref{L2B_2} and \eqref{L2B_3}  into \eqref{L3A13} yields that
	\begin{equation*}
		\begin{aligned}
			\sup_{t\in[s,T]}\E\|X^{s,x,\dl}_{t}\|^{p_0} \leq\|x\|^{p_0} + K(T-s)\big(1+N_d^{\frac{p_0}{2}}\big)+K\big(1+L_d^{\frac{p_0}{2}}\big)\int^{T}_s\sup_{u\in[s,r]}\E\|X_{u}^{s,x,\dl}\|^{p_0}\rd r.
		\end{aligned}
	\end{equation*}
	Then, we apply Gr\"onwall's lemma to obtain that
	\begin{equation*}
		\begin{aligned}
			\sup_{t\in[s,T]}   \E\|X^{s,x,\dl}_{t}\|^{p_0} 
			& \leq K\big( 1+\|x\|^{p_0}+N_d^{\frac{p_0}{2}}\big)e^{K\big(1+L_d^{\frac{p_0}{2}}\big)}.\\
		\end{aligned}
	\end{equation*}
\end{proof}
	\begin{proof}[\textbf{Proof of Lemma \ref{epbounds}}]
Using the fact that
$
\int_{A_\ep} g(z)\,\nu(\rd z)\le \int_{\R^d} g(z)\,\nu(\rd z)
$
for any non-negative integrable function \(g\), the proof follows from the same arguments as in the proof of Lemma \ref{L3}.
\end{proof}

\begin{proof}[\textbf{Proof of Lemma \ref{scheme3bounds}}]
	Fix $s\in [0,T]$, $x\in \R^d$, $\dl \in \Th$, $\ep \in (0,1)$, and $\mathcal{M} \in \N$ with $\mathcal{M}\geq (1 \vee \ep^{-2}\mathcal{Z}_d)^{2}$ throughout this proof. 
	Applying It\^o's formula to $\|X^{s,x,\dl,\ep,\mathcal{M}}_{t}\|^{p_0}$ in \eqref{Tamedscheme2}, by Lemmas \ref{schemebounds}, \ref{Formula for the remainder}, and Cauchy-Schwarz inequality, we obtain, for all $t \in [s,T]$, that
	\begin{equation*}
		\E\|X^{s,x,\dl,\ep,\mathcal{M}}_{t}\|^{p_0} \leq\|x\|^{p_0}+\sum_{k=1}^4 D_k(t),
	\end{equation*}
	where
	\begin{align*}
		D_1(t)&:=p_0\E \int ^{t}_s \|X^{s,x,\dl,\ep,\mathcal{M}}_r\|^{p_0-2}(X^{s,x,\dl,\ep,\mathcal{M}}_r-X^{s,x,\dl,\ep,\mathcal{M}}_{\max\{s,\dl(r)\}})\mu^\dl(X^{s,x,\dl,\ep,\mathcal{M}}_{\max\{s,\dl(r)\}}) \rd r,\\
		D_2(t)&:=\frac{p_0} {2}\E\int^{t}_s\|X^{s,x,\dl,\ep,\mathcal{M}}_r\|^{p_0-2} \{2 X^{s,x,\dl,\ep,\mathcal{M}}_{\max\{s,\dl(r)\}} \mu^\dl(X^{s,x,\dl,\ep,\mathcal{M}}_{\max\{s,\dl(r)\}})+(p_0-1) \|\sg^\dl(X^{s,x,\dl,\ep,\mathcal{M}}_{\max\{s,\dl(r)\}})\|^2\} \rd r, \\
		D_3(t)&:=K\E\int^{t}_s\int_{A_\ep} \{\|X^{s,x,\dl,\ep,\mathcal{M}}_r\|^{p_0-2}\|\gm(X^{s,x,\dl,\ep,\mathcal{M}}_{\max\{s,\dl(r)\}},z)\|^2+\|\gm(X^{s,x,\dl,\ep,\mathcal{M}}_{\max\{s,\dl(r)\}},z)\|^{p_0}\}\nu(\rd z)\rd r,\\
		D_4(t)&:=p_0\E\int_s^{t}\|X^{s,x,\dl,\ep,\mathcal{M}}_r\|^{p_0-2}X^{s,x,\dl,\ep,\mathcal{M}}_r\\
		&\quad\quad\quad\times\lf\{\int_{A_\ep}\gm\lf(X_{\max\{s,\dl(r)\}}^{s,x,\dl,\ep,\mathcal{M}}, z\rt) \nu(dz)- \lf(\frac{\nu(A_\ep)}{\mathcal{M}}\sum_{i=1}^{\mathcal{M}}\gm\lf(X_{\max\{s,\dl(r)\}}^{s,x,\dl,\ep,\mathcal{M}},V^{\dl,\ep,\mathcal{M}}_{i,\max\{s,\dl(r)\}}\rt)\rt) \rt\} \rd r.
	\end{align*}
	To estimate $D_1(t)$, we apply Young's inequality, Remark \ref{remark2}, Lemmas \ref{corollary1}\,\ref{lemma3} and \ref{L3A_4} to derive, for all $t \in [s,T]$, that
	\begin{align*}
		D_1(t) &\leq K|\dl|^{-\frac{1}{4}}\E\int_s^{t } \|X_r^{s,x,\dl,\ep,\mathcal{M}}\|^{p_0-2}\|X^{s,x,\dl,\ep,\mathcal{M}}_r-X^{s,x,\dl,\ep,\mathcal{M}}_{\max\{s,\dl(r)\}}\|(1+\|X_{\max\{s,\dl(r)\}}^{s,x,\dl,\ep,\mathcal{M}}\|)\rd r\nonumber\\ 
		&\leq K|\dl|^{-\frac{1}{4}}\bigg\{\E\int_s^{t} (1+\|X_{\max\{s,\dl(r)\}}^{s,x,\dl,\ep,\mathcal{M}}\|)\E\big[\|X_r^{s,x,\dl,\ep,\mathcal{M}}-X^{s,x,\dl,\ep,\mathcal{M}}_{\max\{s,\dl(r)\}}\|^{p_0-1}|\mathcal{F}_{\max\{s,\dl(r)\}} \big]\rd r\nonumber\\
		&\quad+ \E\int_s^{t} \|X^{s,x,\dl,\ep,\mathcal{M}}_{\max\{s,\dl(r)\}}\|^{p_0-2}(1+\|X_{\max\{s,\dl(r)\}}^{s,x,\dl,\ep,\mathcal{M}}\|) \E\big[\|X^{s,x,\dl,\ep,\mathcal{M}}_r-X^{s,x,\dl,\ep,\mathcal{M}}_{\max\{s,\dl(r)\}}\||\mathcal{F}_{\max\{s,\dl(r)\}} \big] \rd r\bigg\}\nonumber\\
		&\leq K\bigg\{\E\int_s^{t} (1+\|X_{\max\{s,\dl(r)\}}^{s,x,\dl,\ep,\mathcal{M}}\|)\bigg[
		\big(1+N_d^{\frac{p_0-1}{2}}+(1+L_d^{\frac{p_0-1}{2}})\|X^{s,x,\dl,\ep,\mathcal{M}}_{\max\{s,\dl(t)\}}\|^{p_0-1} \big)\\
		&\quad+  \bigg\|\int_{A_\ep} \gm(X^{s,x,\dl,\ep,\mathcal{M}}_{\max\{s,\dl(t)\}},z)\nu(\rd z) - \bigg(\frac{\nu(A_\ep)}{\mathcal{M}}\sum_{i=1}^{\mathcal{M}} \gm\lf(X_{\max\{s,\dl(t)\}}^{s,x,\dl,\ep,\mathcal{M}},V^{\dl,\ep,\mathcal{M}}_{i,\max\{s,\dl(t)\}}\rt)\bigg)\bigg\|^{p_0-1} \bigg] \rd r\nonumber\\
		&\quad+ \E\int_s^{t} \|X^{s,x,\dl,\ep,\mathcal{M}}_{\max\{s,\dl(r)\}}\|^{p_0-2}(1+\|X_{\max\{s,\dl(r)\}}^{s,x,\dl,\ep,\mathcal{M}}\|) \bigg[\big(1+N_d^{\frac{1}{2}}+(1+L_d^{\frac{1}{2}})\|X^{s,x,\dl,\ep,\mathcal{M}}_{\max\{s,\dl(t)\}}\|\big)\\
		&\quad+  \bigg\|\int_{A_\ep} \gm(X^{s,x,\dl,\ep,\mathcal{M}}_{\max\{s,\dl(t)\}},z)\nu(\rd z) - \bigg(\frac{\nu(A_\ep)}{\mathcal{M}}\sum_{i=1}^{\mathcal{M}} \gm\lf(X_{\max\{s,\dl(t)\}}^{s,x,\dl,\ep,\mathcal{M}},V^{\dl,\ep,\mathcal{M}}_{i,\max\{s,\dl(t)\}}\rt)\bigg)\bigg\| \bigg]\rd r\bigg\}\nonumber\\
		&\leq K\bigg\{\E\int_s^{t} 1+\|X_{\max\{s,\dl(r)\}}^{s,x,\dl,\ep,\mathcal{M}}\|^{p_0} \rd r+ \E\int_s^{t} 1+N_d^{\frac{p_0}{2}}+(1+L_d^{\frac{p_0}{2}})\|X^{s,x,\dl,\ep,\mathcal{M}}_{\max\{s,\dl(t)\}}\|^{p_0} \rd r\\
		&\quad+ \E\int_s^{t} \bigg\|\int_{A_\ep} \gm(X^{s,x,\dl,\ep,\mathcal{M}}_{\max\{s,\dl(t)\}},z)\nu(\rd z) - \bigg(\frac{\nu(A_\ep)}{\mathcal{M}}\sum_{i=1}^{\mathcal{M}} \gm\lf(X_{\max\{s,\dl(t)\}}^{s,x,\dl,\ep,\mathcal{M}},V^{\dl,\ep,\mathcal{M}}_{i,\max\{s,\dl(t)\}}\rt)\bigg)\bigg\|^{p_0}  \rd r \bigg\}\nonumber\\
		&\leq K\bigg\{ \int_s^{t} 1+N_d^{\frac{p_0}{2}}+(1+L_d^{\frac{p_0}{2}})\E\|X^{s,x,\dl,\ep,\mathcal{M}}_{\max\{s,\dl(t)\}}\|^{p_0} \rd r+ \int_s^{t} N_d+L_d \E \|X_{\max\{s,\dl(t)\}}^{s,x,\dl,\ep,\mathcal{M}}\|^{p_0} \rd r \bigg\}\nonumber\\
		&\leq K\E \int_s^{t} 1+N_d^{\frac{p_0}{2}}+(1+L_d^{\frac{p_0}{2}})\E\|X^{s,x,\dl,\ep,\mathcal{M}}_{\max\{s,\dl(r)\}}\|^{p_0} \rd r \nonumber\\ 
		&\leq  K(t-s)(1+N_d^{\frac{p_0}{2}})+K(1+L_d^{\frac{p_0}{2}})\int_s^t\sup_{u\in[s,r]} \E\|X_{u}^{s,x,\dl,\ep,\mathcal{M}}\|^{p_0} \rd r.
	\end{align*}
	The bounds for $D_2(t)$--$D_3(t)$ follow as in Lemma \ref{L3} (cf.~\eqref{L2B_2}--\eqref{L2B_3}), using Assumptions \ref{A1}, \ref{A3}, \ref{A4}, Remarks \ref{remark1}, \ref{remark2}, and Lemma \ref{corollary1}\,\ref{lemma3}. In particular, for all $t\in[s,T]$,
	\begin{equation*}
		\begin{aligned}
			D_2(t) &\leq K(t-s)+K\int_s^{t} \sup_{u\in[s,r]} \E\|X_{u }^{s,x,\dl,\ep,\mathcal{M}}\|^{p_0} \rd r,\\
			D_3(t) &\leq K(t-s)(N_d+N_d^{\frac{p_0}{2}}) + K(1+L_d^{\frac{p_0}{2}})\int_s^{t }\sup_{u\in[s,r]} \E\|X_{u}^{s,x,\dl,\ep,\mathcal{M}}\|^{p_0} \rd r.
		\end{aligned}
	\end{equation*}
	For $D_4(t)$, we apply Young's inequality to obtain, for all $t\in[s,T]$, that
	\begin{align*}
		D_4(t) &\leq  K\E\int^{t}_s\|X_r^{s,x,\dl,\ep,\mathcal{M}}\|^{p_0}\rd r\\
		+K\E \int^{t}_s \bigg\|&\int_{A_\ep}\gm\lf(X_{\max\{s,\dl(r)\}}^{s,x,\dl,\ep,\mathcal{M}}, z\rt) \nu(\rd z)- \bigg(\frac{\nu(A_\ep)}{\mathcal{M}}\sum_{i=1}^{\mathcal{M}}\gm\lf(X_{\max\{s,\dl(r)\}}^{s,x,\dl,\ep,\mathcal{M}},V^{\dl,\ep,\mathcal{M}}_{i,\max\{s,\dl(r)\}}\rt)\bigg)\bigg\|^{p_0}\rd r.
	\end{align*}
	By Lemma \ref{L3A_4} and the assumption of $\mathcal{M} \geq (1 \vee \ep^{-2}\mathcal{Z}_d)^{2}$, we have, for all $t\in[s,T]$, that
	\begin{align*}
		D_4(t) &\leq K\E\int^{t}_s \|X_{r}^{s,x,\dl,\ep,\mathcal{M}}\|^{p_0}\rd r+ K\int^{t}_s 
		N_d+L_d\E\|X_{\max\{s,\dl(r)\}}^{s,x,\dl,\ep,\mathcal{M}}\|^{p_0}  \rd r\\
		&\leq K  N_d(t-s)+ K(1+L_d)\int^{t}_s 
		\sup_{u\in[s,r]}\E\|X_{u}^{s,x,\dl,\ep,\mathcal{M}}\|^{p_0} \rd r.
	\end{align*}
	Collecting the above estimates yields,
	\begin{align*}
		\sup_{t\in[s,T]}\E\|X^{s,x,\dl,\ep,\mathcal{M}}_{t}\|^{p_0} &\leq\|x\|^{p_0} + K(T-s)\big(1+N_d^{\frac{p_0}{2}}\big)+K(1+L_d^{\frac{p_0}{2}})\int^{T}_s\sup_{u\in[s,r]}\E\|X_{u}^{s,x,\dl,\ep,\mathcal{M}}\|^{p_0}\rd r.
	\end{align*}
	An application of Gr\"onwall’s lemma completes the proof.
\end{proof}

\subsection{Proof of results in Section \ref{section5.2}}\label{app:section5.2}
The following lemma provides an $O(|\dl|)$ estimate for the one-step error of schemes \eqref{tamedscheme1}, \eqref{tamedscheme2}, and \eqref{Tamedscheme2}. This rate is sharper than the conditional increment estimate of Lemma \ref{corollary1} and will be crucial for obtaining the final convergence rate in Theorem \ref{theorem1}\,\ref{theorem(i)}.
\begin{lemma}\label{corollary2}
	Let Assumptions \ref{A1}--\ref{A4} hold, then there exists $K>0$ such that, 
	\begin{enumerate}
	\item\label{lemma5}  for all $s\in [0,T]$, $x\in \R^d$, $\dl \in \Th$, $\ep\in(0,1)$, and $\rho \in [2,\frac{2p_0}{\chi+2}]$,
	\[
	\sup_{t\in[s,T]} \E \|X_t^{s,x,\dl}-X^{s,x,\dl}_{\max\{s,\dl(t)\}} \|^{\rho} \leq   K|\dl|(1+L_d^{\frac{\rho}{2}})(1+\|x\|^{p_0}+N_d^{\frac{p_0}{2}})e^{K(1+L_d^{\frac{p_0}{2}})},
	\]
     \item\label{corollary2(ii)} for all $\dl \in \Th$, $s\in [0,T], x\in \R^d$, $\ep\in(0,1)$, and $\rho \in [2,\frac{2p_0}{\chi+2}]$,
     \[
     \sup_{t\in[s,T]} \E \|X_t^{s,x,\dl,\ep}-X^{s,x,\dl,\ep}_{\max\{s,\dl(t)\}} \|^{\rho}\leq K|\dl|(1+L_d^{\frac{\rho}{2}})(1+\|x\|^{p_0}+N_d^{\frac{p_0}{2}})e^{K(1+L_d^{\frac{p_0}{2}})},
     \]
     \item\label{lemma9} for all $\dl \in \Th$, $s\in [0,T], x\in \R^d$, $\ep\in(0,1)$, $\mathcal{M} \in \N$ with $\mathcal{M} \geq (1 \vee \ep^{-2}\mathcal{Z}_d)^{2}$, and $\rho \in [2,\frac{2p_0}{\chi+2}]$,
     \[
     \sup_{t\in[s,T]} \E \|X_t^{s,x,\dl,\ep,\mathcal{M}}-X^{s,x,\dl,\ep,\mathcal{M}}_{\max\{s,\dl(t)\}} \|^{\rho} \leq K|\dl|(1+L_d^{\frac{\rho}{2}})\big(1+\|x\|^{p_0}+N_d^{p_0}\big)e^{K(1+L_d^{p_0})}.
     \]
	\end{enumerate}
	\begin{proof}[Proof of (i):]
		Fix $s\in [0,T]$, $x\in \R^d$, and $\dl \in \Th$ throughout this proof. By the definition of scheme \eqref{tamedscheme1}, we obtain, for all $t\in[s,T]$, that
		\begin{align*}
			\E\|X^{s,x,\dl}_t - X^{s,x,\dl}_{\max\{s,\dl(t)\}} \|^\rho &\leq K\E \big\|\int^{t}_{\max\{s,\dl(t)\}} \mu^{\dl}(X^{s,x,\dl}_{\max\{s,\dl(r)\}})\rd r\big\|^\rho  \\
			&\quad+ K\E\big\|\int^{t}_{\max\{s,\dl(t)\}} \sg^{\dl}(X^{s,x,\dl}_{\max\{s,\dl(r)\}})\rd W_r \big\|^\rho \\
			&\quad+K\E\big\|\int^{t}_{\max\{s,\dl(t)\}} \int_{\R^d} \gm(X^{s,x,\dl}_{\max\{s,\dl(r)\}},z) \tl\pi(\rd z,\rd r)\big\|^\rho .
		\end{align*}
		Applying H\"older's inequality, \cite[Theorem 7.3]{xuerongmao}, and \cite[Lemma 1]{Mikulevicius} yields, for all $t\in[s,T]$, that
		\begin{align*}
			\E\|X^{s,x,\dl}_t - X^{s,x,\dl}_{\max\{s,\dl(t)\}} \|^\rho &\leq K|\dl|^{\rho-1}\E \int^{t}_{\max\{s,\dl(t)\}} \|\mu^{\dl}(X^{s,x,\dl}_{\max\{s,\dl(r)\}})\|^{\rho}\rd r\\
			&\quad+ K|\dl|^{\frac{\rho}{2}-1}\E\int^{t}_{\max\{s,\dl(t)\}} \|\sg^{\dl}(X^{s,x,\dl}_{\max\{s,\dl(r)\}})\|^{\rho}\rd r \\
			&\quad+K|\dl|^{\frac{\rho}{2}-1}\E\int^{t}_{\max\{s,\dl(t)\}} \big(\int_{\R^d} \|\gm(X^{s,x,\dl}_{\max\{s,\dl(r)\}},z)\|^2 \nu(\rd z)\big)^{\frac{\rho}{2}} \rd r\\
			&\quad+K\E\int^{t}_{\max\{s,\dl(t)\}}\int_{\R^d} \|\gm(X^{s,x,\dl}_{\max\{s,\dl(r)\}},z)\|^\rho \nu(\rd z) \rd r.
		\end{align*}
		Then, we apply Remarks \ref{remark1} and \ref{remark2}, to obtain, for all $t\in[s,T]$, that
		\begin{align*}
			\E\|X^{s,x,\dl}_t - X^{s,x,\dl}_{\max\{s,\dl(t)\}} \|^\rho &\leq K|\dl|^{\rho}\E (1+\|X^{s,x,\dl}_{\max\{s,\dl(t)\}}\|^{\frac{\chi+2}{2}\rho})\\
			&\quad+ K|\dl|^{\frac{\rho}{2}}\E  (1+\|X^{s,x,\dl}_{\max\{s,\dl(t)\}}\|^{\frac{\chi+2}{2}\rho})\\
			&\quad+K|\dl|^{\frac{\rho}{2}-1}\E\int^{t}_{\max\{s,\dl(t)\}} \big(\int_{\R^d} \|\gm(X^{s,x,\dl}_{\max\{s,\dl(r)\}},z)-\gm(0,z)\|^2 \nu(\rd z)\big)^{\frac{\rho}{2}} \rd r\\
			&\quad+K|\dl|^{\frac{\rho}{2}-1}\E\int^{t}_{\max\{s,\dl(t)\}} \big(\int_{\R^d} \|\gm(0,z)\|^2 \nu(\rd z)\big)^{\frac{\rho}{2}} \rd r\\
			&\quad+K\E\int^{t}_{\max\{s,\dl(t)\}}\int_{\R^d} \|\gm(X^{s,x,\dl}_{\max\{s,\dl(r)\}},z)-\gm(0,z)\|^\rho \nu(\rd z) \rd r\\
			&\quad+K\E\int^{t}_{\max\{s,\dl(t)\}}\int_{\R^d} \|\gm(0,z)\|^\rho \nu(\rd z) \rd r,
		\end{align*}
		which on the application of Assumptions \ref{A3} and \ref{A4} gives, for all $t\in[s,T]$, that
		\begin{align*}
			\E\|X^{s,x,\dl}_t - X^{s,x,\dl}_{\max\{s,\dl(t)\}} \|^\rho &\leq K(|\dl|^{\frac{\rho}{2}}+|\dl|^{\rho})\E  (1+\|X^{s,x,\dl}_{\max\{s,\dl(t)\}}\|^{\frac{\chi+2}{2}\rho})\\
			&\quad+K|\dl|^{\frac{\rho}{2}}(L_d^{\frac{\rho}{2}}\E\|X^{s,x,\dl}_{\max\{s,\dl(t)\}}\|^{\rho}+N_d^{\frac{\rho}{2}})\\
			&\quad+K|\dl|(L_d\E\|X^{s,x,\dl}_{\max\{s,\dl(t)\}}\|^{\rho}+N_d).
		\end{align*}
		Given $\rho \in[2,\frac{2p_0}{\chi+2}]$, applying Lemma \ref{L3} yields that
		\begin{align*}
			\sup_{t\in[s,T]}\E\|X^{s,x,\dl}_t - X^{s,x,\dl}_{\max\{s,\dl(t)\}} \|^\rho 
			&\leq K|\dl|(1+N_d^{\frac{\rho}{2}}+L_d^{\frac{\rho}{2}}+(1+L_d^{\frac{\rho}{2}})\sup_{t\in[s,T]}\E\|X^{s,x,\dl}_{t}\|^{p_0})\\
			&\leq K|\dl|(1+L_d^{\frac{\rho}{2}})(1+\|x\|^{p_0}+N_d^{\frac{p_0}{2}})e^{K(1+L_d^{\frac{p_0}{2}})}.
		\end{align*}
	\end{proof}
	\begin{proof}[Proof of (ii)]
		The proof follows the same lines as above and is therefore omitted.
	\end{proof}
	\begin{proof}[Proof of (iii)]
		Fix $s\in [0,T]$, $x\in \R^d$, $\dl \in \Th$, $\ep \in (0,1)$, and $\mathcal{M} \in \N$ with $\mathcal{M} \geq (1 \vee \ep^{-2}\mathcal{Z}_d)^{2}$ throughout this proof. By the definition of scheme \eqref{Tamedscheme2}, we obtain, for all $t\in[s,T]$, that
		\begin{align*}
			\E\|X^{s,x,\dl,\ep,\mathcal{M}}_t - X^{s,x,\dl,\ep,\mathcal{M}}_{\max\{s,\dl(t)\}} \|^\rho &\leq K\E \big\|\int^{t}_{\max\{s,\dl(t)\}} \mu^{\dl}(X^{s,x,\dl,\ep,\mathcal{M}}_{\max\{s,\dl(r)\}})\rd r\big\|^\rho  \\
			&\quad+ K\E\big\|\int^{t}_{\max\{s,\dl(t)\}} \sg^{\dl}(X^{s,x,\dl,\ep,\mathcal{M}}_{\max\{s,\dl(r)\}})\rd W_r \big\|^\rho \\
			&\quad+K\E\big\|\int^{t}_{\max\{s,\dl(t)\}} \int_{A_\ep} \gm(X^{s,x,\dl,\ep,\mathcal{M}}_{\max\{s,\dl(r)\}},z) \tl\pi(\rd z,\rd r)\big\|^\rho \\
			&\quad+K\E\bigg\|\int^{t}_{\max\{s,\dl(t)\}} \int_{A_\ep} \gm(X^{s,x,\dl,\ep,\mathcal{M}}_{\max\{s,\dl(r)\}},z)\nu(\rd z) \\
			&\quad\quad\quad - \lf\{\frac{\nu(A_\ep)}{\mathcal{M}}\sum_{i=1}^{\mathcal{M}}\gm\lf(X_{\max\{s,\dl(r-)\}}^{s,x,\dl,\ep,\mathcal{M}},V^{\dl,\ep,\mathcal{M}}_{i,\max\{s,\dl(r)\}}\rt)\rt\}\rd r \bigg\|^\rho.
		\end{align*}
		By H\"older's inequality, \cite[Theorem 7.3]{xuerongmao}, and \cite[Lemma 1]{Mikulevicius}, we obtain, for all $t\in[s,T]$, that
		\begin{align*}
			\E\|X^{s,x,\dl,\ep,\mathcal{M}}_t - X^{s,x,\dl,\ep,\mathcal{M}}_{\max\{s,\dl(t)\}} \|^\rho &\leq K|\dl|^{\rho-1}\E \int^{t}_{\max\{s,\dl(t)\}} \|\mu^{\dl}(X^{s,x,\dl,\ep,\mathcal{M}}_{\max\{s,\dl(r)\}})\|^{\rho}\rd r\\
			&\quad+ K|\dl|^{\frac{\rho}{2}-1}\E\int^{t}_{\max\{s,\dl(t)\}} \|\sg^{\dl}(X^{s,x,\dl,\ep,\mathcal{M}}_{\max\{s,\dl(r)\}})\|^{\rho}\rd r \\
			&\quad+K|\dl|^{\frac{\rho}{2}-1}\E\int^{t}_{\max\{s,\dl(t)\}} \big(\int_{A_\ep} \|\gm(X^{s,x,\dl,\ep,\mathcal{M}}_{\max\{s,\dl(r)\}},z)-\gm(0,z)\|^2 \nu(\rd z)\big)^{\frac{\rho}{2}} \rd r\\
			&\quad+K|\dl|^{\frac{\rho}{2}-1}\E\int^{t}_{\max\{s,\dl(t)\}} \big(\int_{A_\ep} \|\gm(0,z)\|^2 \nu(\rd z)\big)^{\frac{\rho}{2}} \rd r\\
			&\quad+K\E\int^{t}_{\max\{s,\dl(t)\}}\int_{A_\ep} \|\gm(X^{s,x,\dl,\ep,\mathcal{M}}_{\max\{s,\dl(r)\}},z)-\gm(0,z)\|^\rho \nu(\rd z) \rd r\\
			&\quad+K\E\int^{t}_{\max\{s,\dl(t)\}}\int_{A_\ep} \|\gm(0,z)\|^\rho \nu(\rd z) \rd r\\
			&\quad+K|\dl|^{\rho-1}\E\int^{t}_{\max\{s,\dl(t)\}} \bigg\|\int_{A_\ep} \gm(X^{s,x,\dl,\ep,\mathcal{M}}_{\max\{s,\dl(r)\}},z)\nu(\rd z) \\
			&\quad\quad\quad - \lf\{\frac{\nu(A_\ep)}{\mathcal{M}}\sum_{i=1}^{\mathcal{M}}\gm\lf(X_{\max\{s,\dl(r-)\}}^{s,x,\dl,\ep,\mathcal{M}},V^{\dl,\ep,\mathcal{M}}_{i,\max\{s,\dl(r)\}}\rt)\rt\}\bigg\|^\rho\rd r.
		\end{align*}
		The first two terms are controlled by Remark \ref{remark1} and the taming bound in Remark \ref{remark2}, while the third to sixth terms are handled using Assumptions \ref{A3} and \ref{A4}. Finally, the Monte Carlo term is estimated by \eqref{L3A_4(t)} (together with $\mathcal{M}\ge (1 \vee \ep^{-2}\mathcal{Z}_d)^{2}$). Collecting these bounds yields, for all $t\in[s,T]$, that
		\begin{align*}
			\E\|X^{s,x,\dl,\ep,\mathcal{M}}_t - X^{s,x,\dl,\ep,\mathcal{M}}_{\max\{s,\dl(t)\}} \|^\rho 
			\leq K|\dl|\lf(1+N_d^{\frac{\rho}{2}}+L_d^{\frac{\rho}{2}}+(1+L_d^{\frac{\rho}{2}})\E\|X^{s,x,\dl,\ep,\mathcal{M}}_{\max\{s,\dl(t)\}}\|^{\frac{\chi+2}{2}\rho}\rt).
		\end{align*}
		As $\rho \in[2,\frac{2p_0}{\chi+2}]$, we apply Lemma \ref{scheme3bounds} to obtain that,
		\begin{align*}
			\sup_{t\in[s,T]}	\E\|X^{s,x,\dl,\ep,\mathcal{M}}_t - X^{s,x,\dl,\ep,\mathcal{M}}_{\max\{s,\dl(t)\}} \|^\rho 
			\leq K|\dl|(1+L_d^{\frac{\rho}{2}})\big(1+\|x\|^{p_0}+N_d^{p_0}\big)e^{K(1+L_d^{p_0})}.
		\end{align*}
	\end{proof}
\end{lemma}
	
\begin{proof}[\textbf{Proof of Lemma \ref{convergencescheme1}}]
	Fix $s\in [0,T]$, $x\in \R^d$, and $\dl \in \Th$ throughout this proof. By the definition of SDE \eqref{exactsolu} and scheme \eqref{tamedscheme1}, we obtain, for all $t \in [s,T]$, that 
	\begin{equation}\label{convergence1}
		\begin{aligned}
			X_t^{s,x,\iota} - X_t^{s,x,\dl}&= \int_s^t \mu(X_{r-}^{s,x,\iota})-\mu^\dl(X^{s,x,\dl}_{\max\{s,\dl(r-)\}}) \rd r\\
			&+\int_s^t \sg(X_{r-}^{s,x,\iota})-\sg^\dl(X^{s,x,\dl}_{\max\{s,\dl(r-)\}}) \rd W_r \\
			&+\int_s^t \int_{\R^d} \gm(X^{s,x,\iota}_{r-},z)-\gm(X^{s,x,\dl}_{\max\{s,\dl(r-)\}},z) \tl{\pi}(\rd z ,\rd r).
		\end{aligned}
	\end{equation}
	Further applying It\^o's formula yields, for all $t \in [s,T]$, that
	\begin{align*}
		&\|X_t^{s,x,\iota} - X_t^{s,x,\dl} \|^{p}\\
		&=p\int^{t}_s\|X_{r-}^{s,x,\iota}-X^{s,x,\dl}_{r-}\|^{p-2} (X_{r-}^{s,x,\iota}-X^{s,x,\dl}_{r-}) (\mu(X_{r-}^{s,x,\iota})-\mu^\dl(X^{s,x,\dl}_{\max\{s,\dl(r-)\}}) ) \rd r \\
		&\quad+ p\int^{t}_s\|X_{r-}^{s,x,\iota}-X^{s,x,\dl}_{r-}\|^{p-2} (X_{r-}^{s,x,\iota}-X^{s,x,\dl}_{r-})(\sg(X_{r-}^{s,x,\iota} )-\sg^\dl(X^{s,x,\dl}_{\max\{s,\dl(r-)\}}))\rd W_r\\
		&\quad+\frac{p(p-2)}{2} \int^{t}_s \|X_{r-}^{s,x,\iota}-X^{s,x,\dl}_{r-}\|^{p-4} \|(\sg(X_{r-}^{s,x,\iota})-\sg^\dl(X^{s,x,\dl}_{\max\{s,\dl(r-)\}}))^*(X_{r-}^{s,x,\iota}-X^{s,x,\dl}_{r-}) \|^{2}\rd r\\
		&\quad+\frac{p}{2}\int^{t}_s\|X_{r-}^{s,x,\iota}-X^{s,x,\dl}_{r-}\|^{p-2}\|\sg(X_{r-}^{s,x,\iota})-\sg^\dl(X^{s,x,\dl}_{\max\{s,\dl(r-)\}})\|^2\rd r\\
		&\quad+p\int^t_s\int_{\R^d}\|X_{r-}^{s,x,\iota}-X^{s,x,\dl}_{r-}\|^{p-2} (X_{r-}^{s,x,\iota}-X^{s,x,\dl}_{r-}) (\gm(X^{s,x,\iota}_{r},z)-\gm(X^{s,x,\dl}_{\max\{s,\dl(r-)\}},z))\tl{\pi}(\rd z,\rd r)\\
		&\quad+\int_s^t\int_{\R^d}\big\{\|X_{r-}^{s,x,\iota}-X^{s,x,\dl}_{r-}+(\gm(X^{s,x,\iota}_{r-},z)-\gm(X^{s,x,\dl}_{\max\{s,\dl(r-)\}},z))\|^{p}-\|X_{r-}^{s,x,\iota}-X^{s,x,\dl}_{r-}\|^{p}\\
		&\quad\quad-p\|X_{r-}^{s,x,\iota}-X^{s,x,\dl}_{r-}\|^{p-2}(X_{r-}^{s,x,\iota}-X^{s,x,\dl}_{r-})(\gm(X^{s,x,\iota}_{r-},z)-\gm(X^{s,x,\dl}_{\max\{s,\dl(r-)\}},z))\big\}\pi(\rd z,\rd r),
	\end{align*}
	almost surely. Then, we define a stopping time $\upsilon^{s,x,\dl}_R$ := $\tau^{s,x}_R \wedge \tau^{s,x,\dl}_R$  with $R \in \N$ where $\tau^{s,x}_R:= \inf\{t\geq s: \|X^{s,x,\iota}_t\|\geq R\}\wedge T$ and $\tau^{s,x,\dl}_R:= \inf\{t\geq s: \|X^{s,x,\dl}_t\|\geq R\}\wedge T$. Then, by Cauchy-Schwarz inequality and Lemma \ref{Formula for the remainder}, we obtain, for all $t \in [s,T]$, that
	\begin{align*}
		&\E\|X_{t \wedge \upsilon^{s,x,\dl}_R}^{s,x,\iota} - X_{t \wedge \upsilon^{s,x,\dl}_R}^{s,x,\dl} \|^{p} \\
		& \leq
		p\E \int^{t \wedge \upsilon^{s,x,\dl}_R}_s \|X_{r}^{s,x,\iota}-X^{s,x,\dl}_{r}\|^{p-2}(X^{s,x,\iota}_r-X^{s,x,\dl}_{r})(\mu(X_{r}^{s,x,\iota})-\mu^\dl(X^{s,x,\dl}_{\max\{s,\dl(r)\}}) )  \rd r\\
		&\quad+\frac{p(p-1)}{2}\E\int^{t \wedge \upsilon^{s,x,\dl}_R}_s\|X_{r}^{s,x,\iota}-X^{s,x,\dl}_{r}\|^{p-2}  \|\sg(X_{r}^{s,x,\iota} )-\sg^\dl(X^{s,x,\dl}_{\max\{s,\dl(r)\}})\|^2 \rd r \\
		&\quad+K\E\int^{t \wedge \upsilon^{s,x,\dl}_R}_s\int_{\R^d} \big\{\|X_{r}^{s,x,\iota}-X^{s,x,\dl}_{r}\|^{p-2} \|\gm(X^{s,x,\iota}_{r},z)-\gm(X^{s,x,\dl}_{\max\{s,\dl(r)\}},z)\|^2\\
		&\hspace{40mm}+\|\gm(X^{s,x,\iota}_{r},z)-\gm(X^{s,x,\dl}_{\max\{s,\dl(r)\}},z)\|^{p}\big\}\nu(\rd z)\rd r.
	\end{align*}
	The above can be further rewritten as,
		\begin{align}\label{33}
			&\E\|X_{t \wedge \upsilon^{s,x,\dl}_R}^{s,x,\iota} - X_{t \wedge \upsilon^{s,x,\dl}_R}^{s,x,\dl} \|^{p}\leq
			\frac{p}{2}\E \int ^{t \wedge \upsilon^{s,x,\dl}_R}_s \|X_{r}^{s,x,\iota}-X^{s,x,\dl}_{r}\|^{p-2}\\
			&\quad\quad \times\bigg(2(X^{s,x,\iota}_r-X^{s,x,\dl}_{r})(\mu(X_{r}^{s,x,\iota})-\mu(X^{s,x,\dl}_{r}))+(p-1) \|\sg(X_{r}^{s,x,\iota} )-\sg(X^{s,x,\dl}_{r})\|^2 \nonumber \\
			&\hspace{15mm} +2(p-1)(\sg(X_{r}^{s,x,\iota}
			)-\sg(X^{s,x,\dl}_{r}))(\sg(X^{s,x,\dl}_{r})-\sg^\dl(X^{s,x,\dl}_{\max\{s,\dl(r)\}}))\bigg)  \rd r\nonumber \\
			&\quad+ p\E \int ^{t \wedge \upsilon^{s,x,\dl}_R}_s \|X_{r}^{s,x,\iota}-X^{s,x,\dl}_{r}\|^{p-2} (X^{s,x,\iota}_r-X^{s,x,\dl}_{r})(\mu(X^{s,x,\dl}_{r})-\mu^\dl(X^{s,x,\dl}_{\max\{s,\dl(r)\}}) ) \rd r\nonumber \\
			&\quad+\frac{p(p-1)} {2}\E\int^{t \wedge \upsilon^{s,x,\dl}_R}_s\|X_{r}^{s,x,\iota}-X^{s,x,\dl}_{r}\|^{p-2}  \|\sg(X_{r}^{s,x,\dl} )-\sg^\dl(X^{s,x,\dl}_{\max\{s,\dl(r)\}})\|^2 \rd r\nonumber  \\
			&\quad+K\E\int^{t \wedge \upsilon^{s,x,\dl}_R}_s\int_{\R^d} \big\{\|X_{r}^{s,x,\iota}-X^{s,x,\dl}_{r}\|^{p-2} \|\gm(X^{s,x,\iota}_{r},z)-\gm(X^{s,x,\dl}_{\max\{s,\dl(r)\}},z)\|^2\nonumber \\
			&\quad\quad+\|\gm(X^{s,x,\iota}_{r},z)-\gm(X^{s,x,\dl}_{\max\{s,\dl(r)\}},z)\|^{p}\big\}\nu(\rd z)\rd r.\nonumber 
		\end{align}
	We note that, by applying Young's inequality, i.e., $2ab\leq a^2/\eta+ \eta b^2$ with $\eta = (p-1)/((p_0-p))$, we obtain, for all $x,y,z\in\R^d$, that
	\begin{align*}
		&(p-1) \|\sg(x )-\sg(y)\|^2 +2(p-1)(\sg(x
		)-\sg(y))(\sg(y)-\sg^\dl(z)) \\
		&\leq(p-1) \|\sg(x)-\sg(y)\|^2+(p-1)\frac{p_0-p}{p-1}\|\sg(x
		)-\sg(y)\|^2+\frac{(p-1)^2}{p_0-p}\|\sg(y)-\sg^\dl(z)\|^2  \\
		&\leq (p_0-1) \|\sg(x )-\sg(y)\|^2 +K\|\sg(y)-\sg^\dl(z)\|^2.
	\end{align*}
	Applying this to \eqref{33} yields, for all $t\in [s,T],$ that
	\begin{align*}
		&\E\|X_{t \wedge \upsilon^{s,x,\dl}_R}^{s,x,\iota} - X_{t \wedge \upsilon^{s,x,\dl}_R}^{s,x,\dl} \|^{p} \\
		&\leq
		\frac{p}{2}\E \int ^{t \wedge \upsilon^{s,x,\dl}_R}_s \|X_{r}^{s,x,\iota}-X^{s,x,\dl}_{r}\|^{p-2}\big(2(X^{s,x,\iota}_r-X^{s,x,\dl}_{r})(\mu(X_{r}^{s,x,\iota})-\mu(X^{s,x,\dl}_{r}) )\\
		&\hspace{25mm} +(p_0-1) \|\sg(X_{r}^{s,x,\iota} )-\sg(X^{s,x,\dl}_{r})\|^2 \big)  \rd r\\
		&\quad+ p\E \int ^{t \wedge \upsilon^{s,x,\dl}_R}_s \|X_{r}^{s,x,\iota}-X^{s,x,\dl}_{r}\|^{p-2} (X^{s,x,\iota}_r-X^{s,x,\dl}_{r})(\mu(X^{s,x,\dl}_{r})-\mu^\dl(X^{s,x,\dl}_{\max\{s,\dl(r)\}}) ) \rd r\\
		&\quad+K\E\int^{t \wedge \upsilon^{s,x,\dl}_R}_s\|X_{r}^{s,x,\iota}-X^{s,x,\dl}_{r}\|^{p-2}  \|\sg(X_{r}^{s,x,\dl} )-\sg^\dl(X^{s,x,\dl}_{\max\{s,\dl(r)\}})\|^2 \rd r \\
		&\quad+K\E\int^{t \wedge \upsilon^{s,x,\dl}_R}_s\int_{\R^d} \|X_{r}^{s,x,\iota}-X^{s,x,\dl}_{r}\|^{p-2} \|\gm(X^{s,x,\iota}_{r},z)-\gm(X^{s,x,\dl}_{r},z)\|^2\nu(\rd z)\rd r\\
		&\quad+K\E\int^{t \wedge \upsilon^{s,x,\dl}_R}_s\int_{\R^d} \|X_{r}^{s,x,\iota}-X^{s,x,\dl}_{r}\|^{p-2} \|\gm(X^{s,x,\dl}_{r},z)-\gm(X^{s,x,\dl}_{\max\{s,\dl(r)\}},z)\|^2\nu(\rd z)\rd r\\
		&\quad+ K\E\int^{t \wedge \upsilon^{s,x,\dl}_R}_s\int_{\R^d} \|\gm(X^{s,x,\iota}_{r},z)-\gm(X^{s,x,\dl}_{r},z)\|^{p}\nu(\rd z)\rd r\\
		&\quad+ K\E\int^{t \wedge \upsilon^{s,x,\dl}_R}_s\int_{\R^d} \|\gm(X^{s,x,\dl}_{r},z)-\gm(X^{s,x,\dl}_{\max\{s,\dl(r)\}},z)\|^{p}\nu(\rd z)\rd r,
	\end{align*}
	which on the application of Assumptions \ref{A1}, \ref{A4}, Cauchy-Schwarz inequality, and Young's inequality implies,
	\begin{align*}
		\E\|X_{t \wedge \upsilon^{s,x,\dl}_R}^{s,x,\iota} - X_{t \wedge \upsilon^{s,x,\dl}_R}^{s,x,\dl} \|^{p} &\leq
		K \E\int ^{t \wedge \upsilon^{s,x,\dl}_R}_s (1+L_d^{\frac{p}{2}})\|X_{r}^{s,x,\iota}-X^{s,x,\dl}_{r}\|^{p} \rd r\\
		&\quad+ K\E \int ^{t \wedge \upsilon^{s,x,\dl}_R}_s \|\mu(X^{s,x,\dl}_{r})-\mu(X^{s,x,\dl}_{\max\{s,\dl(r)\}}) \|^p \rd r\\
		&\quad+ K\E \int ^{t \wedge \upsilon^{s,x,\dl}_R}_s \|\mu(X^{s,x,\dl}_{\max\{s,\dl(r)\}})-\mu^\dl(X^{s,x,\dl}_{\max\{s,\dl(r)\}}) \|^p \rd r\\
		&\quad+K\E\int^{t \wedge \upsilon^{s,x,\dl}_R}_s \|\sg(X_{r}^{s,x,\dl} )-\sg(X^{s,x,\dl}_{\max\{s,\dl(r)\}})\|^p \rd r \\
		&\quad+K\E\int^{t \wedge \upsilon^{s,x,\dl}_R}_s \|\sg(X^{s,x,\dl}_{\max\{s,\dl(r)\}})-\sg^\dl(X^{s,x,\dl}_{\max\{s,\dl(r)\}})\|^p \rd r \\
		&\quad+K\E\int^{t \wedge \upsilon^{s,x,\dl}_R}_s\bigg(\int_{\R^d}  \|\gm(X^{s,x,\dl}_{r},z)-\gm(X^{s,x,\dl}_{\max\{s,\dl(r)\}},z)\|^2\nu(\rd z)\bigg)^{\frac{p}{2}}\rd r\\
		&\quad+ K\E\int^{t \wedge \upsilon^{s,x,\dl}_R}_s\int_{\R^d} \|\gm(X^{s,x,\dl}_{r},z)-\gm(X^{s,x,\dl}_{\max\{s,\dl(r)\}},z)\|^{p}\nu(\rd z)\rd r.
	\end{align*}
	By Assumptions \ref{A2}, \ref{A4} and Remark \ref{remark3}, one obtains, for all $t\in[s,T]$, that
	\begin{align*}
		&\E\|X_{t \wedge \upsilon^{s,x,\dl}_R}^{s,x,\iota} - X_{t \wedge \upsilon^{s,x,\dl}_R}^{s,x,\dl} \|^{p}\\
		&\leq
		K \E\int ^{t \wedge \upsilon^{s,x,\dl}_R}_s (1+L_d^{\frac{p}{2}})\|X_{r}^{s,x,\iota}-X^{s,x,\dl}_{r}\|^{p} \rd r\\
		&\quad+K\E \int ^{t \wedge \upsilon^{s,x,\dl}_R}_s (L_d+L_d^{\frac{p}{2}})\|X_{r}^{s,x,\dl}-X^{s,x,\dl}_{\max\{s,\dl(r)\}}\|^{p} \rd r\\
		&\quad+ K\E \int ^{t \wedge \upsilon^{s,x,\dl}_R}_s \|X^{s,x,\dl}_{r}-X^{s,x,\dl}_{\max\{s,\dl(r)\}}\|^p(1+\|X^{s,x,\dl}_{r}\|^{\frac{\chi}{2}}+\|X^{s,x,\dl}_{\max\{s,\dl(r)\}}\|^{\frac{\chi}{2}})^p \rd r\\
		&\quad+ K|\dl|^{\frac{p}{2}}\E \int ^{t \wedge \upsilon^{s,x,\dl}_R}_s \lf(1+\|X^{s,x,\dl}_{\max\{s,\dl(r)\}}\|^{(\frac{3}{2}\chi+1)p}\rt)  \rd r.
	\end{align*}
	Then, we apply H\"older's inequality to obtain, for all $t\in[s,T]$, that
	\begin{align*}
		\E\|X_{t \wedge \upsilon^{s,x,\dl}_R}^{s,x,\iota} - X_{t \wedge \upsilon^{s,x,\dl}_R}^{s,x,\dl} \|^{p} &\leq
		K \int ^{t}_s (1+L_d^{\frac{p}{2}})\E\|X_{r\wedge \upsilon^{s,x,\dl}_R}^{s,x,\iota}-X^{s,x,\dl}_{r\wedge \upsilon^{s,x,\dl}_R}\|^{p} \rd r \\
		&\quad +K \int ^{t}_s  (L_d+L_d^{\frac{p}{2}})\E\|X_{r \wedge \upsilon^{s,x,\dl}_R}^{s,x,\dl}-X^{s,x,\dl}_{\max\{s,\dl(r \wedge \upsilon^{s,x,\dl}_R)\}}\|^{p} \rd r\\
		&\quad+ K  \int ^{t}_s \lf(\E\|X^{s,x,\dl}_{r \wedge \upsilon^{s,x,\dl}_R}-X^{s,x,\dl}_{\max\{s,\dl(r \wedge \upsilon^{s,x,\dl}_R)\}}\|^{p+\zeta}\rt)^{\frac{p}{p+\zeta}}\\
		&\hspace{14mm}\times\lf(\E(1+\|X^{s,x,\dl}_{r\wedge \upsilon^{s,x,\dl}_R}\|^{\frac{\chi}{2}}+\|X^{s,x,\dl}_{\max\{s,\dl(r\wedge \upsilon^{s,x,\dl}_R)\}}\|^{\frac{\chi}{2}})^{\frac{p(p+\zeta)}{\zeta}}\rt)^{\frac{\zeta}{p+\zeta}} \rd r\\
		&\quad+ K|\dl|^{\frac{p}{2}} \int^{t}_s \lf(1+\E\|X^{s,x,\dl}_{\max\{s,\dl(r\wedge \upsilon^{s,x,\dl}_R)\}}\|^{(\frac{3}{2}\chi+1)p}\rt)  \rd r.
	\end{align*}
	Consequently, we obtain that
	\begin{align*}
		\sup_{t\in[s,T]}\E\|X_{t \wedge \upsilon^{s,x,\dl}_R}^{s,x,\iota}- X_{t \wedge \upsilon^{s,x,\dl}_R}^{s,x,\dl} \|^{p} &\leq K \int ^{T}_s (1+L_d^{\frac{p}{2}})\sup_{u\in[s,r]}\E\|X_{u\wedge \upsilon^{s,x,\dl}_R}^{s,x,\iota}-X^{s,x,\dl}_{u\wedge \upsilon^{s,x,\dl}_R}\|^{p} \rd r\\
		&\quad+K(T-s)(L_d+L_d^{\frac{p}{2}})\sup_{t\in[s,T]} \E\|X_{t}^{s,x,\dl}-X^{s,x,\dl}_{\max\{s,\dl(t)\}}\|^{p}\\
		&\quad+K(T-s) (\sup_{t\in[s,T]} \E\|X_{t}^{s,x,\dl}-X^{s,x,\dl}_{\max\{s,\dl(t)\}}\|^{p+\zeta})^{\frac{p}{p+\zeta}} \\
		&\hspace{20mm}\times \big(\sup_{t\in[s,T]}\E(1+\|X^{s,x,\dl}_{t}\|^{\frac{\chi p(p+\zeta)}{2\zeta}})\big)^{\frac{\zeta}{p+\zeta}}\\
		&\quad+K(T-s)|\dl|^{\frac{p}{2}}  (1+\sup_{t\in[s,T]}\E\|X^{s,x,\dl}_{t}\|^{(\frac{3}{2}\chi+1)p} ).
	\end{align*}
	For $p\in[2,p^*]$ with $p^*= \min\{\frac{2p_0}{\chi+2}-\zeta, \frac{2p_0}{3\chi+2}, \frac{-\chi\zeta+\sqrt{\chi\zeta(\chi\zeta+8 p_0)}}{2\chi}\}$, we apply Lemmas \ref{L3} and \ref{corollary2}\,\ref{lemma5} to yield that
	\begin{align*}
		\sup_{t\in[s,T]}\E\|X_t^{s,x,\iota}- X_t^{s,x,\dl} \|^{p} &\leq K \int ^{T}_s (1+L_d^{\frac{p}{2}})\sup_{u\in[s,r]}\E\|X_{u\wedge \upsilon^{s,x,\dl}_R}^{s,x,\iota}-X^{s,x,\dl}_{u\wedge \upsilon^{s,x,\dl}_R}\|^{p} \rd r\\
		&\quad+K(T-s)|\dl|(L_d+L_d^{\frac{p}{2}})(1+L_d^{\frac{p}{2}})(1+\|x\|^{p_0}+N_d^{\frac{p_0}{2}})e^{K(1+L_d^{\frac{p_0}{2}})}\\
		&\quad+K (T-s) \lf(|\dl|(1+L_d^{\frac{p+\zeta}{2}})(1+\|x\|^{p_0}+N_d^{\frac{p_0}{2}})e^{K(1+L_d^{\frac{p_0}{2}})}\rt)^{\frac{p}{p+\zeta}} \\
		&\hspace{21mm}\times \lf((1+\|x\|^{p_0}+N_d^{\frac{p_0}{2}} )e^{K(1+L_d^{\frac{p_0}{2}})} \rt)^{\frac{\zeta}{p+\zeta}}\\
		&\quad+K(T-s)|\dl|^{\frac{p}{2}} (1+\|x\|^{p_0}+N_d^{\frac{p_0}{2}} )e^{K(1+L_d^{\frac{p_0}{2}})}\\
		&\leq K \int ^{T}_s (1+L_d^{\frac{p}{2}})\sup_{u\in[s,r]}\E\|X_{u\wedge \upsilon^{s,x,\dl}_R}^{s,x,\iota}-X^{s,x,\dl}_{u\wedge \upsilon^{s,x,\dl}_R}\|^{p} \rd r \\
		&\quad+K |\dl|^{\frac{p}{p+\zeta}}(T-s)(1+L_d^{p})(1+\|x\|^{p_0}+N_d^{\frac{p_0}{2}})e^{K(1+L_d^{\frac{p_0}{2}})}.
	\end{align*}
	Therefore, by Gr\"onwall's lemma, we yield that
	\begin{equation*}
		\sup_{t\in[s,T]}\E\|X_{t\wedge \upsilon^{s,x,\dl}_R}^{s,x,\iota}- X_{t\wedge \upsilon^{s,x,\dl}_R}^{s,x,\dl} \|^{p} \leq K |\dl|^{\frac{p}{p+\zeta}}(1+L_d^{p})(1+\|x\|^{p_0}+N_d^{\frac{p_0}{2}})e^{K(1+L_d^{\frac{p_0}{2}})}.
	\end{equation*}
	Finally, the application of Fatou's lemma completes the proof.
\end{proof}
	
\begin{proof}[\textbf{Proof of Lemma \ref{convergencescheme2}}]
	Fix $s\in [0,T]$, $x\in \R^d$, $\dl \in \Th$, and $ \ep \in(0,1)$ throughout this proof. By the definitions \eqref{tamedscheme1} and \eqref{tamedscheme2}, we obtain, for all $t \in [s,T]$, that
	\begin{align*}
		X_t^{s,x,\dl} - X_t^{s,x,\dl,\ep}&= \int_s^t \mu^\dl(X_{\max\{s,\dl(r-)\}}^{s,x,\dl})-\mu^\dl(X^{s,x,\dl,\ep}_{\max\{s,\dl(r-)\}}) \rd r\\
		&\quad+\int_s^t \sg^\dl(X_{\max\{s,\dl(r-)\}}^{s,x,\dl})-\sg^\dl(X^{s,x,\dl,\ep}_{\max\{s,\dl(r-)\}}) \rd W_r \\
		&\quad+\int_s^t \int_{A_\ep} \gm(X^{s,x,\dl}_{\max\{s,\dl(r-)\}},z)-\gm(X^{s,x,\dl,\ep}_{\max\{s,\dl(r-)\}},z) \tl{\pi}(\rd z ,\rd r)\\
		&\quad+\int_s^t \int_{\R^d \backslash A_\ep}\gm(X^{s,x,\dl}_{\max\{s,\dl(r-)\}},z) \tl{\pi}(\rd z ,\rd r).
	\end{align*}
	Applying It\^o's formula yields, for all $t \in [s,T]$, that
	\begin{align*}
		&\|X_t^{s,x,\dl} - X_t^{s,x,\dl,\ep} \|^{p}\\
		&=p\int^{t}_s\|X_{r-}^{s,x,\dl}-X^{s,x,\dl,\ep}_{r-}\|^{p-2} (X_{r-}^{s,x,\dl}-X^{s,x,\dl,\ep}_{r-}) (\mu^\dl(X_{\max\{s,\dl(r-)\}}^{s,x,\dl})-\mu^\dl(X^{s,x,\dl,\ep}_{\max\{s,\dl(r-)\}}) ) \rd r \\
		&\quad+ p\int^{t}_s\|X_{r-}^{s,x,\dl}-X^{s,x,\dl,\ep}_{r-}\|^{p-2} (X_{r-}^{s,x,\dl}-X^{s,x,\dl,\ep}_{r-})(\sg^\dl(X_{\max\{s,\dl(r-)\}}^{s,x,\dl} )-\sg^\dl(X^{s,x,\dl,\ep}_{\max\{s,\dl(r-)\}}))\rd W_r\\
		&\quad+\frac{p(p-2)}{2} \int^{t}_s \|X_{r-}^{s,x,\dl}-X^{s,x,\dl,\ep}_{r-}\|^{p-4} \|(\sg^\dl(X_{\max\{s,\dl(r-)\}}^{s,x,\dl})-\sg^\dl(X^{s,x,\dl,\ep}_{\max\{s,\dl(r-)\}}))^*(X_{r-}^{s,x,\dl}-X^{s,x,\dl,\ep}_{r-}) \|^{2}\rd r\\
		&\quad+\frac{p}{2}\int^{t}_s\|X_{r-}^{s,x,\dl}-X^{s,x,\dl,\ep}_{r-}\|^{p-2}\|\sg^\dl(X_{\max\{s,\dl(r-)\}}^{s,x,\dl})-\sg^\dl(X^{s,x,\dl,\ep}_{\max\{s,\dl(r-)\}})\|^2\rd r\\
		&\quad+p\int^t_s\int_{A_\ep}\|X_{r-}^{s,x,\dl}-X^{s,x,\dl,\ep}_{r-}\|^{p-2} (X_{r-}^{s,x,\dl}-X^{s,x,\dl,\ep}_{r-}) (\gm(X^{s,x,\dl}_{\max\{s,\dl(r-)\}},z)-\gm(X^{s,x,\dl,\ep}_{\max\{s,\dl(r-)\}},z))\tl{\pi}(\rd z,\rd r)\\
		&\quad+\int_s^t\int_{A_\ep}\big\{\|X_{r-}^{s,x,\dl}-X^{s,x,\dl,\ep}_{r-}+(\gm(X^{s,x,\dl}_{\max\{s,\dl(r-)\}},z)-\gm(X^{s,x,\dl,\ep}_{\max\{s,\dl(r-)\}},z))\|^{p}-\|X_{r-}^{s,x,\dl}-X^{s,x,\dl,\ep}_{r-}\|^{p}\\
		&\quad\quad-p\|X_{r-}^{s,x,\dl}-X^{s,x,\dl,\ep}_{r-}\|^{p-2}(X_{r-}^{s,x,\dl}-X^{s,x,\dl,\ep}_{r-})(\gm(X^{s,x,\dl}_{\max\{s,\dl(r-)\}},z)-\gm(X^{s,x,\dl,\ep}_{\max\{s,\dl(r-)\}},z))\big\}\pi(\rd z,\rd r)\\
		&\quad+p\int^t_s\int_{\R^d \backslash A_\ep}\|X_{r-}^{s,x,\dl}-X^{s,x,\dl,\ep}_{r-}\|^{p-2} (X_{r-}^{s,x,\dl}-X^{s,x,\dl,\ep}_{r-}) \gm(X^{s,x,\dl}_{\max\{s,\dl(r-)\}},z)\tl{\pi}(\rd z,\rd r)\\
		&\quad+\int_s^t\int_{\R^d \backslash A_\ep}\big\{\|X_{r-}^{s,x,\dl}-X^{s,x,\dl,\ep}_{r-}+\gm(X^{s,x,\dl}_{\max\{s,\dl(r-)\}},z)\|^{p}-\|X_{r-}^{s,x,\dl}-X^{s,x,\dl,\ep}_{r-}\|^{p}\\
		&\quad\quad-p\|X_{r-}^{s,x,\dl}-X^{s,x,\dl,\ep}_{r-}\|^{p-2}(X_{r-}^{s,x,\dl}-X^{s,x,\dl,\ep}_{r-})\gm(X^{s,x,\dl}_{\max\{s,\dl(r-)\}},z)\big\}\pi(\rd z,\rd r),
	\end{align*}
	almost surely. Then, we introduce the stopping time $\upsilon^{s,x,\dl,\ep}_R := \tau^{s,x,\dl}_R \wedge \tau^{s,x,\dl,\ep}_R$ with $R \in \N$ where
	\[
	\tau^{s,x,\dl}_R:= \inf\{t\geq s: \|X^{s,x,\dl}_t\|\geq R\}\wedge T,\quad
	\tau^{s,x,\dl,\ep}_R:= \inf\{t\geq s: \|X^{s,x,\dl,\ep}_t\|\geq R\}\wedge T .
	\] 
	Futhermore, by Cauchy-Schwarz inequality and Lemma \ref{Formula for the remainder}, the following inequality holds for all $t \in [s,T]$:
	\begin{align*}
		&\E\|X_{t\wedge \upsilon^{s,x,\dl,\ep}_R}^{s,x,\dl} - X_{t\wedge \upsilon^{s,x,\dl,\ep}_R}^{s,x,\dl,\ep} \|^{p}  \\
		&\leq
		p\E \int^{t\wedge \upsilon^{s,x,\dl,\ep}_R}_s \|X_{r}^{s,x,\dl}-X^{s,x,\dl,\ep}_{r}\|^{p-2}(X^{s,x,\dl}_r-X^{s,x,\dl,\ep}_{r})(\mu^\dl(X_{\max\{s,\dl(r)\}}^{s,x,\dl})-\mu^\dl(X^{s,x,\dl,\ep}_{\max\{s,\dl(r)\}}) )  \rd r\\
		&\quad+\frac{p(p-1)}{2}\E\int^{t\wedge \upsilon^{s,x,\dl,\ep}_R}_s\|X_{r}^{s,x,\dl}-X^{s,x,\dl,\ep}_{r}\|^{p-2}  \|\sg^\dl(X_{\max\{s,\dl(r)\}}^{s,x,\dl} )-\sg^\dl(X^{s,x,\dl,\ep}_{\max\{s,\dl(r)\}})\|^2 \rd r\\
		&\quad+K\E\int^{t\wedge \upsilon^{s,x,\dl,\ep}_R}_s\int_{A_\ep} \big\{\|X_{r}^{s,x,\dl}-X^{s,x,\dl,\ep}_{r}\|^{p-2} \|\gm(X^{s,x,\dl}_{\max\{s,\dl(r)\}},z)-\gm(X^{s,x,\dl,\ep}_{\max\{s,\dl(r)\}},z)\|^2\\
		&\quad\quad+\|\gm(X^{s,x,\dl}_{\max\{s,\dl(r)\}},z)-\gm(X^{s,x,\dl,\ep}_{\max\{s,\dl(r)\}},z)\|^{p}\big\}\nu(\rd z)\rd r\\
		&\quad+K\E\int^{t\wedge \upsilon^{s,x,\dl,\ep}_R}_s\int_{\R^d \backslash A_\ep} \big\{\|X_{r}^{s,x,\dl}-X^{s,x,\dl,\ep}_{r}\|^{p-2} \|\gm(X^{s,x,\dl}_{\max\{s,\dl(r)\}},z)\|^2\\
		&\quad\quad+\|\gm(X^{s,x,\dl}_{\max\{s,\dl(r)\}},z)\|^{p}\big\}\nu(\rd z)\rd r.
	\end{align*}
	The above can be further rewritten, for all $t\in[s,T]$, as
	\begin{align*}
		&\E\|X_{t\wedge \upsilon^{s,x,\dl,\ep}_R}^{s,x,\dl} - X_{t\wedge \upsilon^{s,x,\dl,\ep}_R}^{s,x,\dl,\ep} \|^{p} \\
		&\leq
		\frac{p}{2}\E \int ^{t\wedge \upsilon^{s,x,\dl,\ep}_R}_s \|X_{r}^{s,x,\dl}-X^{s,x,\dl,\ep}_{r}\|^{p-2}\big(2(X^{s,x,\dl}_r-X^{s,x,\dl,\ep}_{r})(\mu(X_{r}^{s,x,\dl})-\mu(X^{s,x,\dl,\ep}_{r}) )  \\
		&\quad\quad +(p_0-1) \|\sg(X_{r}^{s,x,\dl} )-\sg(X^{s,x,\dl,\ep}_{r})\|^2 \big)\rd r\\
		&\quad+ K\E \int ^{t\wedge \upsilon^{s,x,\dl,\ep}_R}_s \|X_{r}^{s,x,\dl}-X^{s,x,\dl,\ep}_{r}\|^{p-2} (X^{s,x,\dl}_r-X^{s,x,\dl,\ep}_{r})(\mu^\dl(X_{\max\{s,\dl(r)\}}^{s,x,\dl})-\mu(X_{\max\{s,\dl(r)\}}^{s,x,\dl}) ) \rd r\\
		&\quad+K\E\int^{t\wedge \upsilon^{s,x,\dl,\ep}_R}_s\|X_{r}^{s,x,\dl}-X^{s,x,\dl,\ep}_{r}\|^{p-2}  \|\sg^\dl(X_{\max\{s,\dl(r)\}}^{s,x,\dl})-\sg(X_{\max\{s,\dl(r)\}}^{s,x,\dl}) \|^2 \rd r \\
		&\quad+ K\E \int ^{t\wedge \upsilon^{s,x,\dl,\ep}_R}_s \|X_{r}^{s,x,\dl}-X^{s,x,\dl,\ep}_{r}\|^{p-2} (X^{s,x,\dl}_r-X^{s,x,\dl,\ep}_{r})(\mu(X_{\max\{s,\dl(r)\}}^{s,x,\dl})-\mu(X_{r}^{s,x,\dl}) ) \rd r\\
		&\quad+K\E\int^{t\wedge \upsilon^{s,x,\dl,\ep}_R}_s\|X_{r}^{s,x,\dl}-X^{s,x,\dl,\ep}_{r}\|^{p-2}  \|\sg(X_{\max\{s,\dl(r)\}}^{s,x,\dl})-\sg(X_{r}^{s,x,\dl})\|^2 \rd r \\
		&\quad+ K\E \int ^{t\wedge \upsilon^{s,x,\dl,\ep}_R}_s \|X_{r}^{s,x,\dl}-X^{s,x,\dl,\ep}_{r}\|^{p-2} (X^{s,x,\dl}_r-X^{s,x,\dl,\ep}_{r})(\mu(X_{r}^{s,x,\dl,\ep})-\mu(X_{\max\{s,\dl(r)\}}^{s,x,\dl,\ep}) ) \rd r\\
		&\quad+K\E\int^{t\wedge \upsilon^{s,x,\dl,\ep}_R}_s\|X_{r}^{s,x,\dl}-X^{s,x,\dl,\ep}_{r}\|^{p-2}  \|\sg(X_{r}^{s,x,\dl,\ep})-\sg(X_{\max\{s,\dl(r)\}}^{s,x,\dl,\ep})  \|^2 \rd r \\
		&\quad+ K\E \int ^{t\wedge \upsilon^{s,x,\dl,\ep}_R}_s \|X_{r}^{s,x,\dl}-X^{s,x,\dl,\ep}_{r}\|^{p-2} (X^{s,x,\dl}_r-X^{s,x,\dl,\ep}_{r})(\mu(X_{\max\{s,\dl(r)\}}^{s,x,\dl,\ep})-\mu^\dl(X_{\max\{s,\dl(r)\}}^{s,x,\dl,\ep}) ) \rd r\\
		&\quad+K\E\int^{t\wedge \upsilon^{s,x,\dl,\ep}_R}_s\|X_{r}^{s,x,\dl}-X^{s,x,\dl,\ep}_{r}\|^{p-2}  \|\sg(X_{\max\{s,\dl(r)\}}^{s,x,\dl,\ep})-\sg^\dl(X_{\max\{s,\dl(r)\}}^{s,x,\dl,\ep}) \|^2 \rd r\\
		&\quad+K\E\int^{t\wedge \upsilon^{s,x,\dl,\ep}_R}_s\int_{A_\ep} \big\{\|X_{r}^{s,x,\dl}-X^{s,x,\dl,\ep}_{r}\|^{p-2} \|\gm(X^{s,x,\dl}_{\max\{s,\dl(r)\}},z)-\gm(X^{s,x,\dl,\ep}_{\max\{s,\dl(r)\}},z)\|^2\\
		&\quad\quad+\|\gm(X^{s,x,\dl}_{\max\{s,\dl(r)\}},z)-\gm(X^{s,x,\dl,\ep}_{\max\{s,\dl(r)\}},z)\|^{p}\big\}\nu(\rd z)\rd r\\
		&\quad+K\E\int^{t\wedge \upsilon^{s,x,\dl,\ep}_R}_s\int_{\R^d \backslash A_\ep} \big\{\|X_{r}^{s,x,\dl}-X^{s,x,\dl,\ep}_{r}\|^{p-2} \|\gm(X^{s,x,\dl}_{\max\{s,\dl(r)\}},z)\|^2\\
		&\quad\quad+\|\gm(X^{s,x,\dl}_{\max\{s,\dl(r)\}},z)\|^{p}\big\}\nu(\rd z)\rd r.
	\end{align*}
	By using Assumption \ref{A1} and Young's inequality, one gets, for all $t\in[s,T]$, that
	\begin{align*}
		&\E\|X_{t\wedge \upsilon^{s,x,\dl,\ep}_R}^{s,x,\dl} - X_{t\wedge \upsilon^{s,x,\dl,\ep}_R}^{s,x,\dl,\ep} \|^{p}\\
		&\leq
		K\E \int ^{t\wedge \upsilon^{s,x,\dl,\ep}_R}_s\|X_{r}^{s,x,\dl}-X^{s,x,\dl,\ep}_{r}\|^{p}\rd r\\
		&\quad+ K\E \int ^{t\wedge \upsilon^{s,x,\dl,\ep}_R}_s \|\mu^\dl(X_{\max\{s,\dl(r)\}}^{s,x,\dl})-\mu(X_{\max\{s,\dl(r)\}}^{s,x,\dl})\|^p+\|\sg^\dl(X_{\max\{s,\dl(r)\}}^{s,x,\dl})-\sg(X_{\max\{s,\dl(r)\}}^{s,x,\dl}) \|^p \rd r\\
		&\quad+ K\E \int ^{t\wedge \upsilon^{s,x,\dl,\ep}_R}_s  \|\mu(X_{\max\{s,\dl(r)\}}^{s,x,\dl})-\mu(X_{r}^{s,x,\dl}) \|^p+\|\sg(X_{\max\{s,\dl(r)\}}^{s,x,\dl})-\sg(X_{r}^{s,x,\dl})\|^p \rd r\\
		&\quad+ K\E \int ^{t\wedge \upsilon^{s,x,\dl,\ep}_R}_s \|\mu(X_{r}^{s,x,\dl,\ep})-\mu(X_{\max\{s,\dl(r)\}}^{s,x,\dl,\ep}) \|^p+\|\sg(X_{r}^{s,x,\dl,\ep})-\sg(X_{\max\{s,\dl(r)\}}^{s,x,\dl,\ep})  \|^p \rd r\\
		&\quad+ K\E \int ^{t\wedge \upsilon^{s,x,\dl,\ep}_R}_s\|\mu(X_{\max\{s,\dl(r)\}}^{s,x,\dl,\ep})-\mu^\dl(X_{\max\{s,\dl(r)\}}^{s,x,\dl,\ep}) \|^p+\|\sg(X_{\max\{s,\dl(r)\}}^{s,x,\dl,\ep})-\sg^\dl(X_{\max\{s,\dl(r)\}}^{s,x,\dl,\ep}) \|^p\rd r\\
		&\quad+K\E\int^{t\wedge \upsilon^{s,x,\dl,\ep}_R}_s\big(\int_{A_\ep}  \|\gm(X^{s,x,\dl}_{\max\{s,\dl(r)\}},z)-\gm(X^{s,x,\dl,\ep}_{\max\{s,\dl(r)\}},z)\|^{2}\nu(\rd z)\big)^{\frac{p}{2}}\rd r\\
		&\quad+K\E\int^{t\wedge \upsilon^{s,x,\dl,\ep}_R}_s\big(\int_{\R^d \backslash A_\ep}  \|\gm(X^{s,x,\dl}_{\max\{s,\dl(r)\}},z)\|^{2}\nu(\rd z)\big)^{\frac{p}{2}}\rd r.\\
		&\quad+K\E\int^{t\wedge \upsilon^{s,x,\dl,\ep}_R}_s\int_{A_\ep}  \|\gm(X^{s,x,\dl}_{\max\{s,\dl(r)\}},z)-\gm(X^{s,x,\dl,\ep}_{\max\{s,\dl(r)\}},z)\|^{p}+\|\gm(X^{s,x,\dl}_{\max\{s,\dl(r)\}},z)\|^{p}\nu(\rd z)\rd r\\
		&\quad+K\E\int^{t\wedge \upsilon^{s,x,\dl,\ep}_R}_s\int_{\R^d \backslash A_\ep}  \|\gm(X^{s,x,\dl}_{\max\{s,\dl(r)\}},z)\|^{p}\nu(\rd z)\rd r,
	\end{align*}
	which on the application of Assumptions \ref{A2}, \ref{A4}, \ref{A5}, and Remark \ref{remark3} yields, for all $t\in[s,T]$, that
	\begin{align*}
		&\E\|X_{t\wedge \upsilon^{s,x,\dl,\ep}_R}^{s,x,\dl} - X_{t\wedge \upsilon^{s,x,\dl,\ep}_R}^{s,x,\dl,\ep} \|^{p} \\
		&\leq
		K\E \int ^{t\wedge \upsilon^{s,x,\dl,\ep}_R}_s  \|X_{r}^{s,x,\dl}-X^{s,x,\dl,\ep}_{r}\|^{p} \rd r +K (L_d+L_d^{\frac{p}{2}}) \E\int ^{t\wedge \upsilon^{s,x,\dl,\ep}_R}_s \|X_{\max\{s,\dl(r)\}}^{s,x,\dl}-X^{s,x,\dl,\ep}_{\max\{s,\dl(r)\}}\|^{p} \rd r\\
		&\quad+ K\E \int ^{t\wedge \upsilon^{s,x,\dl,\ep}_R}_s \|X^{s,x,\dl}_{r}-X^{s,x,\dl}_{\max\{s,\dl(r)\}}\|^p(1+\|X^{s,x,\dl}_{r}\|^{\frac{\chi}{2}}+\|X^{s,x,\dl}_{\max\{s,\dl(r)\}}\|^{\frac{\chi}{2}})^p \rd r\\
		&\quad+ K\E \int ^{t\wedge \upsilon^{s,x,\dl,\ep}_R}_s \|X^{s,x,\dl,\ep}_{r}-X^{s,x,\dl,\ep}_{\max\{s,\dl(r)\}}\|^p(1+\|X^{s,x,\dl,\ep}_{r}\|^{\frac{\chi}{2}}+\|X^{s,x,\dl,\ep}_{\max\{s,\dl(r)\}}\|^{\frac{\chi}{2}})^p \rd r\\
		&\quad+ K|\dl|^{\frac{p}{2}}\E \int ^{t\wedge \upsilon^{s,x,\dl,\ep}_R}_s 1+\|X^{s,x,\dl}_{\max\{s,\dl(r)\}}\|^{(\frac{3}{2}\chi+1)p} +\|X^{s,x,\dl,\ep}_{\max\{s,\dl(r)\}}\|^{(\frac{3}{2}\chi+1)p}  \rd r\\
		&\quad+ K\E\int ^{t\wedge \upsilon^{s,x,\dl,\ep}_R}_s ((\ep^{q}\mathcal{Z}_d)^{\frac{p}{2}}+\ep^{q}\mathcal{Z}_d)(1+\|X^{s,x,\dl}_{\max\{s,\dl(r)\}}\|^p) \rd r.
	\end{align*}
	Then, we apply H\"older's inequality to obtain, for all $t\in[s,T]$, that
	\begin{align*}
		&\E\|X_{t\wedge \upsilon^{s,x,\dl,\ep}_R}^{s,x,\dl} - X_{t\wedge \upsilon^{s,x,\dl,\ep}_R}^{s,x,\dl,\ep} \|^{p} \\
		&\leq
		K \int ^{t}_s \E\|X_{r\wedge \upsilon^{s,x,\dl,\ep}_R}^{s,x,\dl}-X^{s,x,\dl,\ep}_{r\wedge \upsilon^{s,x,\dl,\ep}_R}\|^{p} \rd r\\
		&\quad+K(L_d+L_d^{\frac{p}{2}}) \int ^{t}_s \E\|X_{\max\{s,\dl(r\wedge \upsilon^{s,x,\dl,\ep}_R)\}}^{s,x,\dl}-X^{s,x,\dl,\ep}_{\max\{s,\dl(r\wedge \upsilon^{s,x,\dl,\ep}_R)\}}\|^{p} \rd r\\
		&\quad+ K  \int ^{t}_s (\E\|X^{s,x,\dl}_{r\wedge \upsilon^{s,x,\dl,\ep}_R}-X^{s,x,\dl}_{\max\{s,\dl(r\wedge \upsilon^{s,x,\dl,\ep}_R)\}}\|^{p+\zeta})^{\frac{p}{p+\zeta}}\\
		&\hspace{17mm}\times(\E(1+\|X^{s,x,\dl}_{r\wedge \upsilon^{s,x,\dl,\ep}_R}\|^{\frac{\chi}{2}}+\|X^{s,x,\dl}_{\max\{s,\dl(r\wedge \upsilon^{s,x,\dl,\ep}_R)\}}\|^{\frac{\chi}{2}})^{\frac{p(p+\zeta)}{\zeta}})^{\frac{\zeta}{p+\zeta}} \rd r\\
		&\quad+ K  \int ^{t}_s (\E\|X^{s,x,\dl,\ep}_{r\wedge \upsilon^{s,x,\dl,\ep}_R}-X^{s,x,\dl,\ep}_{\max\{s,\dl(r\wedge \upsilon^{s,x,\dl,\ep}_R)\}}\|^{p+\zeta})^{\frac{p}{p+\zeta}}\\
		&\hspace{17mm}\times(\E(1+\|X^{s,x,\dl,\ep}_{r\wedge \upsilon^{s,x,\dl,\ep}_R}\|^{\frac{\chi}{2}}+\|X^{s,x,\dl,\ep}_{\max\{s,\dl(r\wedge \upsilon^{s,x,\dl,\ep}_R)\}}\|^{\frac{\chi}{2}})^{\frac{p(p+\zeta)}{\zeta}})^{\frac{\zeta}{p+\zeta}} \rd r\\
		&\quad+ K|\dl|^{\frac{p}{2}} \int ^{t}_s 1+\E\|X^{s,x,\dl}_{\max\{s,\dl(r\wedge \upsilon^{s,x,\dl,\ep}_R)\}}\|^{(\frac{3}{2}\chi+1)p} +\E\|X^{s,x,\dl,\ep}_{\max\{s,\dl(r\wedge \upsilon^{s,x,\dl,\ep}_R)\}}\|^{(\frac{3}{2}\chi+1)p}  \rd r\\
		&\quad+ K((\ep^{q}\mathcal{Z}_d)^{\frac{p}{2}}+\ep^{q}\mathcal{Z}_d)\int ^{t}_s (1+\E\|X^{s,x,\dl}_{\max\{s,\dl(r\wedge \upsilon^{s,x,\dl,\ep}_R)\}}\|^p) \rd r,
	\end{align*}
	which is followed by 
	\begin{align*}
		&\sup_{t\in[s,T]}	\E\|X_{t\wedge \upsilon^{s,x,\dl,\ep}_R}^{s,x,\dl} - X_{t\wedge \upsilon^{s,x,\dl,\ep}_R}^{s,x,\dl,\ep} \|^{p}\\
		& \leq
		K(1+L_d^{\frac{p}{2}}) \int ^{T}_s \sup_{u\in[s,r]} \E\|X_{u\wedge \upsilon^{s,x,\dl,\ep}_R}^{s,x,\dl}-X^{s,x,\dl,\ep}_{u\wedge \upsilon^{s,x,\dl,\ep}_R}\|^{p}\rd r\\
		&\quad+ K  (T-s) (\sup_{t\in[s,T]}\E\|X^{s,x,\dl}_{t}-X^{s,x,\dl}_{\max\{s,\dl(t)\}}\|^{p+\zeta})^{\frac{p}{p+\zeta}}(\sup_{t\in[s,T]}\E(1+\|X^{s,x,\dl}_{t}\|^{\frac{\chi}{2}})^{\frac{p(p+\zeta)}{\zeta}})^{\frac{\zeta}{p+\zeta}} \\
		&\quad+ K  (T-s) (\sup_{t\in[s,T]}\E\|X^{s,x,\dl,\ep}_{t}-X^{s,x,\dl,\ep}_{\max\{s,\dl(t)\}}\|^{p+\zeta})^{\frac{p}{p+\zeta}}(\sup_{t\in[s,T]}\E(1+\|X^{s,x,\dl,\ep}_{t}\|^{\frac{\chi}{2}})^{\frac{p(p+\zeta)}{\zeta}})^{\frac{\zeta}{p+\zeta}} \\
		&\quad+ K|\dl|^{\frac{p}{2}}(T-s) ( 1+\sup_{t\in[s,T]}\E\|X^{s,x,\dl}_{t}\|^{(\frac{3}{2}\chi+1)p} +\sup_{t\in[s,T]}\E\|X^{s,x,\dl,\ep}_{t}\|^{(\frac{3}{2}\chi+1)p}  )\rd r\\
		&\quad+ K((\ep^{q}\mathcal{Z}_d)^{\frac{p}{2}}+\ep^{q}\mathcal{Z}_d)(T-s) (1+\sup_{t\in[s,T]}\E\|X^{s,x,\dl}_{t}\|^p).
	\end{align*}
	For $p\in[2,p^*]$ with $p^*= \min\{\frac{2p_0}{\chi+2}-\zeta, \frac{2p_0}{3\chi+2}, \frac{-\chi\zeta+\sqrt{\chi\zeta(\chi\zeta+8 p_0)}}{2\chi}\}$, we use Lemmas \ref{L3}, \ref{epbounds}, and \ref{corollary2}\,\ref{lemma5}--\ref{corollary2(ii)} to obtain that
	\begin{align*}
		&\sup_{t\in[s,T]}	\E\|X_{t\wedge \upsilon^{s,x,\dl,\ep}_R}^{s,x,\dl} - X_{t\wedge \upsilon^{s,x,\dl,\ep}_R}^{s,x,\dl,\ep} \|^{p}\\
		&\leq
		K (1+L_d^{\frac{p}{2}})\int ^{T}_s \sup_{u\in[s,r]} \E\|X_{u\wedge \upsilon^{s,x,\dl,\ep}_R}^{s,x,\dl}-X^{s,x,\dl,\ep}_{u\wedge \upsilon^{s,x,\dl,\ep}_R}\|^{p}\rd r\\
		&\quad+K (T-s) \lf(|\dl|(1+L_d^{\frac{p+\zeta}{2}})(1+\|x\|^{p_0}+N_d^{\frac{p_0}{2}})e^{K(1+L_d^{\frac{p_0}{2}})}\rt)^{\frac{p}{p+\zeta}} \\
		&\hspace{23mm} \lf((1+\|x\|^{p_0}+N_d^{\frac{p_0}{2}} )e^{K(1+L_d^{\frac{p_0}{2}})} \rt)^{\frac{\zeta}{p+\zeta}}\\
		&\quad+K(T-s)|\dl|^{\frac{p}{2}} (1+\|x\|^{p_0}+N_d^{\frac{p_0}{2}} )e^{K(1+L_d^{\frac{p_0}{2}})}\\
		&\quad+ K((\ep^{q}\mathcal{Z}_d)^{\frac{p}{2}}+\ep^{q}\mathcal{Z}_d)(T-s) (1+\|x\|^{p_0}+N_d^{\frac{p_0}{2}} )e^{K(1+L_d^{\frac{p_0}{2}})}\\
		&\leq K \int ^{T}_s (1+L_d^{\frac{p}{2}})\sup_{u\in[s,r]}\E\|X_{u\wedge \upsilon^{s,x,\dl,\ep}_R}^{s,x,\iota}-X^{s,x,\dl,\ep}_{u\wedge \upsilon^{s,x,\dl}_R}\|^{p} \rd r \\
		&\quad+K (|\dl|^{\frac{p}{p+\zeta}}+(\ep^{q}\mathcal{Z}_d)^{\frac{p}{2}}+\ep^{q}\mathcal{Z}_d) )(T-s)(1+L_d^{\frac{p}{2}})(1+\|x\|^{p_0}+N_d^{\frac{p_0}{2}})e^{K(1+L_d^{\frac{p_0}{2}})}.
	\end{align*}
	Then, we apply Gr\"onwall's lemma to obtain that,
	\begin{align*}
		&	\sup_{t\in[s,T]}	\E\|X_{t\wedge \upsilon^{s,x,\dl,\ep}_R}^{s,x,\dl} - X_{t\wedge \upsilon^{s,x,\dl,\ep}_R}^{s,x,\dl,\ep} \|^{p} \\
		&\leq	K \big(|\dl|^{\frac{p}{p+\zeta}}+\ep^{q}(1+\mathcal{Z}_d^{\frac{p}{2}})\big)(1+L_d^{\frac{p}{2}})(1+\|x\|^{p_0}+N_d^{\frac{p_0}{2}})e^{K(1+L_d^{\frac{p_0}{2}})}.
	\end{align*}
	Finally, the application of Fatou's lemma completes the proof.
\end{proof}
	
\begin{proof}[\textbf{Proof of Lemma \ref{convergencescheme3}}]
	Fix $s\in [0,T]$, $x\in \R^d$, $\dl \in \Th$, $\ep \in (0,1)$, and $\mathcal{M} \in \N$ with $\mathcal{M}\geq (1 \vee \ep^{-2}\mathcal{Z}_d)^{2}$ throughout this proof. Comparing \eqref{tamedscheme2} and \eqref{Tamedscheme2} yields, for all $t \in [s,T]$, that
	\begin{align*}
		X_t^{s,x,\dl,\ep} - X_t^{s,x,\dl,\ep,\mathcal{M}}&= \int_s^t \mu^\dl(X_{\max\{s,\dl(r-)\}}^{s,x,\dl,\ep})-\mu^\dl(X^{s,x,\dl,\ep,\mathcal{M}}_{\max\{s,\dl(r-)\}}) \rd r\\
		&\quad+\int_s^t \sg^\dl(X_{\max\{s,\dl(r-)\}}^{s,x,\dl,\ep})-\sg^\dl(X^{s,x,\dl,\ep,\mathcal{M}}_{\max\{s,\dl(r-)\}}) \rd W_r \\
		&\quad+\int_s^t \int_{A_\ep} \gm(X^{s,x,\dl,\ep}_{\max\{s,\dl(r-)\}},z)-\gm(X^{s,x,\dl,\ep,\mathcal{M}}_{\max\{s,\dl(r-)\}},z) \tl{\pi}(\rd z ,\rd r)\\
		&\quad+\int^{t}_{s}  \lf\{\frac{\nu(A_\ep)}{\mathcal{M}}\sum_{i=1}^{\mathcal{M}}\gm\lf(X_{\max\{s,\dl(r-)\}}^{s,x,\dl,\ep,\mathcal{M}},V^{\dl,\ep,\mathcal{M}}_{i,\max\{s,\dl(r-)\}}\rt)\rt\}\\
		&\quad\quad\quad\quad -\int_{A_\ep} \gm(X^{s,x,\dl,\ep,\mathcal{M}}_{\max\{s,\dl(r-)\}},z)\nu(\rd z)\rd r.
	\end{align*}
  Define a stopping time $\upsilon^{s,x,\dl,\ep,\mathcal{M}}_R$ := $\tau^{s,x,\dl,\ep}_R \wedge \tau^{s,x,\dl,\ep,\mathcal{M}}_R$ with $R \in \N$ where 
	\[
	\tau^{s,x,\dl,\ep}_R:= \inf\{t\geq s: \|X^{s,x,\dl,\ep}_t\|\geq R\}\wedge T, \qquad\tau^{s,x,\dl,\ep,\mathcal{M}}_R:= \inf\{t\geq s: \|X^{s,x,\dl,\ep,\mathcal{M}}_t\|\geq R\}\wedge T.\] 
	By It\^o's formula, Cauchy-Schwarz inequality, and Lemma \ref{Formula for the remainder}, the following inequality holds for all $t \in [s,T]$:
	\begin{align*}
		&\E\|X_{t \wedge \upsilon^{s,x,\dl,\ep,\mathcal{M}}_R}^{s,x,\dl,\ep} - X_{t \wedge \upsilon^{s,x,\dl,\ep,\mathcal{M}}_R}^{s,x,\dl,\ep,\mathcal{M}} \|^{p}  \\&\leq
		p\E \int^{t \wedge \upsilon^{s,x,\dl,\ep,\mathcal{M}}_R}_s \|X_{r}^{s,x,\dl,\ep}-X^{s,x,\dl,\ep,\mathcal{M}}_{r}\|^{p-2}(X^{s,x,\dl,\ep}_r-X^{s,x,\dl,\ep,\mathcal{M}}_{r})\\
		&\hspace{35mm}\times(\mu^\dl(X_{\max\{s,\dl(r)\}}^{s,x,\dl,\ep})-\mu^\dl(X^{s,x,\dl,\ep,\mathcal{M}}_{\max\{s,\dl(r)\}}) )  \rd r\\
		&\quad+\frac{p(p-1)}{2}\E\int^{t \wedge \upsilon^{s,x,\dl,\ep,\mathcal{M}}_R}_s\|X_{r}^{s,x,\dl,\ep}-X^{s,x,\dl,\ep,\mathcal{M}}_{r}\|^{p-2}  \|\sg^\dl(X_{\max\{s,\dl(r)\}}^{s,x,\dl,\ep} )-\sg^\dl(X^{s,x,\dl,\ep,\mathcal{M}}_{\max\{s,\dl(r)\}})\|^2 \rd r\\
		&\quad+K\E\int^{t \wedge \upsilon^{s,x,\dl,\ep,\mathcal{M}}_R}_s\int_{A_\ep} \big\{\|X_{r}^{s,x,\dl,\ep}-X^{s,x,\dl,\ep,\mathcal{M}}_{r}\|^{p-2} \|\gm(X^{s,x,\dl,\ep}_{\max\{s,\dl(r)\}},z)-\gm(X^{s,x,\dl,\ep,\mathcal{M}}_{\max\{s,\dl(r)\}},z)\|^2\\
		&\quad+\|\gm(X^{s,x,\dl,\ep}_{\max\{s,\dl(r)\}},z)-\gm(X^{s,x,\dl,\ep,\mathcal{M}}_{\max\{s,\dl(r)\}},z)\|^{p}\big\}\nu(\rd z)\rd r\\
		&\quad+p\E\int^{t \wedge \upsilon^{s,x,\dl,\ep,\mathcal{M}}_R}_s\|X_{r}^{s,x,\dl,\ep}-X^{s,x,\dl,\ep,\mathcal{M}}_{r}\|^{p-2} (X_{r}^{s,x,\dl,\ep}-X^{s,x,\dl,\ep,\mathcal{M}}_{r})\\ 
		&\quad\quad\lf(\lf\{\frac{\nu(A_\ep)}{\mathcal{M}}\sum_{i=1}^{\mathcal{M}}\gm\lf(X_{\max\{s,\dl(r)\}}^{s,x,\dl,\ep,\mathcal{M}},V^{\dl,\ep,\mathcal{M}}_{i,\max\{s,\dl(r)\}}\rt)\rt\}-\int_{A_\ep} \gm(X^{s,x,\dl,\ep,\mathcal{M}}_{\max\{s,\dl(r)\}},z)\nu(\rd z)\rt) \rd r.
	\end{align*}
	By using Assumption \ref{A1} and Young's inequality, we have, for all $t\in[s,T]$, that
	\begin{align*}
		&\E\|X_{t \wedge \upsilon^{s,x,\dl,\ep,\mathcal{M}}_R}^{s,x,\dl,\ep} - X_{t \wedge \upsilon^{s,x,\dl,\ep,\mathcal{M}}_R}^{s,x,\dl,\ep,\mathcal{M}} \|^{p} \\
		&\leq
		K\E \int^{t \wedge \upsilon^{s,x,\dl,\ep,\mathcal{M}}_R}_s \|X_{r}^{s,x,\dl,\ep}-X^{s,x,\dl,\ep,\mathcal{M}}_{r}\|^{p}\rd r\\
		&\quad+ K\E \int^{t \wedge \upsilon^{s,x,\dl,\ep,\mathcal{M}}_R}_s \|\mu^\dl(X_{\max\{s,\dl(r)\}}^{s,x,\dl,\ep})-\mu(X_{\max\{s,\dl(r)\}}^{s,x,\dl,\ep})\|^p +\|\sg^\dl(X_{\max\{s,\dl(r)\}}^{s,x,\dl,\ep})-\sg(X_{\max\{s,\dl(r)\}}^{s,x,\dl,\ep}) \|^p\rd r\\
		&\quad+ K\E \int^{t \wedge \upsilon^{s,x,\dl,\ep,\mathcal{M}}_R}_s  \|\mu(X_{\max\{s,\dl(r)\}}^{s,x,\dl,\ep})-\mu(X_{r}^{s,x,\dl,\ep}) \|^p+\|\sg(X_{\max\{s,\dl(r)\}}^{s,x,\dl,\ep})-\sg(X_{r}^{s,x,\dl,\ep})\|^p \rd r\\
		&\quad+ K\E \int^{t \wedge \upsilon^{s,x,\dl,\ep,\mathcal{M}}_R}_s \|\mu(X_{r}^{s,x,\dl,\ep,\mathcal{M}})-\mu(X_{\max\{s,\dl(r)\}}^{s,x,\dl,\ep,\mathcal{M}}) \|^p +\|\sg(X_{r}^{s,x,\dl,\ep,\mathcal{M}})-\sg(X_{\max\{s,\dl(r)\}}^{s,x,\dl,\ep,\mathcal{M}})  \|^p \rd r\\
		&\quad+ K\E \int^{t \wedge \upsilon^{s,x,\dl,\ep,\mathcal{M}}_R}_s\|\mu(X_{\max\{s,\dl(r)\}}^{s,x,\dl,\ep,\mathcal{M}})-\mu^\dl(X_{\max\{s,\dl(r)\}}^{s,x,\dl,\ep,\mathcal{M}}) \|^p+\|\sg(X_{\max\{s,\dl(r)\}}^{s,x,\dl,\ep,\mathcal{M}})-\sg^\dl(X_{\max\{s,\dl(r)\}}^{s,x,\dl,\ep,\mathcal{M}}) \|^p\rd r\\
		&\quad+K\E\int^{t \wedge \upsilon^{s,x,\dl,\ep,\mathcal{M}}_R}_s\big(\int_{A_\ep}  \|\gm(X^{s,x,\dl,\ep}_{\max\{s,\dl(r)\}},z)-\gm(X^{s,x,\dl,\ep,\mathcal{M}}_{\max\{s,\dl(r)\}},z)\|^{2}\nu(\rd z)\big)^{\frac{p}{2}}\rd r\\
		&\quad+K\E\int^{t \wedge \upsilon^{s,x,\dl,\ep,\mathcal{M}}_R}_s\int_{A_\ep}  \|\gm(X^{s,x,\dl,\ep}_{\max\{s,\dl(r)\}},z)-\gm(X^{s,x,\dl,\ep,\mathcal{M}}_{\max\{s,\dl(r)\}},z)\|^{p}\nu(\rd z)\rd r\\
		&+K\E\int^{t \wedge \upsilon^{s,x,\dl,\ep,\mathcal{M}}_R}_s \lf\|\lf\{\frac{\nu(A_\ep)}{\mathcal{M}}\sum_{i=1}^{\mathcal{M}}\gm\lf(X_{\max\{s,\dl(r-)\}}^{s,x,\dl,\ep,\mathcal{M}},V^{\dl,\ep,\mathcal{M}}_{i,\max\{s,\dl(r)\}}\rt)\rt\}-\int_{A_\ep} \gm(X^{s,x,\dl,\ep,\mathcal{M}}_{\max\{s,\dl(r)\}},z)\nu(\rd z)\rt\|^p \rd r,
	\end{align*}
	which on the application of Assumptions \ref{A2}, \ref{A4}, Remark \ref{remark3}, and \eqref{L3A_4(t)} (together with $\mathcal{M}\ge (1 \vee \ep^{-2}\mathcal{Z}_d)^{2}$) gives, for all $t\in[s,T]$, that
	\begin{align*}
		&\E\|X_{t \wedge \upsilon^{s,x,\dl,\ep,\mathcal{M}}_R}^{s,x,\dl,\ep} - X_{t \wedge \upsilon^{s,x,\dl,\ep,\mathcal{M}}_R}^{s,x,\dl,\ep,\mathcal{M}} \|^{p} \\
		&\leq
		K \E \int ^{t \wedge \upsilon^{s,x,\dl,\ep,\mathcal{M}}_R}_s \|X_{r}^{s,x,\dl,\ep}-X^{s,x,\dl,\ep,\mathcal{M}}_{r}\|^{p} \rd r \\
		&\quad+K \E \int ^{t \wedge \upsilon^{s,x,\dl,\ep,\mathcal{M}}_R}_s (L_d+L_d^{\frac{p}{2}})\|X_{\max\{s,\dl(r)\}}^{s,x,\dl,\ep}-X^{s,x,\dl,\ep,\mathcal{M}}_{\max\{s,\dl(r)\}}\|^{p} \rd r\\
		&\quad+ K\E \int^{t \wedge \upsilon^{s,x,\dl,\ep,\mathcal{M}}_R}_s \|X^{s,x,\dl,\ep}_{r}-X^{s,x,\dl,\ep}_{\max\{s,\dl(r)\}}\|^p(1+\|X^{s,x,\dl,\ep}_{r}\|^{\frac{\chi}{2}}+\|X^{s,x,\dl,\ep}_{\max\{s,\dl(r)\}}\|^{\frac{\chi}{2}})^p \rd r\\
		&\quad+ K\E \int^{t \wedge \upsilon^{s,x,\dl,\ep,\mathcal{M}}_R}_s \|X^{s,x,\dl,\ep,\mathcal{M}}_{r}-X^{s,x,\dl,\ep,\mathcal{M}}_{\max\{s,\dl(r)\}}\|^p(1+\|X^{s,x,\dl,\ep,\mathcal{M}}_{r}\|^{\frac{\chi}{2}}+\|X^{s,x,\dl,\ep,\mathcal{M}}_{\max\{s,\dl(r)\}}\|^{\frac{\chi}{2}})^p \rd r\\
		&\quad+ K|\dl|^{\frac{p}{2}}\E \int^{t \wedge \upsilon^{s,x,\dl,\ep,\mathcal{M}}_R}_s 1+\|X^{s,x,\dl,\ep}_{\max\{s,\dl(r)\}}\|^{(\frac{3}{2}\chi+1)p} +\|X^{s,x,\dl,\ep,\mathcal{M}}_{\max\{s,\dl(r)\}}\|^{(\frac{3}{2}\chi+1)p}  \rd r\\
		&\quad+ K(\ep^{-2}\mathcal{Z}_d)^{p-1}\mathcal{M}^{-\frac{p}{2}}\E\int^{t \wedge \upsilon^{s,x,\dl,\ep,\mathcal{M}}_R}_s  N_d+L_d\|X_{\max\{s,\dl(r)\}}^{s,x,\dl,\ep,\mathcal{M}}\|^{p}\rd r.
	\end{align*}
	Then, we apply H\"older's inequality to get, for all $t\in[s,T]$, that
	\begin{align*}
		&\E\|X_{t \wedge \upsilon^{s,x,\dl,\ep,\mathcal{M}}_R}^{s,x,\dl,\ep} - X_{t \wedge \upsilon^{s,x,\dl,\ep,\mathcal{M}}_R}^{s,x,\dl,\ep,\mathcal{M}} \|^{p} \\
		&\leq
		K \int ^{t}_s \E\|X_{r\wedge \upsilon^{s,x,\dl,\ep,\mathcal{M}}_R}^{s,x,\dl,\ep}-X^{s,x,\dl,\ep,\mathcal{M}}_{r\wedge \upsilon^{s,x,\dl,\ep,\mathcal{M}}_R}\|^{p} \rd r\\
		&\quad+K \int ^{t}_s (L_d+L_d^{\frac{p}{2}})\E\|X_{\max\{s,\dl(r\wedge \upsilon^{s,x,\dl,\ep,\mathcal{M}}_R)\}}^{s,x,\dl,\ep}-X^{s,x,\dl,\ep,\mathcal{M}}_{\max\{s,\dl(r\wedge \upsilon^{s,x,\dl,\ep,\mathcal{M}}_R)\}}\|^{p} \rd r\\
		&\quad+ K  \int ^{t}_s (\E\|X^{s,x,\dl,\ep}_{r\wedge \upsilon^{s,x,\dl,\ep,\mathcal{M}}_R}-X^{s,x,\dl,\ep}_{\max\{s,\dl(r\wedge \upsilon^{s,x,\dl,\ep,\mathcal{M}}_R)\}}\|^{p+\zeta})^{\frac{p}{p+\zeta}}\\
		&\hspace{15mm}\times(\E(1+\|X^{s,x,\dl,\ep}_{r\wedge \upsilon^{s,x,\dl,\ep,\mathcal{M}}_R}\|^{\frac{\chi}{2}}+\|X^{s,x,\dl,\ep}_{\max\{s,\dl(r\wedge \upsilon^{s,x,\dl,\ep,\mathcal{M}}_R)\}}\|^{\frac{\chi}{2}})^{\frac{p(p+\zeta)}{\zeta}})^{\frac{\zeta}{p+\zeta}} \rd r\\
		&\quad+ K  \int^{t}_s (\E\|X^{s,x,\dl,\ep,\mathcal{M}}_{r \wedge \upsilon^{s,x,\dl,\ep,\mathcal{M}}_R}-X^{s,x,\dl,\ep,\mathcal{M}}_{\max\{s,\dl(r \wedge \upsilon^{s,x,\dl,\ep,\mathcal{M}}_R)\}}\|^{p+\zeta})^{\frac{p}{p+\zeta}}\\
		&\hspace{15mm}\times(\E(1+\|X^{s,x,\dl,\ep,\mathcal{M}}_{r \wedge \upsilon^{s,x,\dl,\ep,\mathcal{M}}_R}\|^{\frac{\chi}{2}}+\|X^{s,x,\dl,\ep,\mathcal{M}}_{\max\{s,\dl(r \wedge \upsilon^{s,x,\dl,\ep,\mathcal{M}}_R)\}}\|^{\frac{\chi}{2}})^{\frac{p(p+\zeta)}{\zeta}})^{\frac{\zeta}{p+\zeta}} \rd r\\
		&\quad+ K|\dl|^{\frac{p}{2}}\E \int^{t}_s 1+\|X^{s,x,\dl,\ep}_{\max\{s,\dl(r \wedge \upsilon^{s,x,\dl,\ep,\mathcal{M}}_R)\}}\|^{(\frac{3}{2}\chi+1)p} +\|X^{s,x,\dl,\ep,\mathcal{M}}_{\max\{s,\dl(r \wedge \upsilon^{s,x,\dl,\ep,\mathcal{M}}_R)\}}\|^{(\frac{3}{2}\chi+1)p}  \rd r\\
		&\quad+ K(\ep^{-2}\mathcal{Z}_d)^{p-1}\mathcal{M}^{-\frac{p}{2}}\int^{t}_s  N_d+L_d\E\|X_{\max\{s,\dl(r \wedge \upsilon^{s,x,\dl,\ep,\mathcal{M}}_R)\}}^{s,x,\dl,\ep,\mathcal{M}}\|^{p} \rd r,
	\end{align*}
	which is followed by 
	\begin{align*}
		&\sup_{t\in[s,T]}	\E\|X_{t \wedge \upsilon^{s,x,\dl,\ep,\mathcal{M}}_R}^{s,x,\dl,\ep} - X_{t \wedge \upsilon^{s,x,\dl,\ep,\mathcal{M}}_R}^{s,x,\dl,\ep,\mathcal{M}} \|^{p} \\
		&\leq
		K \int ^{T}_s(1+L_d^{\frac{p}{2}}) \sup_{u\in[s,r]} \E\|X_{u \wedge \upsilon^{s,x,\dl,\ep,\mathcal{M}}_R}^{s,x,\dl,\ep}-X^{s,x,\dl,\ep,\mathcal{M}}_{u \wedge \upsilon^{s,x,\dl,\ep,\mathcal{M}}_R}\|^{p}\rd r\\\
		&\quad+ K  (T-s) (\sup_{t\in[s,T]}\E\|X^{s,x,\dl,\ep}_{t}-X^{s,x,\dl,\ep}_{\max\{s,\dl(t)\}}\|^{p+\zeta})^{\frac{p}{p+\zeta}}(\sup_{t\in[s,T]}\E(1+\|X^{s,x,\dl,\ep}_{t}\|^{\frac{\chi}{2}})^{\frac{p(p+\zeta)}{\zeta}})^{\frac{\zeta}{p+\zeta}} \\
		&\quad+ K  (T-s) (\sup_{t\in[s,T]}\E\|X^{s,x,\dl,\ep,\mathcal{M}}_{t}-X^{s,x,\dl,\ep,\mathcal{M}}_{\max\{s,\dl(t)\}}\|^{p+\zeta})^{\frac{p}{p+\zeta}}(\sup_{t\in[s,T]}\E(1+\|X^{s,x,\dl,\ep,\mathcal{M}}_{t}\|^{\frac{\chi}{2}})^{\frac{p(p+\zeta)}{\zeta}})^{\frac{\zeta}{p+\zeta}} \\
		&\quad+ K|\dl|^{\frac{p}{2}}(T-s) ( 1+\sup_{t\in[s,T]}\E\|X^{s,x,\dl,\ep}_{t}\|^{(\frac{3}{2}\chi+1)p} +\sup_{t\in[s,T]}\E\|X^{s,x,\dl,\ep,\mathcal{M}}_{t}\|^{(\frac{3}{2}\chi+1)p}  )\\
		&\quad+ K(\ep^{-2}\mathcal{Z}_d)^{p-1}\mathcal{M}^{-\frac{p}{2}}(T-s) ( N_d+L_d\sup_{t\in[s,T]}\E\|X_{\max\{s,\dl(r)\}}^{s,x,\dl,\ep,\mathcal{M}}\|^p).
	\end{align*}
	For $p\in[2,p^*]$ with $p^*= \min\{\frac{2p_0}{\chi+2}-\zeta, \frac{2p_0}{3\chi+2}, \frac{-\chi\zeta+\sqrt{\chi\zeta(\chi\zeta+8 p_0)}}{2\chi}\}$, we apply Lemmas \ref{epbounds}, \ref{scheme3bounds}, and \ref{corollary2}\,\ref{corollary2(ii)}--\ref{lemma9} to obtain that
	\begin{align*}
		&\sup_{t\in[s,T]}	\E\|X_{t \wedge \upsilon^{s,x,\dl,\ep,\mathcal{M}}_R}^{s,x,\dl,\ep} - X_{t \wedge \upsilon^{s,x,\dl,\ep,\mathcal{M}}_R}^{s,x,\dl,\ep,\mathcal{M}} \|^{p}\\
		&\leq K \int ^{T}_s (1+L_d^{\frac{p}{2}}) \sup_{u\in[s,r]} \E\|X_{u}^{s,x,\dl,\ep}-X^{s,x,\dl,\ep,\mathcal{M}}_{u}\|^{p}\\
		&\quad +K(|\dl|^{\frac{p}{p+\zeta}}+(\ep^{-2}\mathcal{Z}_d)^{p-1}\mathcal{M}^{-\frac{p}{2}}(T-s)(1+L_d^{\frac{p}{2}}) \big(1+\|x\|^{p_0}+N_d^{p_0}\big)e^{K(1+L_d^{p_0})}.
	\end{align*}
	Gr\"onwall’s lemma yields the desired uniform bound and Fatou’s lemma concludes the proof. 
\end{proof}

\subsection{Proof of results in Section \ref{section5.3}}\label{app:section5.3}
	\begin{proof}[\textbf{Proof of Lemma \ref{exactdifferenttimet}}]
	Fix $s\in [0,T]$ and $x\in \R^d$. By the definition of SDE \eqref{exactsolu}, we obtain, for all $t\in[s,T]$ and $\tl{t} \in[t,T]$, that
	\begin{align*}
		\E\|X^{s,x,\iota}_{\tl{t}} - X^{s,x,\iota}_t\|^p &\leq K\E \big\|\int^{\tl{t}}_t \mu(X^{s,x,\iota}_{r})\rd r\big\|^p + K\E\big\|\int^{\tl{t}}_t \sg(X^{s,x,\iota}_{r})\rd W_r \big\|^p\\
		&\quad+K\E\big\|\int^{\tl{t}}_t \int_{\R^d} \gm(X^{s,x,\iota}_{r},z) \tl\pi(\rd z,\rd r)\big\|^p.
	\end{align*}
	By H\"older's inequality, \cite[Theorem 7.3]{xuerongmao}, and \cite[Lemma 1]{Mikulevicius}, we obtain, for all $t\in[s,T]$ and $\tl{t} \in[t,T]$, that
	\begin{align*}
		\E\|X^{s,x,\iota}_{\tl{t}} - X^{s,x,\iota}_t\|^p &\leq K|\tl{t}-t|^{p-1}\E \int^{\tl{t}}_t \|\mu^{\dl}(X^{s,x,\iota}_{r})\|^{p}\rd r+ K|\tl{t}-t|^{\frac{p}{2}-1}\E\int^{\tl{t}}_t \|\sg^{\dl}(X^{s,x,\iota}_{r})\|^{p}\rd r \\
		&\quad+K|\tl{t}-t|^{\frac{p}{2}-1}\E\int^{\tl{t}}_t \big(\int_{\R_d} \|\gm(X^{s,x,\iota}_{r},z)-\gm(0,z)\|^2 \nu(\rd z)\big)^{\frac{p}{2}} \rd r\\
		&\quad+K|\tl{t}-t|^{\frac{p}{2}-1}\E\int^{\tl{t}}_t \big(\int_{\R_d} \|\gm(0,z)\|^2 \nu(\rd z)\big)^{\frac{p}{2}} \rd r\\
		&\quad+K\E\int^{\tl{t}}_t\int_{\R^d} \|\gm(X^{s,x,\iota}_{r},z)-\gm(0,z)\|^p \nu(\rd z) \rd r+K\E\int^{\tl{t}}_t\int_{\R^d} \|\gm(0,z)\|^p \nu(\rd z) \rd r.
	\end{align*}
	Applying Remarks \ref{remark1}, \ref{remark2}, Assumptions \ref{A3} and \ref{A4} yields, for all $t\in[s,T]$ and $\tl{t} \in[t,T]$, that
	\begin{align*}
		\E\|X^{s,x,\iota}_{\tl{t}}- X^{s,x,\iota}_t \|^p 
		\leq K\lf(|\tl{t}-t|+|\tl{t}-t|^{p}\rt)\lf(1+N_d^{\frac{p}{2}}+L_d^{\frac{p}{2}}+(1+L_d^{\frac{p}{2}})\sup_{r\in[t,\tl{t}]}\E\|X^{s,x,\iota}_{r}\|^{\frac{\chi+2}{2}p}\rt).
	\end{align*}
	Since $p \in[2,\frac{2p_0}{\chi+2}]$, by Lemma \ref{exactmomentbounds}, one obtains, for all $t\in[s,T]$ and $\tl{t} \in[t,T]$, that
	\begin{align*}
		\E\|X^{s,x,\iota}_{\tl{t}} - X^{s,x,\iota}_t\|^p \leq K\lf(|\tl{t}-t|+|\tl{t}-t|^{p}\rt)(1+L_d^{\frac{p}{2}})\big(1+\|x\|^{p_0}+N_d^{\frac{p_0}{2}}\big)e^{K(1+L_d^{\frac{p_0}{2}})}.
	\end{align*}
	This and symmetry complete the proof.
\end{proof}

\begin{proof}[\textbf{Proof of Lemma \ref{exactdifferentx}}]
	Fix $s \in[0,T]$ and $x\in\R^d$. By the definition of SDE \eqref{SDE}, the difference satisfies, for all $t\in[s,T]$, that
	\begin{align*}
		X_{t}^{s,x,\iota} - X_t^{s,\tl{x},\iota} &= x -\tl{x} + \int_{s}^t \mu(X_{r-}^{s,x,\iota})-\mu(X^{s,\tl{x},\iota}_{r-}) \rd r+\int_{s}^t \sg(X_{r-}^{s,x,\iota})-\sg(X^{s,\tl{x},\iota}_{r-}) \rd W_r \\
		&\quad+\int_{s}^t \int_{\R^d} \gm(X^{s,x,\iota}_{r-},z)-\gm(X^{s,\tl{x},\iota}_{r-},z) \tl{\pi}(\rd z ,\rd r).
	\end{align*}
	Define the following stopping time $\upsilon^{s,x,\tl{x},\iota}_R$ := $\tau^{s,x,\iota}_R \wedge \tau^{s,\tl{x},\iota}_R$ with $R \in \N$ where 
	\[\tau^{s,x,\iota}_R:= \inf\{t\geq s: \|X^{s,x,\iota}_t\|\geq R\}\wedge T, \qquad \tau^{s,\tl{x},\iota}_R:= \inf\{t\geq s: \|X^{s,\tl{x},\iota}_t\|\geq R\}\wedge T.
	\]
	Applying It\^o's formula, together with Cauchy-Schwarz inequality and Lemma \ref{Formula for the remainder}, one obtains, for all $t \in [s,T]$, that
	\begin{align*}
		&\E\|X_{t \wedge \upsilon^{s,x,\tl{x},\iota}_R}^{s,x,\iota} - X_{t \wedge \upsilon^{s,x,\tl{x},\iota}_R}^{s,\tl{x},\iota} \|^{p} \leq \|x -\tl{x}\|^p\\
		&\quad+
		p\E \int^{t \wedge \upsilon^{s,x,\tl{x},\iota}_R}_{s} \|X_{r}^{s,x,\iota}-X^{s,\tl{x},\iota}_{r}\|^{p-2}(X^{s,x,\iota}_r-X^{s,\tl{x},\iota}_{r})(\mu(X_{r}^{s,x,\iota})-\mu(X^{s,\tl{x},\iota}_{r}) )  \rd r\\
		&\quad+\frac{p(p-1)}{2}\E\int^{t \wedge \upsilon^{s,x,\tl{x},\iota}_R}_{s}\|X_{r}^{s,x,\iota}-X^{s,\tl{x},\iota}_{r}\|^{p-2}  \|\sg(X_{r}^{s,x,\iota} )-\sg(X^{s,\tl{x},\iota}_{r})\|^2 \rd r\\
		&\quad+K\E\int^{t \wedge\upsilon^{s,x,\tl{x},\iota}_R}_{s}\int_{\R^d} \big\{\|X_{r}^{s,x,\iota}-X^{s,\tl{x},\iota}_{r}\|^{p-2} \|\gm(X^{s,x,\iota}_{r},z)-\gm(X^{s,\tl{x},\iota}_{r},z)\|^2\\
		&\hspace{40mm}+\|\gm(X^{s,x,\iota}_{r},z)-\gm(X^{s,\tl{x},\iota}_{r},z)\|^{p}\big\}\nu(\rd z)\rd r.
	\end{align*}
	By Young's inequality, one gets, for all $t \in [s,T]$, that
	\begin{align*}
		&\E\|X_{t \wedge \upsilon^{s,x,\tl{x},\iota}_R}^{s,x,\iota} - X_{t \wedge \upsilon^{s,x,\tl{x},\iota}_R}^{s,\tl{x},\iota} \|^{p} \leq \|x -\tl{x}\|^p\\
		&\quad+
		\frac{p}{2}\E \int^{t \wedge \upsilon^{s,x,\tl{x},\iota}_R}_{s} \|X_{r}^{s,x,\iota}-X^{s,\tl{x},\iota}_{r}\|^{p-2}\big\{2(X^{s,x,\iota}_r-X^{s,\tl{x},\iota}_{r})\\
		&\hspace{30mm}\times(\mu(X_{r}^{s,x,\iota})-\mu(X^{s,\tl{x},\iota}_{r}) ) +(p_0-1)  \|\sg(X_{r}^{s,x,\iota} )-\sg(X^{s,\tl{x},\iota}_{r})\|^2 \big\}\rd r\\
		&\quad+K\E\int^{t \wedge\upsilon^{s,x,\tl{x},\iota}_R}_{s} \|X_{r}^{s,x,\iota}-X^{s,\tl{x},\iota}_{r}\|^{p} \rd r\\
		&\quad+K\E\int^{t \wedge\upsilon^{s,x,\tl{x},\iota}_R}_{s} \bigg(\int_{\R^d}  \|\gm(X^{s,x,\iota}_{r},z)-\gm(X^{s,\tl{x},\iota}_{r},z)\|^2\nu(\rd z)  \bigg)^{\frac{p}{2}}\rd r\\
		&\quad+K\E\int^{t \wedge\upsilon^{s,x,\tl{x},\iota}_R}_{s} \int_{\R^d} \|\gm(X^{s,x,\iota}_{r},z)-\gm(X^{s,\tl{x},\iota}_{r},z)\|^{p}\nu(\rd z)\rd r,
	\end{align*}
	which on the application of Assumptions \ref{A1} and \ref{A4} yields, for all $t \in [s,T]$, that
	\begin{align*}
		\E\|X_{t \wedge \upsilon^{s,x,\tl{x},\iota}_R}^{s,x,\iota} - X_{t \wedge \upsilon^{s,x,\tl{x},\iota}_R}^{s,\tl{x},\iota} \|^{p} \leq \|x-\tl{x}\|^p+K(1+L_d^{\frac{p_0}{2}})\E\int^{t \wedge\upsilon^{s,x,\tl{x},\iota}_R}_{s} \|X_{r}^{s,x,\iota}-X^{s,\tl{x},\iota}_{r}\|^{p} \rd r.
	\end{align*}
	Finally, the application of Gr\"onwall's lemma and Fatou's lemma completes the proof.
\end{proof}

\begin{proof}[\textbf{Proof of Lemma \ref{exactdifferents}}]
	Fix $s \in[0,T]$, $\tl{s} \in [s,T]$, and $x\in\R^d$. By the definition of SDE \eqref{SDE}, we have, for all $t\in[\tl{s},T]$, that
	\begin{align*}
		X_{t}^{s,x,\iota} - X_t^{\tl{s},x,\iota} &= X_{\tl{s}}^{s,x,\iota} -x + \int_{\tl{s}}^t \mu(X_{r-}^{s,x,\iota})-\mu(X^{\tl{s},x,\iota}_{r-}) \rd r+\int_{\tl{s}}^t \sg(X_{r-}^{s,x,\iota})-\sg(X^{\tl{s},x,\iota}_{r-}) \rd W_r \\
		&\quad+\int_{\tl{s}}^t \int_{\R^d} \gm(X^{s,x,\iota}_{r-},z)-\gm(X^{\tl{s},x,\iota}_{r-},z) \tl{\pi}(\rd z ,\rd r).
	\end{align*}
	Define a stopping time $\upsilon^{s,\tl{s},x}_R$ := $\tau^{s,x}_R \wedge \tau^{\tl{s},x}_R$ with $R \in \N$
	where \[\tau^{s,x}_R:= \inf\{t\geq s: \|X^{s,x,\iota}_t\|\geq R\}\wedge T,\qquad \tau^{\tl{s},x}_R:= \inf\{t\geq \tl{s}: \|X^{\tl{s},x,\iota}_t\|\geq R\}\wedge T.\]
	Moreover, we use It\^o's formula, Cauchy-Schwarz inequality, and Lemma \ref{Formula for the remainder} to obtain, for all $t \in [\tl{s},T]$, that
	\begin{align*}
		&\E\|X_{t \wedge \upsilon^{s,\tl{s},x}_R}^{s,x,\iota} - X_{t \wedge \upsilon^{s,\tl{s},x}_R}^{\tl{s},x,\iota} \|^{p} \leq \E\|X_{\tl{s}}^{s,x,\iota} -x\|^p\\
		&\quad+
		p\E \int^{t \wedge \upsilon^{s,\tl{s},x}_R}_{\tl{s}} \|X_{r}^{s,x,\iota}-X^{\tl{s},x,\iota}_{r}\|^{p-2}(X^{s,x,\iota}_r-X^{\tl{s},x,\iota}_{r})(\mu(X_{r}^{s,x,\iota})-\mu(X^{\tl{s},x,\iota}_{r}) )  \rd r\\
		&\quad+\frac{p(p-1)}{2}\E\int^{t \wedge \upsilon^{s,\tl{s},x}_R}_{\tl{s}}\|X_{r}^{s,x,\iota}-X^{\tl{s},x,\iota}_{r}\|^{p-2}  \|\sg(X_{r}^{s,x,\iota} )-\sg(X^{\tl{s},x,\iota}_{r})\|^2 \rd r\\
		&\quad+K\E\int^{t \wedge\upsilon^{s,\tl{s},x}_R}_{\tl{s}}\int_{\R^d} \big\{\|X_{r}^{s,x,\iota}-X^{\tl{s},x,\iota}_{r}\|^{p-2} \|\gm(X^{s,x,\iota}_{r},z)-\gm(X^{\tl{s},x,\iota}_{r},z)\|^2\\
		&\hspace{35mm}+\|\gm(X^{s,x,\iota}_{r},z)-\gm(X^{\tl{s},x,\iota}_{r},z)\|^{p}\big\}\nu(\rd z)\rd r,
	\end{align*}
	which on using Young's inequality, Assumptions \ref{A1} and \ref{A4} gives, for all $t \in [\tl{s},T]$, that
	\begin{align*}
		\E\|X_{t \wedge \upsilon^{s,\tl{s},x}_R}^{s,x,\iota} - X_{t \wedge \upsilon^{s,x,\tl{s},x,\iota}_R}^{\tl{s},x,\iota} \|^{p} \leq \E\|X_{\tl{s}}^{s,x,\iota} -x\|^p+K(1+L_d^{\frac{p}{2}})\E\int^{t \wedge\upsilon^{s,\tl{s},x}_R}_{\tl{s}} \|X_{r}^{s,x,\iota}-X^{\tl{s},x,\iota}_{r}\|^{p} \rd r.
	\end{align*}
	Gr\"onwall's lemma implies that
	\begin{align*}
		\sup_{t\in[\tl{s},T]} \E\|X_{t \wedge \upsilon^{s,\tl{s},x}_R}^{s,x,\iota} - X_{t \wedge \upsilon^{s,\tl{s},x}_R}^{\tl{s},x,\iota} \|^{p} &\leq \E\|X_{\tl{s}}^{s,x,\iota} -x\|^p+K(1+L_d^{\frac{p}{2}})\int^{T}_{\tl{s}} \sup_{u\in[{\tl{s}},r]} \E \|X_{u \wedge \upsilon^{s,\tl{s},x}_R}^{s,x,\iota}-X^{\tl{s},x,\iota}_{u \wedge \upsilon^{s,\tl{s},x}_R}\|^{p} \rd r\\
		&\leq \E\|X_{\tl{s}}^{s,x,\iota} -x\|^p e^{K(1+L_d^{\frac{p}{2}})}.
	\end{align*}
	Then, Lemma \ref{exactdifferenttimet} applied with $(t,\tl{t})  \curvearrowleft (s,\tl{s})$ yields that
	\begin{align*}
		\sup_{t\in[\tl{s},T]} \E\|X_{t \wedge \upsilon^{s,\tl{s},x}_R}^{s,x,\iota} - X_{t \wedge \upsilon^{s,\tl{s},x}_R}^{\tl{s},x,\iota} \|^{p} \leq K\big(|\tl{s}-s|+|\tl{s}-s|^p\big)(1+L_d^{\frac{p}{2}}) \big(1+\|x\|^{p_0}+N_d^{\frac{p_0}{2}}\big)e^{K(1+L_d^{\frac{p_0}{2}})}.
	\end{align*}
	Finally, the application of Fatou's lemma and symmetry finish the proof.
\end{proof}

\begin{proof}[\textbf{Proof of Lemma \ref{differenttimet}}]
	Fix $s\in [0,T]$, $x\in \R^d$, $\dl \in \Th$, $\ep \in (0,1)$, and $\mathcal{M} \in \N$ with $\mathcal{M}\geq (1 \vee \ep^{-2}\mathcal{Z}_d)^{2}$. By the definition of scheme \eqref{Tamedscheme2}, we obtain, for all $t\in[s,T]$ and $\tl{t} \in[t,T]$, that
	\begin{align*}
		\E\|X^{s,x,\dl,\ep,\mathcal{M}}_{\tl{t}} - X^{s,x,\dl,\ep,\mathcal{M}}_t\|^p &\leq K\E \lf\| \int^{\tl{t}}_t \mu^{\dl}(X^{s,x,\dl,\ep,\mathcal{M}}_{\max\{s,\dl(r)\}})\rd r\rt\|^p + K\E\lf\|\int^{\tl{t}}_t \sg^{\dl}(X^{s,x,\dl,\ep,\mathcal{M}}_{\max\{s,\dl(r)\}})\rd W_r \rt\|^p\\
		&\quad+K\E\lf\|\int^{\tl{t}}_t \int_{A_\ep} \gm(X^{s,x,\dl,\ep,\mathcal{M}}_{\max\{s,\dl(r)\}},z) \tl\pi(\rd z,\rd r)\rt\|^p \\
		&\quad+K\E\bigg\|\int^{\tl{t}}_t\int_{A_\ep} \gm(X^{s,x,\dl,\ep,\mathcal{M}}_{\max\{s,\dl(r)\}},z)\nu(\rd z) \\
		&\quad\quad\quad - \lf\{\frac{\nu(A_\ep)}{\mathcal{M}}\sum_{i=1}^{\mathcal{M}}\gm\lf(X_{\max\{s,\dl(r-)\}}^{s,x,\dl,\ep,\mathcal{M}},V^{\dl,\ep,\mathcal{M}}_{i,\max\{s,\dl(r)\}}\rt)\rt\}\rd r \bigg\|^p.
	\end{align*}
	By using H\"older's inequality, \cite[Theorem 7.3]{xuerongmao}, and \cite[Lemma 1]{Mikulevicius}, we obtain, for all $t\in[s,T]$ and $\tl{t} \in[t,T]$, that
	\begin{align*}
		&\E\|X^{s,x,\dl,\ep,\mathcal{M}}_{\tl{t}} - X^{s,x,\dl,\ep,\mathcal{M}}_t\|^p\\
		& \leq K|\tl{t}-t|^{p-1}\E \int^{\tl{t}}_t \|\mu^{\dl}(X^{s,x,\dl,\ep,\mathcal{M}}_{\max\{s,\dl(r)\}})\|^{p}\rd r+ K|\tl{t}-t|^{\frac{p}{2}-1}\E\int^{\tl{t}}_t \|\sg^{\dl}(X^{s,x,\dl,\ep,\mathcal{M}}_{\max\{s,\dl(r)\}})\|^{p}\rd r \\
		&\quad+K|\tl{t}-t|^{\frac{p}{2}-1}\E\int^{\tl{t}}_t \big(\int_{A_\ep} \|\gm(X^{s,x,\dl,\ep,\mathcal{M}}_{\max\{s,\dl(r)\}},z)-\gm(0,z)\|^2 \nu(\rd z)\big)^{\frac{p}{2}} \rd r\\
		&\quad+K|\tl{t}-t|^{\frac{p}{2}-1}\E\int^{\tl{t}}_t \lf(\int_{A_\ep} \|\gm(0,z)\|^2 \nu(\rd z)\rt)^{\frac{p}{2}} \rd r\\
		&\quad+K\E\int^{\tl{t}}_t\int_{A_\ep} \|\gm(X^{s,x,\dl,\ep,\mathcal{M}}_{\max\{s,\dl(r)\}},z)-\gm(0,z)\|^p \nu(\rd z) \rd r
+K\E\int^{\tl{t}}_t\int_{A_\ep} \|\gm(0,z)\|^p \nu(\rd z) \rd r\\
		&\quad+K|\tl{t}-t|^{p-1}\E\int^{\tl{t}}_t \bigg\|\int_{A_\ep} \gm(X^{s,x,\dl,\ep,\mathcal{M}}_{\max\{s,\dl(r)\}},z)\nu(\rd z) - \lf\{\frac{\nu(A_\ep)}{\mathcal{M}}\sum_{i=1}^{\mathcal{M}}\gm\lf(X_{\max\{s,\dl(r-)\}}^{s,x,\dl,\ep,\mathcal{M}},V^{\dl,\ep,\mathcal{M}}_{i,\max\{s,\dl(r)\}}\rt)\rt\}\bigg\|^p\rd r .
	\end{align*}
	Then, applying Remarks \ref{remark1}, \ref{remark2}, Assumption \ref{A2}, \ref{A3}, \ref{A4}, and \eqref{L3A_4(t)} (together with $\mathcal{M}\ge (1 \vee \ep^{-2}\mathcal{Z}_d)^{2}$) yields, for all $t\in[s,T]$ and $\tl{t} \in[t,T]$, that
	\begin{align*}
		\E\|X^{s,x,\dl,\ep,\mathcal{M}}_{\tl{t}}- X^{s,x,\dl,\ep,\mathcal{M}}_t \|^p 
		&\leq K\lf(|\tl{t}-t|+|\tl{t}-t|^{p}(1+(\ep^{-2}\mathcal{Z}_d)^{p-1}\mathcal{M}^{-\frac{p}{2}})\rt)\\
		&\quad\times\lf(1+N_d^{\frac{p}{2}}+L_d^{\frac{p}{2}}+(1+L_d^{\frac{p}{2}})\sup_{r\in[t,\tl{t}]}\E\|X^{s,x,\dl,\ep,\mathcal{M}}_{\max\{s,\dl(r)\}}\|^{\frac{\chi+2}{2}p}\rt).
	\end{align*}
	Since $p \in[2,\frac{2p_0}{\chi+2}]$, one applies Lemma \ref{scheme3bounds} to obtain, for all $t\in[s,T]$ and $\tl{t} \in[t,T]$, that
	\begin{align*}
		&\E\|X^{s,x,\dl,\ep,\mathcal{M}}_{\tl{t}} - X^{s,x,\dl,\ep,\mathcal{M}}_t\|^p \\ &\leq K\lf(|\tl{t}-t|+|\tl{t}-t|^{p}(1+(\ep^{-2}\mathcal{Z}_d)^{p-1}\mathcal{M}^{-\frac{p}{2}})\rt) (1+L_d^{\frac{p}{2}})\big(1+\|x\|^{p_0}+N_d^{p_0}\big)e^{K(1+L_d^{p_0})}.
	\end{align*}
	This and symmetry complete the proof.
\end{proof}

\begin{proof}[\textbf{Proof of Lemma \ref{differentx'}}]
	Fix $s\in[0,T]$, $x,\tl{x}\in\R^d$, $\dl \in \Th$, $\ep \in (0,1)$, and $\mathcal{M} \in \N$ with $\mathcal{M}\geq (1 \vee \ep^{-2}\mathcal{Z}_d)^{2}$ throughout this proof. By the definition of scheme \eqref{Tamedscheme2}, we obtain, for all $t \in [s,T]$, that
	\begin{align*}
		&\E\|X_{t}^{s,x,\dl,\ep,\mathcal{M}} - X_t^{s,\tl{x},\dl,\ep,\mathcal{M}} \|^p \leq K\|x -\tl{x}\|^p \\
		&\quad+K\E \lf\| \int_{s}^t \mu^\dl(X_{\max\{s,\dl(r)\}}^{s,x,\dl,\ep,\mathcal{M}})-\mu^\dl(X^{s,\tl{x},\dl,\ep,\mathcal{M}}_{\max\{s,\dl(r)\}}) \rd r\rt\|^p\\
		&\quad+K\E \lf\|\int_{s}^t \sg^\dl(X_{\max\{s,\dl(r)\}}^{s,x,\dl,\ep,\mathcal{M}})-\sg^\dl(X^{s,\tl{x},\dl,\ep,\mathcal{M}}_{\max\{s,\dl(r)\}}) \rd W_r \rt\|^p\\
		&\quad+K\E \lf\|\int_{s}^t \int_{A_\ep} \gm(X^{s,x,\dl,\ep,\mathcal{M}}_{\max\{s,\dl(r)\}},z)-\gm(X^{s,\tl{x},\dl,\ep,\mathcal{M}}_{\max\{s,\dl(r)\}},z) \tl{\pi}(\rd z ,\rd r)\rt\|^p\\
		&\quad+K\E \bigg\|\int^{t}_{s} \int_{A_\ep} \gm(X^{s,x,\dl,\ep,\mathcal{M}}_{\max\{s,\dl(r)\}},z)-\gm(X^{s,\tl{x},\dl,\ep,\mathcal{M}}_{\max\{s,\dl(r)\}},z)\nu(\rd z)\\
		&\quad\quad- \lf\{\frac{\nu(A_\ep)}{\mathcal{M}}\sum_{i=1}^{\mathcal{M}}\lf(\gm\lf(X_{\max\{s,\dl(r)\}}^{s,x,\dl,\ep,\mathcal{M}},V^{\dl,\ep,\mathcal{M}}_{i,\max\{s,\dl(r)\}}\rt)-\gm\lf(X_{\max\{s,\dl(r)\}}^{s,\tl{x},\dl,\ep,\mathcal{M}},V^{\dl,\ep,\mathcal{M}}_{i,\max\{s,\dl(r)\}}\rt)\rt)\rt\} \rd r \bigg\|^p.
	\end{align*}
		By using H\"older's inequality, \cite[Theorem 7.3]{xuerongmao}, and \cite[Lemma 1]{Mikulevicius}, we obtain, for all $t\in[s,T]$, that
	\begin{align*}
			&\E\|X_{t}^{s,x,\dl,\ep,\mathcal{M}} - X_t^{s,\tl{x},\dl,\ep,\mathcal{M}} \|^p \leq K\|x -\tl{x}\|^p\\
			&\quad+ K\E \int_{s}^t\| \mu^\dl(X_{\max\{s,\dl(r)\}}^{s,x,\dl,\ep,\mathcal{M}})-\mu^\dl(X^{s,\tl{x},\dl,\ep,\mathcal{M}}_{\max\{s,\dl(r)\}}) \|^p\rd r\\
			&\quad+K\E \int_{s}^t \|\sg^\dl(X_{\max\{s,\dl(r)\}}^{s,x,\dl,\ep,\mathcal{M}})-\sg^\dl(X^{s,\tl{x},\dl,\ep,\mathcal{M}}_{\max\{s,\dl(r)\}})\|^p \rd W_r \\
			&\quad+K\E \int_{s}^t \lf( \int_{A_\ep} \|\gm(X^{s,x,\dl,\ep,\mathcal{M}}_{\max\{s,\dl(r)\}},z)-\gm(X^{s,\tl{x},\dl,\ep,\mathcal{M}}_{\max\{s,\dl(r)\}},z)\|^2 \nu(\rd z) \rt)^{\frac{p}{2}} \rd r\\
			&\quad+K\E \int_{s}^t\int_{A_\ep} \|\gm(X^{s,x,\dl,\ep,\mathcal{M}}_{\max\{s,\dl(r)\}},z)-\gm(X^{s,\tl{x},\dl,\ep,\mathcal{M}}_{\max\{s,\dl(r)\}},z)\|^p \nu(\rd z)  \rd r\\
			&\quad+K \E \int^{t}_{s} \bigg\| \int_{A_\ep} \gm(X^{s,x,\dl,\ep,\mathcal{M}}_{\max\{s,\dl(r)\}},z)-\gm(X^{s,\tl{x},\dl,\ep,\mathcal{M}}_{\max\{s,\dl(r)\}},z)\nu(\rd z)\\
			&\quad- \lf\{\frac{\nu(A_\ep)}{\mathcal{M}}\sum_{i=1}^{\mathcal{M}}\lf(\gm\lf(X_{\max\{s,\dl(r)\}}^{s,x,\dl,\ep,\mathcal{M}},V^{\dl,\ep,\mathcal{M}}_{i,\max\{s,\dl(r)\}}\rt)-\gm\lf(X_{\max\{s,\dl(r)\}}^{s,\tl{x},\dl,\ep,\mathcal{M}},V^{\dl,\ep,\mathcal{M}}_{i,\max\{s,\dl(r)\}}\rt)\rt)\rt\} \bigg\|^{p} \rd r.
		\end{align*}
	On using Assumptions \ref{A4}, Remark \ref{remark4}, and the analogous argument to get \eqref{L3A_4(t)} (together with $\mathcal{M}\ge (1 \vee \ep^{-2}\mathcal{Z}_d)^{2}$), we obtain, for all $t\in[s,T]$, that
	\begin{align*}
		&\E\|X_{t}^{s,x,\dl,\ep,\mathcal{M}} - X_{t}^{s,\tl{x},\dl,\ep,\mathcal{M}} \|^{p} \leq K\|x -\tl{x}\|^p\\
		&\quad+
		K|\dl|^{-\frac{p}{4}}\E \int ^{t}_{s} \|X_{r}^{s,x,\dl,\ep,\mathcal{M}}-X^{s,\tl{x},\dl,\ep,\mathcal{M}}_{r}\|^{p} \rd r \\
		&\quad+K(L_d+L_d^{\frac{p}{2}}) \E \int ^{t}_{s} \|X_{\max\{s,\dl(r)\}}^{s,x,\dl,\ep,\mathcal{M}}-X^{s,\tl{x},\dl,\ep,\mathcal{M}}_{\max\{s,\dl(r)\}}\|^{p} \rd r\\
		&\quad+ K(\ep^{-2}\mathcal{Z}_d)^{p-1}\mathcal{M}^{-\frac{p}{2}}L_d \E \int^{t }_{s}  \| X^{s,x,\dl,\ep,\mathcal{M}}_{\max\{s,\dl(r)\}}-X_{\max\{s,\dl(r)\}}^{s,\tl{x},\dl,\ep,\mathcal{M}}\|^{p}\rd r.
	\end{align*}
It is followed by
	\begin{align*}
		&\sup_{t\in[s,T]}	\E\|X_{t}^{s,x,\dl,\ep,\mathcal{M}} - X_{t }^{s,\tl{x},\dl,\ep,\mathcal{M}} \|^{p} \leq K\|x -\tl{x}\|^p\\
		&\quad+K\lf(|\dl|^{-\frac{p}{4}}+L_d^{\frac{p}{2}} +(\ep^{-2}\mathcal{Z}_d)^{p-1}\mathcal{M}^{-\frac{p}{2}}L_d \rt)\E \int^{T}_{s} \sup_{u\in[s,r]} \| X^{s,x,\dl,\ep,\mathcal{M}}_{u}-X_{u}^{s,\tl{x},\dl,\ep,\mathcal{M}}\|^{p}\rd r.
	\end{align*}
	Finally, Gr\"onwall's lemma completes the proof.
\end{proof}

\begin{proof}[\textbf{Proof of Lemma \ref{differents'}}]
	\begin{figure}[t]
		\centering
		
		% A small helper style for consistent look
		\tikzset{
			timeline/.style={>=Latex,thick},
			gridpt/.style={circle,fill=black,inner sep=1.2pt},
			spt/.style={circle,fill=blue,inner sep=1.6pt},
			tpt/.style={circle,fill=red,inner sep=1.6pt},
			bpt/.style={circle,fill=black,inner sep=1.6pt},
			brace/.style={decorate,decoration={brace,amplitude=4pt}}
		}
		
		% -------------------- Case 1 --------------------
		\begin{subfigure}[t]{0.48\textwidth}
			\centering
			\resizebox{\linewidth}{!}{%
				\begin{tikzpicture}[timeline]
					\draw[-] (0,0) -- (8,0) node[right] {};
					
					\foreach \x/\lab in  {1/{t_i},4/{t_{i+1}},6/{t_{i+2}}} {
						\node[crosspt] at (\x,0) {};
						\node[below] at (\x,0) {\small $\lab$};
					}
					
					\node[spt] at (1.5,0) {};
					\node[above,blue] at (1.5,0) {\small $s$};
					\node[tpt] at (2.8,0) {};
					\node[above,red] at (2.8,0) {\small $\tilde s$};
					\node[tpt,fill=teal] at (3.5,0) {};
					\node[above,teal] at (3.5,0) {\small $\bar s$};
					\node at (6.9,-0.25) { $\cdots$};
					\node[below,black] at (7.8,0) {\small $t$};
					\node[tpt,fill=black] at (7.8,0) {};
				\end{tikzpicture}%
			}
			\caption{(i) No grid point in \((s,\bar s)\).}
		\end{subfigure}
		\hfill
		% -------------------- Case 2 --------------------
		\begin{subfigure}[t]{0.48\textwidth}
			\centering
			\resizebox{\linewidth}{!}{%
				\begin{tikzpicture}[timeline]
					\draw[-] (0,0) -- (8,0) node[right] {};
					
					\foreach \x/\lab in  {1/{t_i},4/{t_{i+1}},6/{t_{i+2}}} {
						\node[crosspt] at (\x,0) {};
						\node[below] at (\x,0) {\small $\lab$};
					}
					
					\node[spt] at (1.5,0) {};
					\node[above,blue] at (1.5,0) {\small $s$};
					\node[tpt] at (2.8,0) {};
					\node[above,red] at (2.8,0) {\small $\tilde s$};
					\node[tpt,fill=teal] at (4.0,0) {};
					\node[above,teal] at (4.0,0) {\small $\bar s$};
					\node at (6.9,-0.25) { $\cdots$};
					\node[below,black] at (7.8,0) {\small $t$};
					\node[tpt,fill=black] at (7.8,0) {};
				\end{tikzpicture}%
			}
			\caption{(ii) $\bar s$ is the smallest grid point on $(\tilde{s},t]$.}
		\end{subfigure}
		
		\vspace{0.8em}
		
		% -------------------- Case 3 --------------------
		\begin{subfigure}[t]{0.48\textwidth}
			\centering
			\resizebox{\linewidth}{!}{%
				\begin{tikzpicture}[timeline]
					\draw[-] (0,0) -- (8,0) node[right] {};
					
					\foreach \x/\lab in  {1/{t_i},4/{t_{i+1}},6/{t_{i+2}}} {
						\node[crosspt] at (\x,0) {};
						\node[below] at (\x,0) {\small $\lab$};
					}
					
					\node[spt] at (3,0) {};
					\node[above,blue,align=center] at (3,0) {\small $s$};
					
					\node[tpt] at (6.0,0) {};
					\node[above,red] at (6.0,0) {\small $\tilde s$};
					\node at (6.9,-0.25) { $\cdots$};
					\node[below,black] at (7.8,0) {\small $t$};
					\node[tpt,fill=black] at (7.8,0) {};
				\end{tikzpicture}%
			}
			\caption{(iii) $\tl{s}$ is a grid point.}
		\end{subfigure}
		\hfill
		% -------------------- Case 4 --------------------
		\begin{subfigure}[t]{0.48\textwidth}
			\centering
			\resizebox{\linewidth}{!}{%
				\begin{tikzpicture}[timeline]
					\draw[-] (0,0) -- (8,0) node[right] {};
					
					\foreach \x/\lab in  {1/{t_i},4/{t_{i+1}},6/{t_{i+2}}} {
						\node[crosspt] at (\x,0) {};
						\node[below] at (\x,0) {\small $\lab$};
					}
					
					\node[spt] at (0.8,0) {};
					\node[above,blue] at (0.8,0) {\small $s$};
					
					\node[tpt,fill=teal] at (4,0) {};
					\node[above,teal] at (4,0) {\small $\bar s$};
					
					\node[tpt] at (5.0,0) {};
					\node[above,red] at (5.0,0) {\small $\tilde s$};
					\node at (6.9,-0.25) { $\cdots$};
					\node[below,black] at (7.8,0) {\small $t$};
					\node[tpt,fill=black] at (7.8,0) {};
				\end{tikzpicture}%
			}
			\caption{(iv) $\bar{s}$ is the largest grid point on $[s,\tl{s})$.}
		\end{subfigure}
		
		\caption{A illustration for the case distinction. The points $\{t_i,t_{i+1},t_{i+2}\}$ indicate grid points (drawn by $\times$).}
		\label{fig:four-cases-initial-time}
	\end{figure}
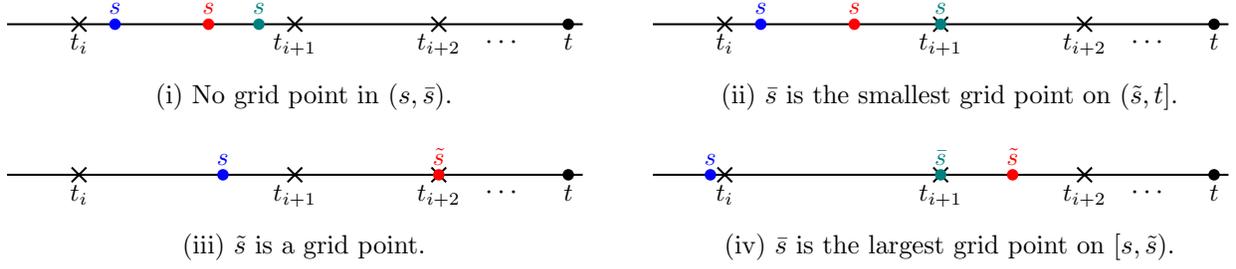
	
	By symmetry, it suffices to consider $s\le \tl{s}\le T$ and $t\in[\tl{s},T]$. Fix  $x\in\R^d$, $\dl \in \Th$, $\ep \in (0,1)$, and $\mathcal{M} \in \N$ with $\mathcal{M}\geq (1 \vee \ep^{-2}\mathcal{Z}_d)^{2}$ throughout this proof. Then, we follow the case distinction given in Figure \ref{fig:four-cases-initial-time}. Note that the case distinction is not independent; it is a strategy that upgrades a local estimate to the general situation.
	In Case~(i), we prove the one-step bound \eqref{enumerate1'} on an interval without grid points, i.e., for $(s,\bar s)\cap\dl([0,T])=\emptyset$, by estimating directly the drift, Brownian, jump and Monte-Carlo compensation terms on $[s,\tilde s]$.
	This local estimate is then the key input for Case~(ii): when $(s,\tilde s)\cap\dl([0,T])=\emptyset$ but we want a bound at time $t\ge \tilde s$, we choose the first grid time $\bar s=\min([\tilde s,t]\cap(\dl([0,T])\cup\{t\}))$ and apply the disintegration theorem (see, e.g., \cite[Lemma 2.2]{disintegration}) and the Markov property on grid points, so that \eqref{enumerate1'} (together with Lemma~\ref{differentx'}) yields \eqref{enumerate2'}.
	In Case~(iii) we treat the complementary situation where $\tilde s$ is a grid time, i.e., $\tilde s\in\dl([0,T])$, and again apply the disintegration theorem and the Markov property on grid points. Then, Lemmas~\ref{differenttimet} and \ref{differentx'} give \eqref{enumerate3'}.
	Finally, Case~(iv) reduces the general situation of $s\le \tilde s$ to the previously established cases: we select
	\(\bar s\in [s,\tilde s]\cap(\dl([0,T])\cup\{s\})\) such that \((\bar s,\tilde s)\cap\dl([0,T])=\emptyset\), split by the triangle inequality as in \eqref{enumerate4'}, and bound the term from $s$ to $\bar s$ by Case~(iii) (since $\bar s$ is a grid point) and the term from $\bar s$ to $\tilde s$ by Case~(ii) (since $(\bar s,\tilde s)$ contains no grid point).
	Combining these steps yields the desired estimate for all $s\le \tilde s$ and all $t\in[\tilde s,T]$. We now proceed with the detailed arguments for the four cases.
	
	\begin{enumerate}
		\item By definition of scheme \eqref{Tamedscheme2}, H\"older's inequality, \cite[Theorem 7.3]{xuerongmao}, \cite[Lemma 1]{Mikulevicius}, Remarks \ref{remark1},  \ref{remark2}, Assumptions \ref{A3}, \ref{A4}, and \eqref{L3A_4(t)} (together with $\mathcal{M}\ge (1 \vee \ep^{-2}\mathcal{Z}_d)^{2}$), one obtains, for all $s\in[0,T]$, $\tl{s}\in[s,T]$, and $\bar{s} \in[\tl{s},T]$ with $(s,\bar{s})\cap \dl([0,T])=\emptyset$, that
			\begin{align}\label{enumerate1'}
				&\E\|X_{\bar{s}}^{s,x,\dl,\ep,\mathcal{M}} - X_{\bar{s}}^{\tl{s},x,\dl,\ep,\mathcal{M}} \|^p\leq K\E\bigg[\|\mu^{\dl}(x)|\tl{s}-s|\|^p +\|\sg^{\dl}(x)(W_{\tl{s}}-W_s)\|^p\nonumber\\
				&\quad+\big\|\int^{\tl{s}}_s \int_{A_\ep}  \gm(x,z) -\gm(0,z) \tl{\pi}(\rd z,\rd r)\big\|^p +\big\|\int^{\tl{s}}_s \int_{A_\ep}  \gm(0,z) \tl{\pi}(\rd z,\rd r) \big\|^p\nonumber\\
				&\quad+ \big\|\int^{\tl{s}}_s \int_{A_\ep}  \gm(x,z) \nu(\rd z) - \frac{\nu(A_\ep)}{\mathcal{M}} \sum_{i=1}^{\mathcal{M}} \gm(x, V^{\dl,\ep,\mathcal{M}}_{i,s}) \rd r\big\|^p\bigg]\nonumber\\
				&\leq K\big\{  |\tl{s}-s|^{p}(1+\|x\|^{(\frac{\chi}{2}+1)p}) +  |\tl{s}-s|^{\frac{p}{2}}(1+\|x\|^{(\frac{\chi}{2}+1)p})\nonumber\\
				&\quad+  |\tl{s}-s|^{\frac{p}{2}}(N_d^{\frac{p}{2}}+L_d^{\frac{p}{2}}\|x\|^{p})+  |\tl{s}-s|(N_d+L_d\|x\|^{p}) \nonumber\\
				&\quad+  |\tl{s}-s|^p(\ep^{-2}\mathcal{Z}_d)^{p-1}\mathcal{M}^{-\frac{p}{2}} (N_d+L_d\|x\|^p)\big\}\nonumber\\
				&\leq K\big(|\tl{s}-s|+|\tl{s}-s|^p(1+(\ep^{-2}\mathcal{Z}_d)^{p-1}\mathcal{M}^{-\frac{p}{2}})\big) (1+N_d^{\frac{p}{2}}+L_d^{\frac{p}{2}}+(1+L_d^{\frac{p}{2}})\|x\|^{p_0}).
			\end{align}
		
		\item Note that for all $s\in[0,T]$, $\tl{s}\in[s,T]$, and $t\in[\tl{s},T]$ with $(s,\tl{s})\cap \dl([0,T])=\emptyset$, there exists $\bar{s}\in[\tl{s},t]\cap (\dl([0,T]) \cup \{t\})$ with $(s,\bar{s})\cap \dl([0,T]) = \emptyset$. Moreover, by the fact that Euler approximations \eqref{Tamedscheme2} restricted to their grid points satisfy the Markov property, the disintegration theorem (see, e.g., \cite[Lemma 2.2]{disintegration}), Lemma \ref{differentx'}, and \eqref{enumerate1'}, show that, for all $s\in[0,T]$, $\tl{s}\in[s,T]$, $t\in[\tl{s},T]$, and $\bar{s} = \min([\tl{s},t] \cap (\dl([0,T]) \cup \{t\}))$, it holds that
			\begin{align}\label{enumerate2'}
				&\E\|X_{t}^{s,x,\dl,\ep,\mathcal{M}} - X_{t}^{\tl{s},x,\dl,\ep,\mathcal{M}} \|^p=\E\lf[  \E\|   X_{t}^{\bar{s},\eta,\dl,\ep,\mathcal{M}} - X_{t}^{\bar{s},\tl{\eta},\dl,\ep,\mathcal{M}} \|^p \bigg\vert_{\eta = X_{\bar{s}}^{s,x,\dl,\ep,\mathcal{M}},\tl{\eta}=X_{\bar{s}}^{\tl{s},x,\dl,\ep,\mathcal{M}}   }  \rt] \nonumber\\
				&\leq K\E\lf[ \big(\|\eta -\tl{\eta}\|^pe^{K(|\dl|^{-\frac{p}{4}}+L_d^{\frac{p}{2}})} \bigg\vert_{\eta = X_{\bar{s}}^{s,x,\dl,\ep,\mathcal{M}},\tl{\eta}=X_{\bar{s}}^{\tl{s},x,\dl,\ep,\mathcal{M}}   } \rt]\nonumber\\  
				&\leq K\big(|\tl{s}-s|+|\tl{s}-s|^p(1+(\ep^{-2}\mathcal{Z}_d)^{p-1}\mathcal{M}^{-\frac{p}{2}})\big) (1+N_d^{\frac{p}{2}}+L_d^{\frac{p}{2}}+(1+L_d^{\frac{p}{2}})\|x\|^{p_0}) e^{K(|\dl|^{-\frac{p}{4}}+L_d^{\frac{p}{2}})}.
			\end{align}
		
		\item Next, we apply the disintegration theorem (see, e.g., \cite[Lemma 2.2]{disintegration}), Lemmas \ref{differenttimet} and \ref{differentx'} to show for all $s\in[0,T]$, $\tl{s} \in \dl([0,T])\cap[s,T]$, and $t\in[\tl{s},T]$, that
			\begin{align}\label{enumerate3'}
				&\E\|X_{t}^{s,x,\dl,\ep,\mathcal{M}} - X_{t}^{\tl{s},x,\dl,\ep,\mathcal{M}} \|^p=\E\lf[  \E\|   X_{t}^{\tl{s},\eta,\dl,\ep,\mathcal{M}} - X_{t}^{\tl{s},x,\dl,\ep,\mathcal{M}} \|^p \bigg\vert_{\eta = X_{\tl{s}}^{s,x,\dl,\ep,\mathcal{M}}  }  \rt] \nonumber\\
				&d \leq K\E\lf[ \big(\|\eta -x\|^p e^{K(|\dl|^{-\frac{p}{4}}+L_d^{\frac{p}{2}})}  \bigg\vert_{\eta = X_{\tl{s}}^{s,x,\dl,\ep,\mathcal{M}}  } \rt]\nonumber\\
				&\leq K \big[\big(|\tl{s}-s|+|\tl{s}-s|^p(1+(\ep^{-2}\mathcal{Z}_d)^{p-1}\mathcal{M}^{-\frac{p}{2}})\big) (1+L_d^{\frac{p}{2}})(1+\|x\|^{p_0}+N_d^{p_0})e^{K(|\dl|^{-\frac{p}{4}}+L_d^{p_0})}.
			\end{align}

		\item Then, by triangle inequality, \eqref{enumerate2'} and \eqref{enumerate3'}, it holds for all $s\in[0,T]$, $\tl{s} \in [s,T]$, $t\in[\tl{s},T]$, and $\bar{s} \in [s,\tl{s}] \cap(\dl([0,T]) \cup \{s\})$ with $(\bar{s},\tl{s})\cap \dl([0,T])=\emptyset$ that
		\begin{equation}\label{enumerate4'}
			\begin{aligned}
				&\E\|X_{t}^{s,x,\dl,\ep,\mathcal{M}} - X_{t}^{\tl{s},x,\dl,\ep,\mathcal{M}} \|^p\leq \E\|  X_{t}^{s,x,\dl,\ep,\mathcal{M}} - X_{t}^{\bar{s},x,\dl,\ep,\mathcal{M}}  \|^p +\E\|  X_{t}^{\bar{s},x,\dl,\ep,\mathcal{M}} - X_{t}^{\tl{s},x,\dl,\ep,\mathcal{M}}  \| ^p\\ 
				&\leq  K \big[\big(|\bar{s}-s|+|\bar{s}-s|^p(1+(\ep^{-2}\mathcal{Z}_d)^{p-1}\mathcal{M}^{-\frac{p}{2}})\big) (1+L_d^{\frac{p}{2}})(1+\|x\|^{p_0}+N_d^{p_0})e^{K(|\dl|^{-\frac{p}{4}}+L_d^{p_0})}\\
				&\quad+  K\big(|\tl{s}-\bar{s}|+|\tl{s}-\bar{s}|^p(1+(\ep^{-2}\mathcal{Z}_d)^{p-1}\mathcal{M}^{-\frac{p}{2}})\big) (1+N_d^{\frac{p}{2}}+L_d^{\frac{p}{2}}+(1+L_d^{\frac{p}{2}})\|x\|^{p_0}) e^{K(|\dl|^{-\frac{p}{4}}+L_d^{\frac{p}{2}})}\\
				&\leq  K\big( |\tl{s}-s|+|\tl{s}-s|^p(1+(\ep^{-2}\mathcal{Z}_d)^{p-1}\mathcal{M}^{-\frac{p}{2}}) \big)(1+L_d^{\frac{p}{2}})(1+\|x\|^{p_0}+N_d^{p_0})e^{K(|\dl|^{-\frac{p}{4}}+L_d^{p_0})}.
			\end{aligned}
		\end{equation}
	\end{enumerate}
	Finally, note that  for all $s\in[0,T]$, $\tl{s} \in [s,T]$, and $t\in[\tl{s},T]$, there must exist $\bar{s} \in [s,\tl{s}] \cap(\dl([0,T]) \cup \{s\})$ with $(\bar{s},\tl{s})\cap \dl([0,T])=\emptyset$, it implies that for all $s\in[0,T]$, $\tl{s} \in [s,T]$, and $t\in[\tl{s},T]$, the following inequality holds:
	\begin{align*}
		&\E\|X_{t}^{s,x,\dl,\ep,\mathcal{M}} - X_{t}^{\tl{s},x,\dl,\ep,\mathcal{M}} \|^p\\
		&\leq K\big( |\tl{s}-s|+|\tl{s}-s|^p(1+(\ep^{-2}\mathcal{Z}_d)^{p-1}\mathcal{M}^{-\frac{p}{2}}) \big)(1+L_d^{\frac{p}{2}})(1+\|x\|^{p_0}+N_d^{p_0})e^{K(|\dl|^{-\frac{p}{4}}+L_d^{p_0})}.
	\end{align*}
\end{proof}

\subsection{Proof of results in Section \ref{section5.4}}\label{app:section5.4}
	\begin{proof}[\textbf{Proof of Lemma \ref{differentx}}]
		Fix $s\in[0,T]$, $x,\tl{x}\in\R^d$, $\dl \in \Th$, $\ep \in (0,1)$, and $\mathcal{M} \in \N$ with $\mathcal{M}\geq (1 \vee \ep^{-2}\mathcal{Z}_d)^{2}$ throughout this proof. By the definition  \eqref{Tamedscheme2} of scheme, we obtain, for all $t \in [s,T]$, that
		\begin{align*}
			X_{t}^{s,x,\dl,\ep,\mathcal{M}} - X_t^{s,\tl{x},\dl,\ep,\mathcal{M}} &= x -\tl{x}+ \int_{s}^t \mu^\dl(X_{\max\{s,\dl(r-)\}}^{s,x,\dl,\ep,\mathcal{M}})-\mu^\dl(X^{s,\tl{x},\dl,\ep,\mathcal{M}}_{\max\{s,\dl(r-)\}}) \rd r\\
			&\quad+\int_{s}^t \sg^\dl(X_{\max\{s,\dl(r-)\}}^{s,x,\dl,\ep,\mathcal{M}})-\sg^\dl(X^{s,\tl{x},\dl,\ep,\mathcal{M}}_{\max\{s,\dl(r-)\}}) \rd W_r \\
			&\quad+\int_{s}^t \int_{A_\ep} \gm(X^{s,x,\dl,\ep,\mathcal{M}}_{\max\{s,\dl(r-)\}},z)-\gm(X^{s,\tl{x},\dl,\ep,\mathcal{M}}_{\max\{s,\dl(r-)\}},z) \tl{\pi}(\rd z ,\rd r)\\
			&\quad+\int^{t}_{s} \int_{A_\ep} \gm(X^{s,x,\dl,\ep,\mathcal{M}}_{\max\{s,\dl(r-)\}},z)-\gm(X^{s,\tl{x},\dl,\ep,\mathcal{M}}_{\max\{s,\dl(r-)\}},z)\nu(\rd z)\\
			&\hspace{-23mm}- \lf\{\frac{\nu(A_\ep)}{\mathcal{M}}\sum_{i=1}^{\mathcal{M}}\lf(\gm\lf(X_{\max\{s,\dl(r-)\}}^{s,x,\dl,\ep,\mathcal{M}},V^{\dl,\ep,\mathcal{M}}_{i,\max\{s,\dl(r-)\}}\rt)-\gm\lf(X_{\max\{s,\dl(r-)\}}^{s,\tl{x},\dl,\ep,\mathcal{M}},V^{\dl,\ep,\mathcal{M}}_{i,\max\{s,\dl(r-)\}}\rt)\rt)\rt\} \rd r.
		\end{align*}
		Define a stopping time $\upsilon^{s,x,\tl{x},\dl,\ep,\mathcal{M}}_R$ := $\tau^{s,x,\dl,\ep,\mathcal{M}}_R \wedge \tau^{s,\tl{x},\dl,\ep,\mathcal{M}}_R$ with $R \in \N$ where
		\[ \tau^{s,x,\dl,\ep,\mathcal{M}}_R:= \inf\{t\geq s: \|X^{s,x,\dl,\ep,\mathcal{M}}_t\|\geq R\}\wedge T, \qquad \tau^{s,\tl{x},\dl,\ep,\mathcal{M}}_R:= \inf\{t\geq s: \|X^{s,\tl{x},\dl,\ep,\mathcal{M}}_t\|\geq R\}\wedge T.\] 
		By It\^o's formula, Cauchy-Schwarz inequality, and Lemma \ref{Formula for the remainder}, the following inequality holds for all $t \in [s,T]$:
		\begin{align*}
			&\E\|X_{t \wedge \upsilon^{s,x,\tl{x},\dl,\ep,\mathcal{M}}_R}^{s,x,\dl,\ep,\mathcal{M}} - X_{t \wedge \upsilon^{s,x,\tl{x},\dl,\ep,\mathcal{M}}_R}^{s,\tl{x},\dl,\ep,\mathcal{M}} \|^{p}\leq \|x -\tl{x}\|^p\\
			&\quad+\frac{p}{2}\E \int ^{t \wedge \upsilon^{s,x,\tl{x},\dl,\ep,\mathcal{M}}_R}_{s} \|X_{r}^{s,x,\dl,\ep,\mathcal{M}}-X^{s,\tl{x},\dl,\ep,\mathcal{M}}_{r}\|^{p-2}\big(2(X^{s,x,\dl,\ep,\mathcal{M}}_r-X^{s,\tl{x},\dl,\ep,\mathcal{M}}_{r})\\
			&\hspace{20mm}\times(\mu(X_{r}^{s,x,\dl,\ep,\mathcal{M}})-\mu(X^{s,\tl{x},\dl,\ep,\mathcal{M}}_{r}) )  
			+(p_0-1) \|\sg(X_{r}^{s,x,\dl,\ep,\mathcal{M}} )-\sg(X^{s,\tl{x},\dl,\ep,\mathcal{M}}_{r})\|^2 \big)\rd r\\
			&\quad+ K\E \int ^{t \wedge\upsilon^{s,x,\tl{x},\dl,\ep,\mathcal{M}}_R}_{s} \|X_{r}^{s,x,\dl,\ep,\mathcal{M}}-X^{s,\tl{x},\dl,\ep,\mathcal{M}}_{r}\|^{p-2} (X^{s,x,\dl,\ep,\mathcal{M}}_r-X^{s,\tl{x},\dl,\ep,\mathcal{M}}_{r})\\
			&\hspace{20mm}\times(\mu^\dl(X_{\max\{s,\dl(r)\}}^{s,x,\dl,\ep,\mathcal{M}})-\mu(X_{\max\{s,\dl(r)\}}^{s,x,\dl,\ep,\mathcal{M}}) ) \rd r\\
			&\quad+K\E\int^{t \wedge \upsilon^{s,x,\tl{x},\dl,\ep,\mathcal{M}}_R}_{s}\|X_{r}^{s,x,\dl,\ep,\mathcal{M}}-X^{s,\tl{x},\dl,\ep,\mathcal{M}}_{r}\|^{p-2}  \|\sg^\dl(X_{\max\{s,\dl(r)\}}^{s,x,\dl,\ep,\mathcal{M}})-\sg(X_{\max\{s,\dl(r)\}}^{s,x,\dl,\ep,\mathcal{M}}) \|^2 \rd r \\
			&\quad+ K\E \int^{t \wedge \upsilon^{s,x,\tl{x},\dl,\ep,\mathcal{M}}_R}_{s} \|X_{r}^{s,x,\dl,\ep,\mathcal{M}}-X^{s,\tl{x},\dl,\ep,\mathcal{M}}_{r}\|^{p-2} (X^{s,x,\dl,\ep,\mathcal{M}}_r-X^{s,\tl{x},\dl,\ep,\mathcal{M}}_{r})\\
			&\hspace{20mm}\times(\mu(X_{\max\{s,\dl(r)\}}^{s,x,\dl,\ep,\mathcal{M}})-\mu(X_{r}^{s,x,\dl,\ep,\mathcal{M}}) ) \rd r\\
			&\quad+K\E\int^{t \wedge \upsilon^{s,x,\tl{x},\dl,\ep,\mathcal{M}}_R}_{s}\|X_{r}^{s,x,\dl,\ep,\mathcal{M}}-X^{s,\tl{x},\dl,\ep,\mathcal{M}}_{r}\|^{p-2}  \|\sg(X_{\max\{s,\dl(r)\}}^{s,x,\dl,\ep,\mathcal{M}})-\sg(X_{r}^{s,x,\dl,\ep,\mathcal{M}})\|^2 \rd r \\
			&\quad+ K\E \int^{t \wedge \upsilon^{s,x,\tl{x},\dl,\ep,\mathcal{M}}_R}_{s} \|X_{r}^{s,x,\dl,\ep,\mathcal{M}}-X^{s,\tl{x},\dl,\ep,\mathcal{M}}_{r}\|^{p-2} (X^{s,x,\dl,\ep,\mathcal{M}}_r-X^{s,\tl{x},\dl,\ep,\mathcal{M}}_{r})\\
			&\hspace{20mm}\times(\mu(X_{r}^{s,\tl{x},\dl,\ep,\mathcal{M}})-\mu(X_{\max\{s,\dl(r)\}}^{s,\tl{x},\dl,\ep,\mathcal{M}}) ) \rd r\\
			&\quad+K\E\int^{t \wedge \upsilon^{s,x,\tl{x},\dl,\ep,\mathcal{M}}_R}_{s}\|X_{r}^{s,x,\dl,\ep,\mathcal{M}}-X^{s,\tl{x},\dl,\ep,\mathcal{M}}_{r}\|^{p-2}  \|\sg(X_{r}^{s,\tl{x},\dl,\ep,\mathcal{M}})-\sg(X_{\max\{s,\dl(r)\}}^{s,\tl{x},\dl,\ep,\mathcal{M}})  \|^2 \rd r \\
			&\quad+ K\E \int ^{t \wedge \upsilon^{s,x,\tl{x},\dl,\ep,\mathcal{M}}_R}_{s} \|X_{r}^{s,x,\dl,\ep,\mathcal{M}}-X^{s,\tl{x},\dl,\ep,\mathcal{M}}_{r}\|^{p-2} (X^{s,x,\dl,\ep,\mathcal{M}}_r-X^{s,\tl{x},\dl,\ep,\mathcal{M}}_{r})\\
			&\hspace{20mm}\times(\mu(X_{\max\{s,\dl(r)\}}^{s,\tl{x},\dl,\ep,\mathcal{M}})-\mu^\dl(X_{\max\{s,\dl(r)\}}^{s,\tl{x},\dl,\ep,\mathcal{M}}) ) \rd r\\
			&\quad+K\E\int^{t \wedge \upsilon^{s,x,\tl{x},\dl,\ep,\mathcal{M}}_R}_{s}\|X_{r}^{s,x,\dl,\ep,\mathcal{M}}-X^{s,\tl{x},\dl,\ep,\mathcal{M}}_{r}\|^{p-2}  \|\sg(X_{\max\{s,\dl(r)\}}^{s,\tl{x},\dl,\ep,\mathcal{M}})-\sg^\dl(X_{\max\{s,\dl(r)\}}^{s,\tl{x},\dl,\ep,\mathcal{M}}) \|^2 \rd r\\
			&\quad+K\E\int^{t \wedge \upsilon^{s,x,\tl{x},\dl,\ep,\mathcal{M}}_R}_{s}\int_{A_\ep} \big\{\|X_{r}^{s,x,\dl,\ep,\mathcal{M}}-X^{s,\tl{x},\dl,\ep,\mathcal{M}}_{r}\|^{p-2} \|\gm(X^{s,x,\dl,\ep,\mathcal{M}}_{\max\{s,\dl(r)\}},z)-\gm(X^{s,\tl{x},\dl,\ep,\mathcal{M}}_{\max\{s,\dl(r)\}},z)\|^2\\
			&\hspace{20mm}+\|\gm(X^{s,x,\dl,\ep,\mathcal{M}}_{\max\{s,\dl(r)\}},z)-\gm(X^{s,\tl{x},\dl,\ep,\mathcal{M}}_{\max\{s,\dl(r)\}},z)\|^{p}\big\}\nu(\rd z)\rd r\\
			&\quad+p\E\int^{t \wedge \upsilon^{s,x,\tl{x},\dl,\ep,\mathcal{M}}_R}_{s}\|X_{r}^{s,x,\dl,\ep,\mathcal{M}}-X^{s,\tl{x},\dl,\ep,\mathcal{M}}_{r}\|^{p-2} (X_{r}^{s,x,\dl,\ep,\mathcal{M}}-X^{s,\tl{x},\dl,\ep,\mathcal{M}}_{r})\\
			&\hspace{20mm} \times\bigg(\int_{A_\ep} \gm(X^{s,x,\dl,\ep,\mathcal{M}}_{\max\{s,\dl(r)\}},z)-\gm(X^{s,\tl{x},\dl,\ep,\mathcal{M}}_{\max\{s,\dl(r-)\}},z)\nu(\rd z)\\
			&\quad\quad- \bigg\{\frac{\nu(A_\ep)}{\mathcal{M}}\sum_{i=1}^{\mathcal{M}}\lf(\gm\lf(X_{\max\{s,\dl(r)\}}^{s,x,\dl,\ep,\mathcal{M}},V^{\dl,\ep,\mathcal{M}}_{i,\max\{s,\dl(r)\}}\rt)-\gm\lf(X_{\max\{s,\dl(r)\}}^{s,\tl{x},\dl,\ep,\mathcal{M}},V^{\dl,\ep,\mathcal{M}}_{i,\max\{s,\dl(r)\}}\rt)\rt)\bigg\} \bigg)\rd r.
		\end{align*}
		One then further applies Assumption \ref{A1} and Young's inequality to obtain, for all $t\in[s,T]$, that
		\begin{align*}
			&\E\|X_{t \wedge \upsilon^{s,x,\tl{x},\dl,\ep,\mathcal{M}}_R}^{s,x,\dl,\ep,\mathcal{M}} - X_{t \wedge\upsilon^{s,x,\tl{x},\dl,\ep,\mathcal{M}}_R}^{s,\tl{x},\dl,\ep,\mathcal{M}} \|^{p} \leq \|x -\tl{x}\|^p\\
			&\quad+
			K\E \int^{t \wedge \upsilon^{s,x,\tl{x},\dl,\ep,\mathcal{M}}_R}_{s} \|X_{r}^{s,x,\dl,\ep,\mathcal{M}}-X^{s,\tl{x},\dl,\ep,\mathcal{M}}_{r}\|^{p}\rd r\\
			&\quad+ K\E \int^{t \wedge \upsilon^{s,x,\tl{x},\dl,\ep,\mathcal{M}}_R}_{s} \|\mu^\dl(X_{\max\{s,\dl(r)\}}^{s,x,\dl,\ep,\mathcal{M}})-\mu(X_{\max\{s,\dl(r)\}}^{s,x,\dl,\ep,\mathcal{M}})\|^p+\|\sg^\dl(X_{\max\{s,\dl(r)\}}^{s,x,\dl,\ep,\mathcal{M}})-\sg(X_{\max\{s,\dl(r)\}}^{s,x,\dl,\ep,\mathcal{M}}) \|^p \rd r\\
			&\quad+ K\E \int^{t \wedge \upsilon^{s,x,\tl{x},\dl,\ep,\mathcal{M}}_R}_{s}  \|\mu(X_{\max\{s,\dl(r)\}}^{s,x,\dl,\ep,\mathcal{M}})-\mu(X_{r}^{s,x,\dl,\ep,\mathcal{M}}) \|^p +\|\sg(X_{\max\{s,\dl(r)\}}^{s,x,\dl,\ep,\mathcal{M}})-\sg(X_{r}^{s,x,\dl,\ep,\mathcal{M}})\|^p\rd r\\
			&\quad+ K\E \int^{t \wedge \upsilon^{s,x,\tl{x},\dl,\ep,\mathcal{M}}_R}_{s} \|\mu(X_{r}^{s,\tl{x},\dl,\ep,\mathcal{M}})-\mu(X_{\max\{s,\dl(r)\}}^{s,\tl{x},\dl,\ep,\mathcal{M}}) \|^p+\|\sg(X_{r}^{s,\tl{x},\dl,\ep,\mathcal{M}})-\sg(X_{\max\{s,\dl(r)\}}^{s,\tl{x},\dl,\ep,\mathcal{M}})  \|^p \rd r\\
			&\quad+ K\E \int^{t \wedge \upsilon^{s,x,\tl{x},\dl,\ep,\mathcal{M}}_R}_{s}\|\mu(X_{\max\{s,\dl(r)\}}^{s,\tl{x},\dl,\ep,\mathcal{M}})-\mu^\dl(X_{\max\{s,\dl(r)\}}^{s,\tl{x},\dl,\ep,\mathcal{M}}) \|^p+\|\sg(X_{\max\{s,\dl(r)\}}^{s,\tl{x},\dl,\ep,\mathcal{M}})-\sg^\dl(X_{\max\{s,\dl(r)\}}^{s,\tl{x},\dl,\ep,\mathcal{M}}) \|^p\rd r\\
			&\quad+K\E\int^{t \wedge \upsilon^{s,x,\tl{x},\dl,\ep,\mathcal{M}}_R}_{s}\big(\int_{A_\ep}  \|\gm(X^{s,x,\dl,\ep,\mathcal{M}}_{\max\{s,\dl(r)\}},z)-\gm(X^{s,\tl{x},\dl,\ep,\mathcal{M}}_{\max\{s,\dl(r)\}},z)\|^{2}\nu(\rd z)\big)^{\frac{p}{2}}\rd r\\
			&\quad+K\E\int^{t \wedge \upsilon^{s,x,\tl{x},\dl,\ep,\mathcal{M}}_R}_{s}\int_{A_\ep}  \|\gm(X^{s,x,\dl,\ep,\mathcal{M}}_{\max\{s,\dl(r)\}},z)-\gm(X^{s,\tl{x},\dl,\ep,\mathcal{M}}_{\max\{s,\dl(r)\}},z)\|^{p}\nu(\rd z)\rd r\\
			&\quad+K\E\int^{t \wedge \upsilon^{s,x,\tl{x},\dl,\ep,\mathcal{M}}_R}_{s} \bigg\|\int_{A_\ep} \gm(X^{s,x,\dl,\ep,\mathcal{M}}_{\max\{s,\dl(r)\}},z)-\gm(X^{s,\tl{x},\dl,\ep,\mathcal{M}}_{\max\{s,\dl(r-)\}},z)\nu(\rd z)\\
			&\qquad- \bigg\{\frac{\nu(A_\ep)}{\mathcal{M}}\sum_{i=1}^{\mathcal{M}}\lf(\gm\lf(X_{\max\{s,\dl(r-)\}}^{s,x,\dl,\ep,\mathcal{M}},V^{\dl,\ep,\mathcal{M}}_{i,\max\{s,\dl(r)\}}\rt)-\gm\lf(X_{\max\{s,\dl(r-)\}}^{s,\tl{x},\dl,\ep,\mathcal{M}},V^{\dl,\ep,\mathcal{M}}_{i,\max\{s,\dl(r)\}}\rt)\rt)\bigg\} \bigg\|^p \rd r.
		\end{align*}
		On using Assumptions \ref{A2}, \ref{A4}, Remark \ref{remark3}, and the analogous argument to get \eqref{L3A_4(t)} (together with $\mathcal{M}\ge (1 \vee \ep^{-2}\mathcal{Z}_d)^{2}$), we obtain, for all $t\in[s,T]$, that
		\begin{align*}
			&\E\|X_{t \wedge \upsilon^{s,x,\tl{x},\dl,\ep,\mathcal{M}}_R}^{s,x,\dl,\ep,\mathcal{M}} - X_{t \wedge \upsilon^{s,x,\tl{x},\dl,\ep,\mathcal{M}}_R}^{s,\tl{x},\dl,\ep,\mathcal{M}} \|^{p} \leq \|x -\tl{x}\|^p\\
			&\quad+
			K\E \int ^{t \wedge \upsilon^{s,x,\tl{x},\dl,\ep,\mathcal{M}}_R}_{s} \|X_{r}^{s,x,\dl,\ep,\mathcal{M}}-X^{s,\tl{x},\dl,\ep,\mathcal{M}}_{r}\|^{p} \rd r \\
			&\quad+K(L_d+L_d^{\frac{p}{2}}) \E \int ^{t \wedge \upsilon^{s,x,\tl{x},\dl,\ep,\mathcal{M}}_R}_{s} \|X_{\max\{s,\dl(r)\}}^{s,x,\dl,\ep,\mathcal{M}}-X^{s,\tl{x},\dl,\ep,\mathcal{M}}_{\max\{s,\dl(r)\}}\|^{p} \rd r\\
			&\quad+ K\E \int^{t \wedge \upsilon^{s,x,\tl{x},\dl,\ep,\mathcal{M}}_R}_{s} \|X^{s,x,\dl,\ep,\mathcal{M}}_{r}-X^{s,x,\dl,\ep,\mathcal{M}}_{\max\{s,\dl(r)\}}\|^p(1+\|X^{s,x,\dl,\ep,\mathcal{M}}_{r}\|^{\frac{\chi}{2}}+\|X^{s,x,\dl,\ep,\mathcal{M}}_{\max\{s,\dl(r)\}}\|^{\frac{\chi}{2}})^p \rd r\\
			&\quad+ K\E \int^{t \wedge \upsilon^{s,x,\tl{x},\dl,\ep,\mathcal{M}}_R}_{s} \|X^{s,\tl{x},\dl,\ep,\mathcal{M}}_{r}-X^{s,\tl{x},\dl,\ep,\mathcal{M}}_{\max\{s,\dl(r)\}}\|^p(1+\|X^{s,\tl{x},\dl,\ep,\mathcal{M}}_{r}\|^{\frac{\chi}{2}}+\|X^{s,\tl{x},\dl,\ep,\mathcal{M}}_{\max\{s,\dl(r)\}}\|^{\frac{\chi}{2}})^p \rd r\\
			&\quad+ K|\dl|^{\frac{p}{2}}\E \int^{t \wedge \upsilon^{s,x,\tl{x},\dl,\ep,\mathcal{M}}_R}_{s} 1+\|X^{s,x,\dl,\ep,\mathcal{M}}_{\max\{s,\dl(r)\}}\|^{(\frac{3}{2}\chi+1)p} +\|X^{s,\tl{x},\dl,\ep,\mathcal{M}}_{\max\{s,\dl(r)\}}\|^{(\frac{3}{2}\chi+1)p}  \rd r\\
			&\quad+ K(\ep^{-2}\mathcal{Z}_d)^{p-1}\mathcal{M}^{-\frac{p}{2}}L_d \E \int^{t \wedge \upsilon^{s,x,\tl{x},\dl,\ep,\mathcal{M}}_R}_{s}  \| X^{s,x,\dl,\ep,\mathcal{M}}_{\max\{s,\dl(r)\}}-X_{\max\{s,\dl(r)\}}^{s,\tl{x},\dl,\ep,\mathcal{M}}\|^{p}\rd r.
		\end{align*}
		Then, H\"older's inequality gives that
		\begin{align*}
			&\sup_{t\in[s,T]} \E\|X_{t \wedge \upsilon^{s,x,\tl{x},\dl,\ep,\mathcal{M}}_R}^{s,x,\dl,\ep,\mathcal{M}} - X_{t \wedge \upsilon^{s,x,\tl{x},\dl,\ep,\mathcal{M}}_R}^{s,\tl{x},\dl,\ep,\mathcal{M}} \|^{p} \leq \|x -\tl{x}\|^p\\
			&\quad+ K(1+L_d^{\frac{p}{2}}+(\ep^{-2}\mathcal{Z}_d)^{p-1}\mathcal{M}^{-\frac{p}{2}}L_d)\int ^{T}_{s} \sup_{u\in[{s},r]}\E\|X_{u\wedge \upsilon^{s,x,\tl{x},\dl,\ep,\mathcal{M}}_R}^{s,x,\dl,\ep,\mathcal{M}} - X_{u\wedge \upsilon^{s,x,\tl{x},\dl,\ep,\mathcal{M}}_R}^{s,\tl{x},\dl,\ep,\mathcal{M}} \|^{p} \rd r\\
			&\quad+ K  (T-{s}) (\sup_{t\in[{s},T]}\E\|X^{s,x,\dl,\ep,\mathcal{M}}_{t}-X^{s,x,\dl,\ep,\mathcal{M}}_{\max\{s,\dl(t)\}}\|^{p+\zeta})^{\frac{p}{p+\zeta}}(\sup_{t\in[{s},T]}\E(1+\|X^{s,x,\dl,\ep,\mathcal{M}}_{t}\|^{\frac{\chi}{2}})^{\frac{p(p+\zeta)}{\zeta}})^{\frac{\zeta}{p+\zeta}} \\
			&\quad+ K  (T-{s}) (\sup_{t\in[{s},T]}\E\|X^{s,\tl{x},\dl,\ep,\mathcal{M}}_{t}-X^{s,\tl{x},\dl,\ep,\mathcal{M}}_{\max\{s,\dl(t)\}}\|^{p+\zeta})^{\frac{p}{p+\zeta}}(\sup_{t\in[{s},T]}\E(1+\|X^{s,\tl{x},\dl,\ep,\mathcal{M}}_{t}\|^{\frac{\chi}{2}})^{\frac{p(p+\zeta)}{\zeta}})^{\frac{\zeta}{p+\zeta}} \\
			&\quad+ K|\dl|^{\frac{p}{2}}(T-s) ( 1+\sup_{t\in[{s},T]}\E\|X^{s,x,\dl,\ep,\mathcal{M}}_{t}\|^{(\frac{3}{2}\chi+1)p} +\sup_{t\in[{s},T]}\E\|X^{s,\tl{x},\dl,\ep,\mathcal{M}}_{t}\|^{(\frac{3}{2}\chi+1)p})\rd r.
		\end{align*}
		Since $p\in[2,p^*]$ with $p^*= \min\{\frac{2p_0}{\chi+2}-\zeta, \frac{2p_0}{3\chi+2}, \frac{-\chi\zeta+\sqrt{\chi\zeta(\chi\zeta+8 p_0)}}{2\chi}\}$, the application of Lemmas \ref{scheme3bounds} and \ref{corollary2}\,\ref{lemma9} gives that
		\begin{align*}
			&\sup_{t\in[s,T]}	\E\|X_{t \wedge \upsilon^{s,x,\tl{x},\dl,\ep,\mathcal{M}}_R}^{s,x,\dl,\ep,\mathcal{M}} - X_{t \wedge \upsilon^{s,x,\tl{x},\dl,\ep,\mathcal{M}}_R}^{s,\tl{x},\dl,\ep,\mathcal{M}} \|^{p}\\
			&\leq \|x-\tl{x}\|^p +K|\dl|^{\frac{p}{p+\zeta}} (1+L_d^{\frac{p}{2}})(1+\|x\|^{p_0}+\|\tl{x}\|^{p_0}+N_d^{p_0}) e^{K(1+L_d^{p_0})}\\
			&\quad + K(1+L_d^{\frac{p}{2}}+(\ep^{-2}\mathcal{Z}_d)^{p-1}\mathcal{M}^{-\frac{p}{2}}L_d) \int ^{T}_{s} \ \sup_{u\in[{s},r]} \E\|X_{u \wedge \upsilon^{s,x,\tl{x},\dl,\ep,\mathcal{M}}_R}^{s,x,\dl,\ep,\mathcal{M}}-X^{s,\tl{x},\dl,\ep,\mathcal{M}}_{u \wedge \upsilon^{s,x,\tl{x},\dl,\ep,\mathcal{M}}_R}\|^{p} \rd r.
		\end{align*}
		Finally, Gr\"onwall's lemma yields the desired uniform bound and the application of Fatou's lemma completes the proof.
	\end{proof}

	\begin{proof}[\textbf{Proof of Lemma \ref{differents}}]
	 Fix  $x\in\R^d$, $\dl \in \Th$, $\ep \in (0,1)$, and $\mathcal{M} \in \N$ with $\mathcal{M}\geq (1 \vee \ep^{-2}\mathcal{Z}_d)^{2}$ throughout this proof. By symmetry, it suffices to consider $s\le \tl{s}\le T$ and $t\in[\tl{s},T]$. We then distinguish whether $(s,\tl{s})$ intersects the grid $\dl([0,T])$ following the same case decomposition as in Lemma~\ref{differents'} (see Figure \ref{fig:four-cases-initial-time}). We now proceed with the detailed arguments for the four cases.
		\begin{enumerate}
			\item By definition of scheme \eqref{Tamedscheme2}, H\"older's inequality, \cite[Theorem 7.3]{xuerongmao}, \cite[Lemma 1]{Mikulevicius}, Remarks \ref{remark1},  \ref{remark2}, Assumptions \ref{A3}, \ref{A4}, and \eqref{L3A_4(t)} (together with $\mathcal{M}\ge (1 \vee \ep^{-2}\mathcal{Z}_d)^{2}$), one obtains, for all $s\in[0,T]$, $\tl{s}\in[s,T]$, and $\bar{s} \in[\tl{s},T]$ with $(s,\bar{s})\cap \dl([0,T])=\emptyset$, that
				\begin{align}\label{enumerate1}
					&\E\|X_{\bar{s}}^{s,x,\dl,\ep,\mathcal{M}} - X_{\bar{s}}^{\tl{s},x,\dl,\ep,\mathcal{M}} \|^p\leq K\big\{\|\mu^{\dl}(x)|\tl{s}-s|\|^p +\|\sg^{\dl}(x)(W_{\tl{s}}-W_s)\|^p \nonumber\\
					&\quad+\|\int^{\tl{s}}_s \int_{A_\ep}  \gm(x,z) -\gm(0,z) \tl{\pi}(\rd z,\rd r)\|^p +\|\int^{\tl{s}}_s \int_{A_\ep}  \gm(0,z) \tl{\pi}(\rd z,\rd r)\|^p\nonumber\\
					&\quad+ \|\int^{\tl{s}}_s \int_{A_\ep}  \gm(x,z) \nu(\rd z) - \frac{\nu(A_\ep)}{\mathcal{M}} \sum_{i=1}^{\mathcal{M}} \gm(x, V^{\dl,\ep,\mathcal{M}}_{i,s}) \rd r\|^p\big\}\nonumber\\
					&\leq K\big\{  |\tl{s}-s|^{p}(1+\|x\|^{(\frac{\chi}{2}+1)p}) +  |\tl{s}-s|^{\frac{p}{2}}(1+\|x\|^{(\frac{\chi}{2}+1)p})\nonumber\\
					&\quad+  |\tl{s}-s|^{\frac{p}{2}}(N_d^{\frac{p}{2}}+L_d^{\frac{p}{2}}\|x\|^{p})+  |\tl{s}-s|(N_d+L_d\|x\|^{p})\nonumber \\
					&\quad+  |\tl{s}-s|^p(\ep^{-2}\mathcal{Z}_d)^{p-1}\mathcal{M}^{-\frac{p}{2}} (N_d+L_d\|x\|^p)\big\}\nonumber\\
					&\leq K\big(|\tl{s}-s|+|\tl{s}-s|^p(1+(\ep^{-2}\mathcal{Z}_d)^{p-1}\mathcal{M}^{-\frac{p}{2}})\big) (1+N_d^{\frac{p}{2}}+L_d^{\frac{p}{2}}+(1+L_d^{\frac{p}{2}})\|x\|^{p_0}).
				\end{align}

			\item Note that for all $s\in[0,T]$, $\tl{s}\in[s,T]$, and $t\in[\tl{s},T]$ with $(s,\tl{s})\cap \dl([0,T])=\emptyset$, there exists $\bar{s}\in[\tl{s},t]\cap (\dl([0,T]) \cup \{t\})$ with $(s,\bar{s})\cap \dl([0,T]) = \emptyset$. Moreover, by the fact that Euler approximations \eqref{Tamedscheme2} restricted to their grid points satisfy the Markov property, the disintegration theorem (see, e.g., \cite[Lemma 2.2]{disintegration}), Lemmas \ref{scheme3bounds}, \ref{differentx} and \eqref{enumerate1}, show that, for all $s\in[0,T]$, $\tl{s}\in[s,T]$, $t\in[\tl{s},T]$, and $\bar{s} = \min([\tl{s},t] \cap (\dl([0,T]) \cup \{t\}))$, it holds that
				\begin{align}\label{enumerate2}
					&\E\|X_{t}^{s,x,\dl,\ep,\mathcal{M}} - X_{t}^{\tl{s},x,\dl,\ep,\mathcal{M}} \|^p=\E\lf[  \E\|   X_{t}^{\bar{s},\eta,\dl,\ep,\mathcal{M}} - X_{t}^{\bar{s},\tl{\eta},\dl,\ep,\mathcal{M}} \|^p \bigg\vert_{\eta = X_{\bar{s}}^{s,x,\dl,\ep,\mathcal{M}},\tl{\eta}=X_{\bar{s}}^{\tl{s},x,\dl,\ep,\mathcal{M}}   }  \rt]\nonumber\\
					&\leq K\E\lf[ \big(\|\eta -\tl{\eta}\|^p +|\dl|^{\frac{p}{p+\zeta}} (1+L_d^{\frac{p}{2}})(1+\|\eta\|^{p_0}+\|\tl{\eta}\|^{p_0}+N_d^{p_0}) \big)e^{K(1+L_d^{p_0})}\bigg\vert_{\eta = X_{\bar{s}}^{s,x,\dl,\ep,\mathcal{M}},\tl{\eta}=X_{\bar{s}}^{\tl{s},x,\dl,\ep,\mathcal{M}}   } \rt]\nonumber\\  
					&\leq K\big[\big(|\tl{s}-s|+|\tl{s}-s|^p(1+(\ep^{-2}\mathcal{Z}_d)^{p-1}\mathcal{M}^{-\frac{p}{2}})\big) (1+N_d^{\frac{p}{2}}+L_d^{\frac{p}{2}}+(1+L_d^{\frac{p}{2}})\|x\|^{p_0})\nonumber\\
					& \quad\quad\quad\quad+|\dl|^{\frac{p}{p+\zeta}} (1+L_d^{\frac{p}{2}})(1+\|x\|^{p_0}+N_d^{p_0}) \big)e^{K(1+L_d^{p_0})}\big]\nonumber\\
					&\leq K\big( |\tl{s}-s|+|\tl{s}-s|^p(1+(\ep^{-2}\mathcal{Z}_d)^{p-1}\mathcal{M}^{-\frac{p}{2}})+ |\dl|^{\frac{p}{p+\zeta}}  \big)(1+L_d^{\frac{p}{2}})(1+\|x\|^{p_0}+N_d^{p_0})e^{K(1+L_d^{p_0})}.
				\end{align}
			
			\item Next, we apply the disintegration theorem (see, e.g., \cite[Lemma 2.2]{disintegration}), Lemmas \ref{differenttimet} and \ref{differentx} to show for all $s\in[0,T]$, $\tl{s} \in \dl([0,T])\cap[s,T]$, and $t\in[\tl{s},T]$, that
				\begin{align}\label{enumerate3}
					&\E\|X_{t}^{s,x,\dl,\ep,\mathcal{M}} - X_{t}^{\tl{s},x,\dl,\ep,\mathcal{M}} \|^p=\E\lf[  \E\|   X_{t}^{\tl{s},\eta,\dl,\ep,\mathcal{M}} - X_{t}^{\tl{s},x,\dl,\ep,\mathcal{M}} \|^p \bigg\vert_{\eta = X_{\tl{s}}^{s,x,\dl,\ep,\mathcal{M}}  }  \rt]\nonumber\\
					&\quad \leq K\E\lf[ \big(\|\eta -x\|^p +|\dl|^{\frac{p}{p+\zeta}} (1+L_d^{\frac{p}{2}})(1+\|\eta\|^{p_0}+\|x\|^{p_0}+N_d^{p_0}) \big)e^{K(1+L_d^{p_0})}  \bigg\vert_{\eta = X_{\tl{s}}^{s,x,\dl,\ep,\mathcal{M}}  } \rt]\nonumber\\
					&\leq K \big[\big(|\tl{s}-s|+|\tl{s}-s|^p(1+(\ep^{-2}\mathcal{Z}_d)^{p-1}\mathcal{M}^{-\frac{p}{2}})\big) (1+L_d^{\frac{p}{2}})(1+\|x\|^{p_0}+N_d^{p_0})e^{K(1+L_d^{p_0})}\nonumber\\
					&\quad\quad\quad + |\dl|^{\frac{p}{p+\zeta}}(1+L_d^{\frac{p}{2}})(1+\|x\|^{p_0}+N_d^{p_0})\big]e^{K(1+L_d^{p_0})} \nonumber\\
					&\leq K \big(|\tl{s}-s|+|\tl{s}-s|^p(1+(\ep^{-2}\mathcal{Z}_d)^{p-1}\mathcal{M}^{-\frac{p}{2}})+ |\dl|^{\frac{p}{p+\zeta}}  \big)(1+L_d^{\frac{p}{2}})(1+\|x\|^{p_0}+N_d^{p_0})e^{K(1+L_d^{p_0})}.
				\end{align}
			
			\item Then, by triangle inequality, \eqref{enumerate2} and \eqref{enumerate3}, it holds for all $s\in[0,T]$, $\tl{s} \in [s,T]$, $t\in[\tl{s},T]$, and $\bar{s} \in [s,\tl{s}] \cap(\dl([0,T]) \cup \{s\})$ with $(\bar{s},\tl{s})\cap \dl([0,T])=\emptyset$ that
				\begin{align*}
					&\E\|X_{t}^{s,x,\dl,\ep,\mathcal{M}} - X_{t}^{\tl{s},x,\dl,\ep,\mathcal{M}} \|^p\leq \E\|  X_{t}^{s,x,\dl,\ep,\mathcal{M}} - X_{t}^{\bar{s},x,\dl,\ep,\mathcal{M}}  \|^p +\E\|  X_{t}^{\bar{s},x,\dl,\ep,\mathcal{M}} - X_{t}^{\tl{s},x,\dl,\ep,\mathcal{M}}  \| ^p\\ 
					&\leq  K\big( |\bar{s}-s|+|\bar{s}-s|^p(1+(\ep^{-2}\mathcal{Z}_d)^{p-1}\mathcal{M}^{-\frac{p}{2}})+ |\dl|^{\frac{p}{p+\zeta}}  \big)(1+L_d^{\frac{p}{2}})(1+\|x\|^{p_0}+N_d^{p_0})e^{K(1+L_d^{p_0})} \\
					&\quad+  K\big( |\tl{s}-\bar{s}|+|\tl{s}-\bar{s}|^p(1+(\ep^{-2}\mathcal{Z}_d)^{p-1}\mathcal{M}^{-\frac{p}{2}})+ |\dl|^{\frac{p}{p+\zeta}}  \big)(1+L_d^{\frac{p}{2}})(1+\|x\|^{p_0}+N_d^{p_0})e^{K(1+L_d^{p_0})} \\
					&\leq  K\big( |\tl{s}-s|+|\tl{s}-s|^p(1+(\ep^{-2}\mathcal{Z}_d)^{p-1}\mathcal{M}^{-\frac{p}{2}})+ |\dl|^{\frac{p}{p+\zeta}}  \big)(1+L_d^{\frac{p}{2}})(1+\|x\|^{p_0}+N_d^{p_0})e^{K(1+L_d^{p_0})}.
				\end{align*}
		\end{enumerate}
		Finally, note that  for all $s\in[0,T]$, $\tl{s} \in [s,T]$, and $t\in[\tl{s},T]$, there must exist $\bar{s} \in [s,\tl{s}] \cap(\dl([0,T]) \cup \{s\})$ with $(\bar{s},\tl{s})\cap \dl([0,T])=\emptyset$, it implies that for all $s\in[0,T]$, $\tl{s} \in [s,T]$, and $t\in[\tl{s},T]$, the following inequality holds:
		\begin{align*}
			&\E\|X_{t}^{s,x,\dl,\ep,\mathcal{M}} - X_{t}^{\tl{s},x,\dl,\ep,\mathcal{M}} \|^p\\
			&\leq K\big( |\tl{s}-s|+|\tl{s}-s|^p(1+(\ep^{-2}\mathcal{Z}_d)^{p-1}\mathcal{M}^{-\frac{p}{2}})+ |\dl|^{\frac{p}{p+\zeta}}  \big)(1+L_d^{\frac{p}{2}})(1+\|x\|^{p_0}+N_d^{p_0})e^{K(1+L_d^{p_0})}.
		\end{align*}
	\end{proof}

\section{Proof of auxiliary results}\label{app:D}
\begin{lemma}[Formula for the remainder]\label{Formula for the remainder} Let $p\ge 2$ and $y_1,y_2\in\R^d$. Then there exists a constant $K>0$ only depending on $p$ such that
	\[
\|y_1+y_2\|^{p}-\|y_1\|^{p}-p\|y_1\|^{p-2} y_1^*y_2
	\le
	K\bigl(\|y_1\|^{p-2}\|y_2\|^{2}+\|y_2\|^{p}\bigr).
	\]
	
\begin{proof}
Set $f(y)=\|y\|^{p}$ for $y\in\R^d$. Then, we have the gradient and the Hessian of \(f\) as follows:\footnote{Recall that 0/0:=0 by convention}
	\[
	 \nabla f(y)=p\|y\|^{p-2}y,
	\qquad
     \nabla^2 f(y)=p(p-2)\|y\|^{p-4}yy^*+p\|y\|^{p-2}I_d.
	\]
	Using the second-order Taylor formula with integral remainder (see, e.g., \cite[Section~8.4, Theorem~4]{MR2033094}),
	\[
	f(y_1+y_2)=f(y_1)+\nabla f(y_1)^*y_2
	+\int_0^1(1-q) y_2^* \nabla^2f(y_1+qy_2)y_2 \;\rd q .
	\]
	Hence, we have that
	\begin{align*}
	\|y_1+y_2\|^{p}-\|y_1\|^{p}-p\|y_1\|^{p-2}y_1^*y_2
	&=\int_0^1(1-q) y_2^* \nabla^2 f(y_1+qy_2)y_2 \;\rd q \\
	&\le K \int_0^1(1-q) \|y_1+qy_2\|^{p-2}\|y_2\|^{2}\; \rd q\\
	&\le K\big(\|y_1\|^{p-2}\|y_2\|^{2}+\|y_2\|^{p}\big),
	\end{align*}
for a constant $K>0$ depending only on $p$.
\end{proof}
\end{lemma}


\begin{thebibliography}{10}
	
	\bibitem{MR2512800}
	D.~Applebaum.
	\newblock {\em L\'evy processes and stochastic calculus}, volume 116 of {\em
		Cambridge Studies in Advanced Mathematics}.
	\newblock Cambridge University Press, Cambridge, second edition, 2009.
	
	\bibitem{MR3225549}
	J.~Baldeaux and A.~Badran.
	\newblock Consistent modelling of {VIX} and equity derivatives using a 3/2 plus
	jumps model.
	\newblock {\em Appl. Math. Finance}, 21(4):299--312, 2014.
	
	\bibitem{barndorff2012levy}
	O.~E. Barndorff-Nielsen, T.~Mikosch, and S.~I. Resnick.
	\newblock {\em L{\'e}vy Processes: Theory and Applications}.
	\newblock Springer Science \& Business Media, 2012.
	
	\bibitem{MR2042661}
	R.~Cont and P.~Tankov.
	\newblock {\em Financial modelling with jump processes}.
	\newblock Chapman \& Hall/CRC Financial Mathematics Series. Chapman \&
	Hall/CRC, Boca Raton, FL, 2004.
	
	\bibitem{ProbabilityB}
	C.~Dellacherie and P.-A. Meyer.
	\newblock {\em Probabilities and potential. {B}}, volume~72 of {\em
		North-Holland Mathematics Studies}.
	\newblock North-Holland Publishing Co., Amsterdam, 1982.
	\newblock Theory of martingales, Translated from the French by J. P. Wilson.
	
	\bibitem{MR4338044}
	W.~E, M.~Hutzenthaler, A.~Jentzen, and T.~Kruse.
	\newblock Multilevel {P}icard iterations for solving smooth semilinear
	parabolic heat equations.
	\newblock {\em Partial Differ. Equ. Appl.}, 2(6):Paper No. 80, 31, 2021.
	
	\bibitem{MR3070371}
	J.~Goard and M.~Mazur.
	\newblock Stochastic volatility models and the pricing of {VIX} options.
	\newblock {\em Math. Finance}, 23(3):439--458, 2013.
	
	\bibitem{L.GononC.Schwab}
	L.~Gonon and C.~Schwab.
	\newblock Deep {R}e{LU} neural networks overcome the curse of dimensionality
	for partial integrodifferential equations.
	\newblock {\em Anal. Appl. (Singap.)}, 21(1):1--47, 2023.
	
	\bibitem{I.Gyongy1}
	I.~Gy\"ongy and N.~V. Krylov.
	\newblock On stochastic equations with respect to semimartingales. {I}.
	\newblock {\em Stochastics}, 4(1):1--21, 1980/81.
	
	\bibitem{Itosformula}
	I.~Gy\"ongy and S.~Wu.
	\newblock On {I}t\^o{} formulas for jump processes.
	\newblock {\em Queueing Syst.}, 98(3-4):247--273, 2021.
	
	\bibitem{MR3364862}
	M.~Hutzenthaler and A.~Jentzen.
	\newblock {Numerical approximations of stochastic differential equations with
		non-globally {L}ipschitz continuous coefficients}.
	\newblock {\em Mem. Amer. Math. Soc.}, 236(1112):v+99, 2015.
	
	\bibitem{divergence}
	M.~Hutzenthaler, A.~Jentzen, and P.~E. Kloeden.
	\newblock Strong and weak divergence in finite time of {E}uler's method for
	stochastic differential equations with non-globally {L}ipschitz continuous
	coefficients.
	\newblock {\em Proc. R. Soc. Lond. Ser. A Math. Phys. Eng. Sci.},
	467(2130):1563--1576, 2011.
	
	\bibitem{MR2985171}
	M.~Hutzenthaler, A.~Jentzen, and P.~E. Kloeden.
	\newblock Strong convergence of an explicit numerical method for {SDE}s with
	nonglobally {L}ipschitz continuous coefficients.
	\newblock {\em Ann. Appl. Probab.}, 22(4):1611--1641, 2012.
	
	\bibitem{MR4462404}
	M.~Hutzenthaler, A.~Jentzen, and T.~Kruse.
	\newblock Overcoming the curse of dimensionality in the numerical approximation
	of parabolic partial differential equations with gradient-dependent
	nonlinearities.
	\newblock {\em Found. Comput. Math.}, 22(4):905--966, 2022.
	
	\bibitem{hutzenthaler2020multilevel1}
	M.~Hutzenthaler, A.~Jentzen, T.~Kruse, and T.~A. Nguyen.
	\newblock {Multilevel Picard approximations for high-dimensional semilinear
		second-order PDEs with Lipschitz nonlinearities}.
	\newblock {\em arXiv preprint arXiv:2009.02484}, 2020.
	
	\bibitem{MR4557620}
	M.~Hutzenthaler, A.~Jentzen, T.~Kruse, and T.~A. Nguyen.
	\newblock Overcoming the curse of dimensionality in the numerical approximation
	of backward stochastic differential equations.
	\newblock {\em J. Numer. Math.}, 31(1):1--28, 2023.
	
	\bibitem{disintegration}
	M.~Hutzenthaler, A.~Jentzen, T.~Kruse, T.~A. Nguyen, and P.~von Wurstemberger.
	\newblock Overcoming the curse of dimensionality in the numerical approximation
	of semilinear parabolic partial differential equations.
	\newblock {\em Proc. Roy. Soc. A}, 476(2244):20190630, 25, 2020.
	
	\bibitem{MR4075337}
	M.~Hutzenthaler and T.~Kruse.
	\newblock Multilevel {P}icard approximations of high-dimensional semilinear
	parabolic differential equations with gradient-dependent nonlinearities.
	\newblock {\em SIAM J. Numer. Anal.}, 58(2):929--961, 2020.
	
	\bibitem{MR4423846}
	M.~Hutzenthaler and T.~A. Nguyen.
	\newblock {Strong convergence rate of {E}uler-{M}aruyama approximations in
		temporal-spatial {H}\"older-norms}.
	\newblock {\em J. Comput. Appl. Math.}, 413:Paper No. 114391, 17, 2022.
	
	\bibitem{MR1214374}
	P.~E. Kloeden and E.~Platen.
	\newblock {\em Numerical solution of stochastic differential equations},
	volume~23 of {\em Applications of Mathematics (New York)}.
	\newblock Springer-Verlag, Berlin, 1992.
	
	\bibitem{2017eulerdiffusionlevynoise}
	C.~Kumar and S.~Sabanis.
	\newblock {On explicit approximations for L\'evy driven SDEs with super-linear
		diffusion coefficients}.
	\newblock {\em Electron. J. Probab.}, 22:Paper No. 73, 19, 2017.
	
	\bibitem{MR4032893}
	C.~Kumar and S.~Sabanis.
	\newblock {On {M}ilstein approximations with varying coefficients: the case of
		super-linear diffusion coefficients}.
	\newblock {\em BIT}, 59(4):929--968, 2019.
	
	\bibitem{xuerongmao}
	X.~Mao.
	\newblock {\em Stochastic differential equations and applications}.
	\newblock Horwood Publishing Limited, Chichester, second edition, 2008.
	
	\bibitem{Mikulevicius}
	R.~Mikulevicius and H.~Pragarauskas.
	\newblock On {$L_p$}-estimates of some singular integrals related to jump
	processes.
	\newblock {\em SIAM J. Math. Anal.}, 44(4):2305--2328, 2012.
	
	\bibitem{MR4951649}
	A.~Neufeld and S.~Wu.
	\newblock Multilevel {P}icard approximation algorithm for semilinear partial
	integro-differential equations and its complexity analysis.
	\newblock {\em Stoch. Partial Differ. Equ. Anal. Comput.}, 13(3):1220--1278,
	2025.
	
	\bibitem{2013anoteoneuler}
	S.~Sabanis.
	\newblock A note on tamed {E}uler approximations.
	\newblock {\em Electron. Commun. Probab.}, 18:1--10, 2013.
	
	\bibitem{2016eulerdiffusion}
	S.~Sabanis.
	\newblock {Euler approximations with varying coefficients: the case of
		superlinearly growing diffusion coefficients}.
	\newblock {\em Ann. Appl. Probab.}, 26(4):2083--2105, 2016.
	
	\bibitem{MR2160585}
	R.~Situ.
	\newblock {\em Theory of stochastic differential equations with jumps and
		applications}.
	\newblock Mathematical and Analytical Techniques with Applications to
	Engineering. Springer, New York, 2005.
	
	\bibitem{MR3428878}
	C.~H. Yuen, W.~Zheng, and Y.~K. Kwok.
	\newblock Pricing exotic discrete variance swaps under the 3/2-stochastic
	volatility models.
	\newblock {\em Appl. Math. Finance}, 22(5):421--449, 2015.
	
	\bibitem{MR2033094}
	V.~A. Zorich.
	\newblock {\em Mathematical analysis. {I}}.
	\newblock Universitext. Springer-Verlag, Berlin, {R}ussian edition, 2004.

	
\end{thebibliography}
\end{document}